\documentclass[11 pt]{amsart}
\usepackage{latexsym,amscd,amssymb, graphicx, shuffle, young, float}  
\usepackage[margin=1in]{geometry}

\numberwithin{equation}{section}
\newtheorem{theorem}{Theorem}[section]
\newtheorem{proposition}[theorem]{Proposition}
\newtheorem{corollary}[theorem]{Corollary}
\newtheorem{lemma}[theorem]{Lemma}
\newtheorem{conjecture}[theorem]{Conjecture}
\newtheorem{observation}[theorem]{Observation}

\newtheorem{problem}[theorem]{Problem}
\newtheorem{example}[theorem]{Example}
\newtheorem{remark}[theorem]{Remark}

\newtheorem{defn}[theorem]{Definition}
\theoremstyle{definition}

\newcommand{\Ind}{{\mathrm {Ind}}}
\newcommand{\Res}{{\mathrm {Res}}}
\newcommand{\val}{{\mathrm {val}}}

\newcommand{\Stab}{{\mathrm {Stab}}}

\newcommand{\Cat}{{\mathsf{Cat}}}
\newcommand{\Nar}{{\mathsf{Nar}}}

\newcommand{\symm}{{\mathfrak{S}}}

\newcommand{\CC}{{\mathbb {C}}}
\newcommand{\ZZ}{{\mathbb {Z}}}
\newcommand{\NN}{{\mathbb{N}}}
\newcommand{\QQ}{{\mathbb{Q}}}

\begin{document}

\title[A skein action of the symmetric group on noncrossing partitions]
{A skein action of the symmetric group on noncrossing partitions}

\author{Brendon Rhoades}
\address
{Deptartment of Mathematics \newline \indent
University of California, San Diego \newline \indent
La Jolla, CA, 92093-0112, USA}
\email{bprhoades@math.ucsd.edu}

\begin{abstract}
We introduce and study a new action of the symmetric group $\symm_n$ on the vector space spanned by noncrossing partitions
of $\{1, 2, \dots, n\}$ in which the adjacent transpositions $(i, i+1) \in \symm_n$ act on noncrossing partitions by means of skein relations.
We characterize the isomorphism type of the resulting module
and use it to obtain new representation theoretic proofs of cyclic sieving results due to Reiner-Stanton-White and Pechenik 
for the action of rotation on various classes of noncrossing partitions
and the action of K-promotion on two-row rectangular increasing tableaux.
Our skein relations generalize the Kauffman bracket (or Ptolemy relation) and can be used to 
resolve any set partition as a linear combination of noncrossing partitions
in a
$\symm_n$-equivariant way.
\end{abstract}

\maketitle

\section{Introduction}
\label{Introduction}

The  purpose of this paper is to advertise a new action of the symmetric group
$\symm_n$ on the Catalan-dimensional
 $\QQ$-vector space $V(n)$ with basis given by the set of noncrossing partitions of 
$[n] := \{1, 2, \dots, n\}$.  We highlight several properties of the module to be constructed.
\begin{enumerate}
\item  The action of $\symm_n$ on $V(n)$ will be combinatorial in that the adjacent transpositions $s_i = (i, i+1) \in \symm_n$
will act on noncrossing partitions by means of easily visualized {\it skein relations}.
The simplest of these skein relations is the  {\it Ptolemy relation} or {\it Kauffman bracket} shown below which appears in various places 
in topology, algebra, and combinatorics. 
\begin{center}
\includegraphics[scale = 0.5]{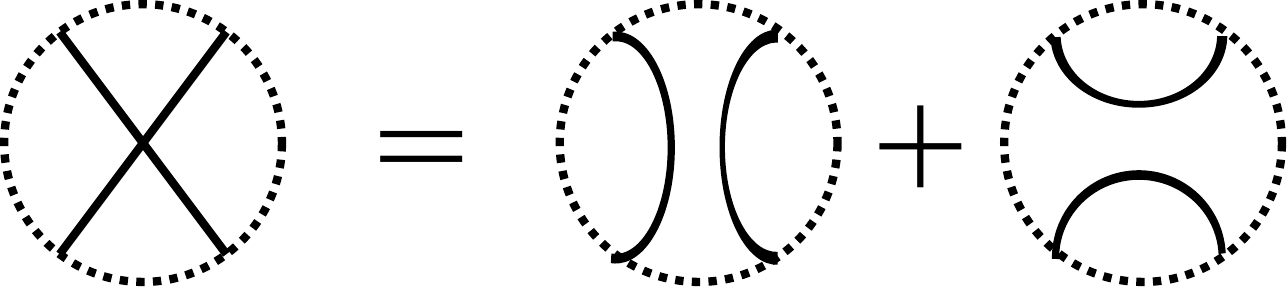}
\end{center}
We leave much of the topological, algebraic, and combinatorial application of our generalized skein
relations as a topic of future study.  To emphasize the role skein relations play in our construction, we call the natural basis of $V(n)$
given by the set of noncrossing partitions the {\it skein basis}.  The representing matrices for permutations in $\symm_n$ with respect
to the skein basis will have entries in $\ZZ$.
\item  Our skein relations will give us a systematic way to ``resolve" {\it any} set partition of $[n]$ as a $\ZZ$-linear combination of noncrossing partitions.
This resolution will be a generalization of the Ptolemy relation shown above; here is an example for $n = 8$.
\begin{center}
\includegraphics[scale = 0.3]{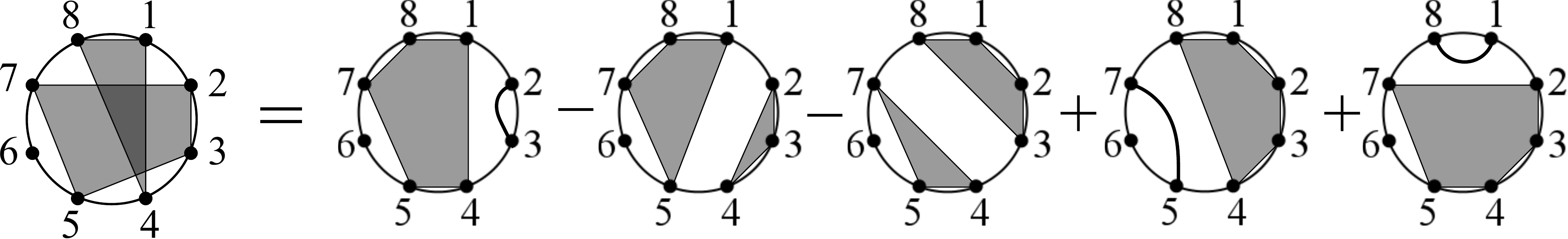}
\end{center}
This resolution is compatible with the natural $\symm_n$-action on all set partitions of $[n]$ in that it affords an $\symm_n$-epimorphism
from (a partially sign-twisted version of) this action to the module $V(n)$.  As far as we know, this resolution of crossings is new at this level of generality.
\item  Our construction will  generalize a natural representation of the Temperley-Lieb algebra $TL_n(x)$  and 
the Kazhdan-Lusztig cellular representation of shape given by a two-row rectangle.  Our skein basis will yield a new ``canonical basis"
of the {\it flag shaped} irreducible representations of $\symm_n$, i.e., those irreducible representations corresponding to partitions
of the form $\lambda = (k, k, 1^{n-2k})$.  When $n = 2k$ and $\lambda$ is a rectangle, the resulting skein basis will coincide with the KL basis, but these 
bases will be different for general flag shapes $\lambda$. 
To the best of our knowledge, the skein basis gives a new basis for flag shaped $\symm_n$-irreducibles.
\item Our skein basis will be an improvement on the KL basis of flag shaped $\symm_n$-irreducibles in that it will give an algebraic proof of certain
cyclic sieving results involving the action of rotation on various classes of
noncrossing partitions.  These cyclic sieving results were proven by Reiner-Stanton-White \cite{RSWCSP}
and Pechenik \cite{Pechenik}
 using brute force enumeration;
the skein basis will give a representation theoretic model of the combinatorial action, allowing us to get a representation theoretic proof.
The KL basis fails at giving such a representation theoretic model.
\end{enumerate}

Item (4) listed above was the original motivation for our skein construction, so let us explain it in more detail.  Let $X$ be a finite set, let 
$C = \langle c \rangle$ be a finite cyclic group acting on $X$, and let $\zeta \in \CC$ be a root-of-unity having multiplicative order $|C|$.  If 
$X(q) \in \NN[q]$ is a polynomial with nonnegative integer coefficients, the triple $(X, C, X(q))$ is said to {\it exhibit the cyclic sieving phenomenon (CSP)}
if for every integer $d \geq 0$, we have that the size $|X^{c^d}|$ of the fixed point set 
$X^{c^d} = \{x \in X \,:\, c^d.x = x\}$ equals the polynomial evaluation $X(\zeta^d) = [X(q)]_{q = \zeta^d}$.  

If the triple $(X, C, X(q))$ exhibits the CSP, we have that 
$X(1) = [X(q)]_{q = 1} = |X|$, so that $X(q)$ is a $q$-analog of an enumerator for the set $X$.
The CSP was defined by Reiner, Stanton, and White \cite{RSWCSP} in 2004 
and generalizes the earlier {\it $q = -1$ phenomenon} of Stembridge \cite{Stembridge}, 
which corresponds 
to the case where $C = \ZZ_2$ has order two.

In the decade since its definition, the CSP has been highly ubiquitous in combinatorics.  If $X$ is a set of natural combinatorial objects and 
$C$ is a natural cyclic group action on $X$, then the triple $(X, C, X(q)$ has an uncanny tendency to exhibit the CSP when $X(q) = \sum_{x \in X} q^{\textsc{stat}(x)}$ 
is the generating function for some natural statistic $\textsc{stat}: X \rightarrow \NN$.  
ŒThe polynomial $X(q)$ typically also has a product formula which is a natural $q$-analog of a product 
enumerator for $X$, as well as an algebraic interpretation as (for example) the Hilbert series of some natural graded representation of degree $|X|$.

There have been two major approaches to proving that a given triple $(X, C, X(q))$ exhibits the CSP: brute force and algebraic.
If one has explicit knowledge of the action of $C = \langle c \rangle$ on $X$, it is sometimes possible to calculate
the fixed point set sizes $|X^{c^d}|$ directly.  Moreover, it may be possible to directly compute the root-of-unity evaluations 
$X(\zeta^d)$ and show that these numbers coincide.  
This `brute force' approach is how CSPs are typically (but not always) first proven.

In the algebraic approach to proving that a given triple $(X, C, X(q))$ exhibits the CSP, one builds a representation
theoretic model for the action of $C = \langle c \rangle$ on $X$.
More precisely, one seeks a group $G$ and a  representation $V$  of $G$ with a distinguished basis $\{e_x \,:\, x \in X\}$ indexed by the elements of $X$.  
One models the action of $C$ on $X$ by finding a group element $g \in G$ such that for all $x \in X$ we have
$g.e_x = e_{c.x}$, so that the representing matrix for $g$ with respect to our distinguished basis is the permutation matrix for the action of $c$ on $X$.
It follows that for any $d \geq 0$ we have 
$|X^{c^d}| = \mathrm{trace}_V(g^d) = \chi(g^d)$, where $\chi: G \rightarrow \CC$ is the character of the $G$-module $V$.  One then uses tools from algebra
to show that the character evaluation $\chi(g^d)$ agrees with the polynomial evaluation $X(\zeta^d)$.
\footnote{In fact, the representation theoretic models used to prove CSPs can be somewhat looser.  Sometimes one can only show that
$g.e_x = \gamma e_{c.x}$ for some $\gamma \in \CC$, so that $\chi(g^d)$ agrees with $|X^{c^d}|$ only up to a constant.  
Even worse, sometimes one can only show that $g.e_x = \gamma(x) e_{c.x}$, where $\gamma: X \rightarrow \CC$ is some function
(hopefully as simple as possible).}

Algebraic proofs of CSPs are generally preferable to brute force proofs.  There are some instances where the action of $C$ on $X$ is sufficiently complicated
such that an algebraic proof is the only proof known.  Even when a brute force proof can be found, algebraic proofs tend to be more elegant and give
representation theoretic intuition for ``why" one should expect a given CSP to hold.  
Conversely, the fact that a triple $(X, C, X(q))$ exhibits the CSP has 
combinatorially foreshadowed various new results in representation theory.

We will use our skein module $V(n)$ to give algebraic proofs
of several CSPs due to Reiner-Stanton-White \cite{RSWCSP}
and Pechenik \cite{Pechenik}. 
 The most basic of these CSPs is for the action of rotation on noncrossing partitions of $[n]$.
Noncrossing partitions of $[n]$ are counted by the {\it Catalan number} $\Cat(n) := \frac{1}{n+1}{2n \choose n}$.  The polynomial appearing in the corresponding 
CSP is a natural $q$-analog of $\Cat(n)$.  We use the standard $q$-analog notation:
\begin{align}
[r]_q &:= 1 + q + q^2 + \cdots + q^{r-1}  = \frac{1-q^r}{1-q}, \\
[m]!_q &:= [m]_q [m-1]_q \cdots [1]_q,  \\
{n \brack k}_q &:= \frac{[n]!_q}{[k]!_q [n-k]!_q}.
\end{align}
The following result can be derived from \cite[Theorem 7.2]{RSWCSP}
 by summing over $k$ and applying the $q$-Chu-Vandermonde convolution;
 for completeness, we perform this calculation at the end of Section~\ref{Applications}.

\begin{theorem}
\label{catalan-csp}  (Reiner-Stanton-White \cite{RSWCSP})
Let $X$ be the set of noncrossing partitions of $[n]$ and let $C = \ZZ_n$ act on $X$ by rotation.  The triple $(X, C, X(q))$ exhibits the 
cyclic sieving phenomenon, where
\begin{equation}
X(q) = \Cat_q(n) := \frac{1}{[n+1]_q}{2n \brack n}_q
\end{equation}
is the {\em $q$-Catalan number}.
\end{theorem}

There are many interpretations of the $q$-Catalan number $X(q)$ appearing in Theorem~\ref{catalan-csp}.
Up to an overall $q$-shift, it is the generating function for the major index statistic on
 standard Young tableaux with shape given by a $2 \times n$ rectangle.  It is also the fake degree polynomial 
$f^{(n,n)}(q)$ for the $(n,n)$-isotypic component of the coinvariant algebra attached to $\symm_{2n}$ and 
the $q$-hook length formula corresponding to the partition $(n, n)$.

Theorem~\ref{catalan-csp} was  proven in \cite{RSWCSP} by brute force, but has since received numerous algebraic proofs.
In \cite{RhoadesJDT} an algebraic proof was given involving the action of the long cycle $(1, 2, \dots, 2n) \in \symm_{2n}$ on the KL cellular representation
of shape $(n, n)$.  A simpler algebraic proof using the representation theory of Temperley-Lieb algebras was given 
by Petersen, Pylyavskyy, and the author
in \cite{PPR}.
Another algebraic proof can be obtained as a weight space refinement of the results in \cite{RhoadesCluster},
where the representations involved come from type A cluster algebras.
Still another algebraic proof was given by Armstrong, Reiner, and the author  using the representation theory of rational Cherednik algebras  \cite{ARR}.
We will give yet another algebraic proof of this result in this paper.

Since the action of rotation preserves the number of blocks in a set partition, it is natural to  refine 
Theorem~\ref{catalan-csp} to get a CSP for the 
action of rotation on noncrossing set partitions of $[n]$ with precisely $k$ blocks.  Such partitions are counted by
the {\it Narayana numbers} $\Nar(n, k) := \frac{1}{n} {n \choose k}{n \choose k-1}$, and a $q$-analog of the Narayana 
number appears in the corresponding CSP.

\begin{theorem}
\label{narayana-csp} (Reiner-Stanton-White \cite[Theorem 7.2]{RSWCSP})
Let $X$ be the set of noncrossing partitions of $[n]$ with $k$ blocks and let $C = \ZZ_n$ act on $X$ by rotation.  The triple 
$(X, C, X(q))$ exhibits the cyclic sieving phenomenon, where
\begin{equation}
X(q) = \Nar_q(n,k) := \frac{1}{[n]_q}{n \brack k}_q {n \brack k-1}_q
\end{equation}
is the {\em $q$-Narayana number}.
\end{theorem}

In fact, a slightly different $q$-analog $X(q) = \frac{q^{k(k+1)}  }{[n]_q}{n \brack k}_q {n \brack k-1}_q$ appears in
\cite[Theorem 7.2]{RSWCSP}, but it can be seen combinatorially or algebraically that these 
polynomials agree when evaluated at all relevant roots of unity.
Up to $q$-shift, the polynomial $X(q)$ in Theorem~\ref{narayana-csp} is the generating function for major index
on Dyck paths of size $n$ having exactly $k$ peaks.
It is also a natural $q$-analog of a product formula due to Cayley \cite{Cayley}
counting polygon dissections with a fixed number of diagonals.
Br\"and\'en \cite{Branden} proved that it is also the principal specialization of 
the Schur function $s_{(2^k)}(x_1, x_2, \dots, x_{n-1})$ which gives the Weyl character of the 
irreducible polynomial representation of $GL_{n-1}(\CC)$ of highest weight $(2^k)$.
As such, it is an instance of the $q$-hook content formula corresponding to the partition $(2^k)$.

The proof of Theorem~\ref{narayana-csp} given in \cite{RSWCSP} is brute force.
In contrast to the Catalan case of  Theorem~\ref{catalan-csp}, no algebraic proof of
the Narayana refinement in Theorem~\ref{narayana-csp} has been found; we will give one in this paper with our skein module construction.
A fundamental obstruction to refining the representation theoretic models used to prove Theorem~\ref{catalan-csp}
to get a proof of Theorem~\ref{narayana-csp} is that these models
 involved {\it irreducible} representations.  In this way, the reducibility of our $\symm_n$-module
$V(n)$ is a feature, not a bug.

Just as the Catalan CSP of Theorem~\ref{catalan-csp} follows from the Narayana CSP of 
Theorem~\ref{narayana-csp},
it will turn out that the Narayana CSP of Theorem~\ref{narayana-csp}
will follow from a still more refined CSP obtained by keeping track not just of the total number of blocks, but also the number
of singleton blocks.  The family of such CSPs can easily be derived from the case of no singleton blocks.  
The following 
CSP was proven by Oliver Pechenik.

\begin{theorem}
\label{no-singletons-csp} (Pechenik \cite[Theorem 1.4]{Pechenik})
Let $X$ be the set of noncrossing partitions of $[n]$ with $k$ blocks and no singleton blocks and let $C = \ZZ_n$ act on $X$ 
by rotation.  The triple $(X, C, X(q))$ exhibits the cyclic sieving phenomenon, where
\begin{equation}
X(q) = \frac{[k]_q}{[n-k]_q[n-k+1]_q} \times \frac{[n]!_q}{([k]!_q)^2[n-2k]!_q}.
\end{equation}
\end{theorem}

In fact, the action involved in the statement of \cite[Theorem 1.4]{Pechenik} is the action of K-promotion on two-row rectangular
increasing tableaux with a fixed maximal entry.
In \cite[Section 2]{Pechenik} an equivariant bijection is given between increasing tableaux under K-promotion
and the action given in Theorem~\ref{no-singletons-csp}, so these results are equivalent.
The action of K-promotion (and the underlying algorithm of K-jeu-de-taquin)
on increasing tableaux gives a K-theoretic analog of the action of 
ordinary promotion and jeu-de-taquin on standard tableaux, so the action
of Theorem~\ref{no-singletons-csp} is of interest in geometry.

The polynomial $X(q)$ appearing in Theorem~\ref{no-singletons-csp} has a number of guises.  It is an instance of the 
$q$-hook length formula corresponding to the flag partition $\lambda = (k, k, 1^{n-2k})$.
As such, it is (up to $q$-shift) the generating function for the major index statistic on standard Young tableaux of 
shape $\lambda$ and the fake degree polynomial for the $\lambda$-isotypic component of the coinvariant algebra for $\symm_n$.

Pechenik proved Theorem~\ref{no-singletons-csp} by brute force and asked for an algebraic proof \cite[Section 5]{Pechenik}.
In the special case $n = 2k$, the set $X$ is the set of noncrossing perfect matchings on the set $[2k]$ and 
a relevant representation theoretic model can be constructed using Kazhdan-Lusztig theory \cite{RhoadesJDT}
or Temperley-Lieb algebras \cite{PPR}.  
Our skein construction will give an algebraic proof of Theorem~\ref{no-singletons-csp} in general. 

In \cite{RhoadesJDT}, the KL cellular basis of $\symm_n$-irreducibles $S^{\lambda}$ of rectangular shape $\lambda$
was used to give an algebraic proof of a CSP for the action of promotion on rectangular standard Young tableaux of 
shape $\lambda$.  
In particular, it was shown that the action of the long cycle $(1, 2, \dots, n) \in \symm_n$ permutes (up to sign) the KL cellular
basis of $S^{\lambda}$ by promotion.
When $\lambda = (k, k)$ is a two-row rectangle, this specializes to an algebraic proof of Theorem~\ref{no-singletons-csp}
in the  case $n = 2k$.
It is natural to guess that when $\lambda = (k, k, 1^{n-2k})$ is a flag shaped partition, the KL cellular basis of 
$S^{\lambda}$ will again give an algebraic proof of Theorem~\ref{no-singletons-csp}.

However, this approach does not work.
Consider the case of $\lambda = (2,2,1,1)$, so that $S^{\lambda}$ is a degree $9$ representation of $\symm_6$.
There is no permutation in $\symm_6$ whose representing 
matrix acting on $S^{\lambda}$ with respect to the KL basis is (a monomial matrix whose nonzero entries give) the permutation matrix for the action of 
rotation on noncrossing partitions of $[6]$ with two blocks and no singleton blocks.

The original motivation for the skein basis of $V(n)$ was to get an algebraic proof of 
Theorems~\ref{narayana-csp} and \ref{no-singletons-csp}.  We have a natural permutation action of the symmetric group
$\symm_n$ on the collection of {\it all} set partitions of $[n]$ (crossing or noncrossing), under which the long cycle $c = (1, 2, \dots, n) \in \symm_n$ acts as rotation.
Since noncrossing partitions carry dihedral symmetry, this action restricts to an action of the dihedral
group $\langle c, w_0 \rangle$ on the set $NC(n)$ of noncrossing partitions of $[n]$, where $w_0$ is the long element in $\symm_n$
whose one-line notation is $w_0 = n(n-1) \dots 1$.  This action preserves the total number of blocks and the number of singleton blocks, so
that the images of  powers $c^d$ of $c$ under the associated dihedral characters give the fixed point set sizes in Theorems~\ref{narayana-csp}
and \ref{no-singletons-csp}.
If one could equate these character evaluations with the appropriate root-of-unity evaluations of $X(q)$, 
Theorems~\ref{narayana-csp} and \ref{no-singletons-csp} would have algebraic proofs.

Unfortunately, the setup of the last paragraph proved insufficient to give algebraic proofs
of Theorems~\ref{narayana-csp} and \ref{no-singletons-csp}.  The  dihedral group $\langle c, w_0 \rangle$  acting on $V(n)$ is ``too small"
to equate the images of powers $c^d$ under its characters with root-of-unity evaluations of the appropriate polynomials.  If this action could 
be extended to the full symmetric group $\symm_n$, then powerful representation theoretic tools (namely Springer's Theorem on regular elements \cite{Springer})
could be used to interpret these character evaluations as root-of-unity evaluations.

However, na\"ive restriction of the $\symm_n$-action on set partitions cannot be used to get an action on $V(n)$ because 
noncrossing partitions of $[n]$ are not closed under the usual permutation action of $\symm_n$ for $n \geq 4$.
From a cyclic sieving point of view, this lack of a $\symm_n$-action is a defect of the noncrossing partitions.  

To remedy this defect,
one might proceed as follows.  For $1 \leq i \leq n-1$ and any noncrossing partition $\pi$ of $[n]$, one could declare $s_i.\pi \in V(n)$ to be some 
linear combination of noncrossing partitions arising from ``resolving"  any crossings formed by swapping $i$ and $i+1$ in $\pi$.  These resolutions
would have to be defined in such a way that
the operators $\{s_i \,:\, 1 \leq i \leq n-1\}$ would satisfy the Coxeter relations.  Moreover, the composition $s_1  \circ s_2 \circ \cdots \circ s_{n-1}$ of
operators given by the action of $c$
would have to act on noncrossing partitions by rotation and one would have to determine the isomorphism type of the $\symm_n$-module 
$V(n)$ to apply Springer's Theorem.  
This all depends on a carefully chosen resolution of crossings $s_i.\pi$, and it is not obvious that such a resolution exists.

This paper is the result of finding such a resolution.  If 
$\pi$ is a noncrossing partition and $1 \leq i \leq n-1$, then
$s_i.\pi \in V(n)$ 
will be a $\ZZ$-linear combination of noncrossing partitions obtained
by a local resolution of any crossings near $i$ in the set partition obtained by
swapping $i$ and $i+1$ in $\pi$.  The skein relations involved in this resolution are a natural generalization
of the Ptolemy relation (which suffices in the special case where $\pi$ is a noncrossing  matching).
The $\symm_n$-module structure afforded on the vector space $V(n)$
will
decompose in a way corresponding to the combinatorial refinements 
shown in Theorems~\ref{catalan-csp}, \ref{narayana-csp}, and \ref{no-singletons-csp}.
We will show that the long cycle $c \in \symm_n$ acts (up to sign) on $V(n)$ by rotation on noncrossing partitions.
More generally, we will show that the action of $\symm_n$ on $V(n)$ preserves {\it local symmetries} in the sense 
that if $\pi$ is a noncrossing partition, $w \in \symm_n$, and the set partition $w(\pi)$ obtained by replacing $i$ with
$w(i)$ in $\pi$ is also noncrossing, then $w$ sends $\pi$ to $w(\pi)$ (up to sign) in the module $V(n)$.

The remainder of the paper is organized as follows.  
In {\bf Section~\ref{Background}} we give background and notation
related to set partitions, noncrossing set partitions, and representations of $\symm_n$.
We also define a slight modification of the standard permutation action of $\symm_n$
on set partitions (the $\star$ action of Lemma~\ref{rhos-satisfy-coxeter}) that will be convenient for our purposes.
In {\bf Section~\ref{Almost}} we introduce the combinatorial notion of an {\it almost noncrossing partition}
and define the four basic skein relations $\sigma_1-\sigma_4$ which are the core of this work.
In {\bf Section~\ref{Symmetric}} we define a collection
$\{\tau_i: V(n) \rightarrow V(n) \,:\, 1 \leq i \leq n-1\}$ of linear endomorphisms of $V(n)$
based on these skein relations and show that they satisfy the Coxeter relations, so that
$s_i \mapsto \tau_i$ gives $V(n)$ the structure of a $\symm_n$-module.
In {\bf Section~\ref{Module}} we determine the isomorphism type of the module $V(n)$.
In {\bf Section~\ref{Long}} we prove that the long cycle $c$ and the long element $w_0$
act on the skein basis of $V(n)$ as plus or minus rotation and reflection, respectively.
In {\bf Section~\ref{Resolving}} we generalize these results on the global symmetries of rotation and reflection to prove that the 
action of $\symm_n$ on $V(n)$ preserves any local symmetry of the noncrossing set partitions (Corollary~\ref{symmetry-corollary}).
This is packaged in terms of a projection of $\symm_n$-modules given in Theorem~\ref{projection-corollary}.
In {\bf Section~\ref{Applications}} we use the module $V(n)$ to get 
representation theoretic proofs of Theorems~\ref{catalan-csp}, \ref{narayana-csp}, and \ref{no-singletons-csp}.
We conclude in 
{\bf Section~\ref{Closing}} by proposing some generalizations of the skein module $V(n)$.

\section{Background}
\label{Background}

Let $\symm_n$ denote the symmetric group on the set $[n]$.  For $1 \leq i \leq n-1$, we let $s_i \in \symm_n$ be the adjacent
transposition $s_i = (i, i+1)$.  The group
$\symm_n$ has a Coxeter presentation with generators $s_1, s_2, \dots, s_{n-1}$ and relations
\begin{equation}
\begin{cases}
s_i^2 = e, & \\
s_i s_j = s_j s_i & |i - j| > 1, \\
s_i s_{i+1} s_i = s_{i+1} s_i s_{i+1}.
\end{cases}
\end{equation}

Let $n, k, s \geq 0$.
We define three collections $\Pi(n), \Pi(n, k),$ and $\Pi(n, k, s)$ of set partitions as follows.
\begin{align}
\Pi(n) &:= \{\text{all set partitions of $[n]$} \}, \\
\Pi(n, k) &:= \{ \text{all set partitions of $[n]$ with $k$ blocks} \}, \\
\Pi(n, k, s) &:= \{ \text{all set partitions of $[n]$ with $k$ blocks and $s$ singletons} \}.
\end{align}
For example, if $\pi$ is the set partition of $[7]$ given by
$\pi = \{ \{1, 3, 4\}, \{2, 7\}, \{5\}, \{6\} \}$, then $\pi \in \Pi(7, 4, 2)$.
We have the disjoint union decompositions $\Pi(n) = \biguplus_k \Pi(n, k)$ and $\Pi(n, k) = \biguplus_s \Pi(n, k, s)$.

We draw set partitions $\pi \in \Pi(n)$ on a disk by labeling the boundary of the disk clockwise with $1, 2, \dots, n$
and connecting cyclically adjacent blockmates in the set partition $\pi$ with arcs.
The pictorial  representations of the set partitions 
$\{ \{1, 6\}, \{2, 5, 4\}, \{3\} \}$ and 
$\{ \{1, 2, 5\}, \{3\}, \{4, 6\} \}$ are shown below.

\begin{center}
\includegraphics[scale = 0.25]{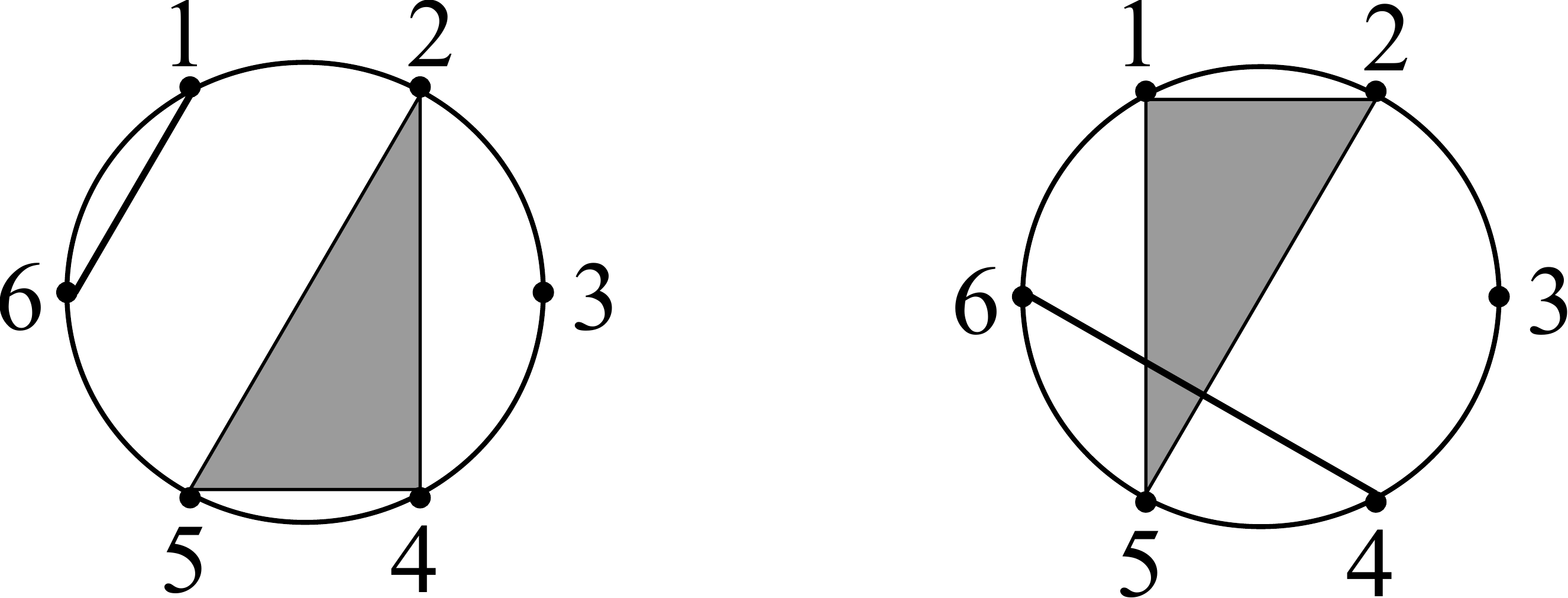}
\end{center}

A set partition $\pi$ is called {\it noncrossing}
if the blocks of $\pi$ do not cross in the interior of the disk.  
For example, the set partition $\{ \{1, 6\}, \{2, 5, 4\}, \{3\} \}$ is noncrossing and the set partition
$\{ \{1, 2, 5\}, \{3\}, \{4, 6\} \}$ is not noncrossing.
Equivalently, the set partition $\pi$ is noncrossing if whenever 
we have indices $1 \leq a < b < c < d \leq n$ such that $a \sim c$ and $b \sim d$ in $\pi$, we necessarily have
$a \sim b \sim c \sim d$ in $\pi$.

Given $n, k, s \geq 0$, we define sets of noncrossing partitions $NC(n), NC(n, k),$ and $NC(n, k, s)$ as follows.
\begin{align}
NC(n) &:= \{\text{all noncrossing partitions of $[n]$} \}, \\
NC(n, k) &:= \{ \text{all noncrossing partitions of $[n]$ with $k$ blocks} \}, \\
NC(n, k, s) &:= \{ \text{all noncrossing partitions of $[n]$ with $k$ blocks and $s$ singletons} \}.
\end{align}

Given $\pi \in \Pi(n)$ and  $1 \leq i \leq n$, we define $B_i(\pi) \subseteq [n]$ to be the block of $\pi$ containing the index $i$.
For example, if $\pi = \{ \{1, 2, 5\}, \{3\}, \{4, 6\} \}$, then
$B_1(\pi) = B_2(\pi) = B_5(\pi) = \{1, 2, 5\}$,
$B_3(\pi) = \{3\}$, and
$B_4(\pi) = B_6(\pi) = \{4, 6\}$.
We also define the {\it valence of $\pi$ at $i$} $\val_i(\pi) \in \{0, 1, 2\}$ to be the number of edges incident to the vertex $i$ when the 
set partition $\pi$ is drawn on a disk.  Equivalently,
\begin{equation}
\val_i(\pi) := \begin{cases}
0 & \text{if $|B_i(\pi)| = 1$,} \\
1 & \text{if $|B_i(\pi)| = 2$,} \\
2 & \text{if $|B_i(\pi)| > 2$.}
\end{cases}
\end{equation}
For example, if $\pi = \{ \{1, 2, 5\}, \{3\}, \{4, 6\} \}$, then
$\val_1(\pi) = \val_2(\pi) = \val_5(\pi) = 2$,
$\val_3(\pi) = 0$,
and $\val_4(\pi) = \val_6(\pi) = 1$.

Given $w \in \symm_n$ and $\pi \in \Pi(n)$, let $w(\pi) \in \Pi(n)$ be the set partition defined by
$w(i) \sim w(j)$ in $w(\pi)$ if and only if $i \sim j$ in $\pi$.
For example, if $\pi = \{ \{1, 2, 5\}, \{3\}, \{4, 6\} \} \in \Pi(6)$ and $w \in \symm_6$ has cycle notation $w = (1, 5, 3)(2,6)$, then
$w(\pi) = \{ \{5, 6, 3\}, \{1\}, \{4, 2\} \}$.
The map $\symm_n \times \Pi(n) \rightarrow \Pi(n)$ given by $(w, \pi) \mapsto w(\pi)$ is a permutation action of 
$\symm_n$ on $\Pi(n)$.
The subsets $\Pi(n, k)$ and $\Pi(n, k, s)$ are stable under this action
of $\symm_n$ for any $k$ and $s$.  We will need a ``partially sign twisted" version of this action; to define this we will have to linearize.

We define three $\QQ$-vector spaces $P(n), P(n, k),$ and $P(n, k, s)$ by taking formal $\QQ$-linear combinations of the corresponding
classes of  set partitions.
That is,
\begin{align}
P(n) &:= \QQ[\Pi(n)], \\
P(n, k) &:= \QQ[\Pi(n, k)], \\
P(n, k, s) &:= \QQ[\Pi(n, k, s)].
\end{align}
Let $V(n), V(n, k),$ and $V(n, k, r)$ be the ``noncrossing" subspaces
\begin{align}
V(n) &:= \QQ[NC(n)], \\
V(n, k) &:= \QQ[NC(n, k)], \\
V(n, k, s) &:= \QQ[NC(n, k, s)].
\end{align}

We define an action of $\symm_n$  on the spaces $P(n), P(n, k),$ and $P(n, k, s)$ as follows. 
For $1 \leq i \leq n-1$, let $\rho_i: P(n) \rightarrow P(n)$ be the linear operator defined on the basis $\Pi(n)$ by
\begin{equation}
\rho_i(\pi) := \begin{cases}
-s_i(\pi) & \text{if $\val_i(\pi), \val_{i+1}(\pi) \geq 1$} \\
s_i(\pi) & \text{else.}
\end{cases}
\end{equation}
In other words, we have that $\rho_i(\pi) = \pm s_i(\pi)$, where the sign is negative unless at least one of the singletons $\{i\}$ or $\{i+1\}$ 
is a block of $\pi$.  
The $\rho_i$ operators assemble to give an action of $\symm_n$ on $P(n)$.

\begin{lemma}
\label{rhos-satisfy-coxeter}
The collection of linear operators $\{\rho_i \,:\, 1 \leq i \leq n-1\}$ satisfy the Coxeter relations.  That is, we have 
$\rho_i^2 = 1$, $\rho_i \rho_j = \rho_j \rho_i$ for $|i - j| > 1$, and $\rho_i \rho_{i+1} \rho_i = \rho_{i+1} \rho_i \rho_{i+1}$.
\end{lemma}

\begin{proof}
Each Coxeter relation can be checked directly on the basis $\Pi(n)$ of $P(n)$ by a case-by-case analysis of valences.
Let $\pi \in \Pi(n)$.

For $1 \leq i \leq n-1$, if $\{i\}$ or $\{i+1\}$ is a block of $\pi$, then we have
$\rho_i^2(\pi) = \rho_i(s_i(\pi)) = \pi$.  Otherwise, we have
$\rho_i^2(\pi) =  -\rho_i(s_i(\pi)) = \pi$.

 Suppose $1 \leq i, j \leq n-1$ and $|i - j| > 1$.  We want to show that $\rho_i \rho_j(\pi) = \rho_j \rho_i(\pi)$.  If none of 
 $\{i\}, \{i+1\}, \{j\},$ or $\{j+1\}$ are blocks of $\pi$, we have that
 $\rho_i \rho_j(\pi) =  s_i s_j(\pi) = s_j s_i(\pi) =  \rho_j \rho_i(\pi)$.  If at least one of 
 $\{i\}$ or $\{i+1\}$ is a block of $\pi$, but neither $\{j\}$ nor $\{j+1\}$ is a block of $\pi$, we have that
 $\rho_i \rho_j(\pi) =  -s_i s_j(\pi) = - s_j s_i(\pi) =  \rho_j \rho_i(\pi)$.  If at least one of 
 $\{i\}$ or $\{i+1\}$ and at least one of $\{j\}$ or $\{j+1\}$ are blocks of $\pi$, we have that 
 $\rho_i \rho_j(\pi) =  s_i s_j(\pi) = s_j s_i(\pi) =  \rho_j \rho_i(\pi)$.
 
 Finally, suppose $1 \leq i \leq n-2$.  We want to show that $\rho_i \rho_{i+1} \rho_i(\pi) = \rho_{i+1} \rho_i \rho_{i+1}(\pi)$.
  If at least two of $\{i\}, \{i+1\},$ and $\{i+2\}$ are blocks of $\pi$, we have that
 $\rho_i \rho_{i+1} \rho_i(\pi) 
 = s_i s_{i+1} s_i(\pi) = s_{i+1} s_i s_{i+1}(\pi) 
 = \rho_{i+1} \rho_i \rho_{i+1}(\pi)$.
 If at most one of $\{i\}, \{i+1\},$ and $\{i+2\}$ are blocks of $\pi$, then
 $\rho_i \rho_{i+1} \rho_i(\pi) = - s_i s_{i+1} s_i(\pi) = - s_{i+1} s_i s_{i+1}(\pi) = \rho_{i+1} \rho_i \rho_{i+1}(\pi)$.
\end{proof}

By Lemma~\ref{rhos-satisfy-coxeter}, the rule $s_i \star \pi := \rho_i(\pi)$ induces a well defined action of
$\symm_n$ on $P(n)$.  For any $w \in \symm_n$ and any $\pi \in \Pi(n)$, we have that
$w \star \pi = \pm w(\pi)$.  We will use this action (as opposed to the standard permutation 
action $(w, \pi) \mapsto w(\pi)$
of $\symm_n$ on $\Pi(n)$) to give
$P(n)$ the structure of a $\symm_n$-module.  
For any $\pi \in \Pi(n)$ and $w \in \symm_n$,
the numbers of total  and singleton blocks in $w(\pi)$ is the same as the corresponding number in $\pi$.
Therefore,  
both $P(n, k, s)$ and $P(n, k)$ are submodules of $P(n)$ for any $n, k,$ and $s$. 
On the other hand, the space $V(n)$ is {\em not} closed under the $\star$ action of $\symm_n$ for $n \geq 4$.
This slightly unusual action of $\symm_n$ on $P(n)$ will prove useful for our
purposes.

For $n \geq 0$, a {\em partition  of $n$} is a weakly decreasing sequence
$\lambda = (\lambda_1, \lambda_2, \dots, \lambda_k)$ of positive integers such that
$\lambda_1 + \lambda_2 + \cdots + \lambda_k = n$.
We write $\lambda \vdash n$ to mean that $\lambda$ is a partition of $n$.
We define a statistic $b$ on partitions by
$b(\lambda) := 0 \lambda_1 + 1 \lambda_2 + 2 \lambda_3 + \cdots$.

The {\em Ferrers diagram} of a partition $\lambda = (\lambda_1, \lambda_2, \dots, \lambda_k)$ of $n$
consists of $\lambda_i$ left justified boxes in row $i$.
The Ferrers diagram of $\lambda = (4, 4, 2, 1) \vdash 11$ is shown below on the left.  
The {\em conjugate} $\lambda'$ of a partition $\lambda$ is the partition obtained by reflecting the Ferrers diagram of $\lambda$
across its main diagonal.  For example, the diagram below shows that $(4, 4, 2, 1)' = (4, 3, 2, 2)$. 

\begin{tiny}
\begin{center}
\begin{Young}
& & &  &, &,  & & & & \cr
& & & &, &, & & &  \cr
& &, &, &,  &, & &  \cr
&, &, &, &, &, & &  \cr
\end{Young}
\end{center}
\end{tiny}

We let $\preceq$ be  {\it dominance order} on partitions of $n$.  If $\lambda, \mu \vdash n$, then
$\lambda \preceq \mu$ if for all $i$ we have 
$\lambda_1 + \lambda_2 + \cdots + \lambda_i \leq \mu_1 + \mu_2 + \cdots + \mu_i$,
where we append an infinite string of zeros to the ends of $\lambda$ and $\mu$ so that these sums make sense.
%Conjugation is an anti-automorphism of dominance order in the sense that we have  
%$\lambda \preceq \mu$ if and only if $\mu' \preceq \lambda'$ for any partitions $\lambda, \mu \vdash n$.

Given $\lambda = (\lambda_1, \lambda_2, \dots, \lambda_k) \vdash n$, we let $\symm_{\lambda} \subseteq \symm_n$
denote the corresponding Young subgroup given by
$\symm_{\lambda} := \symm_{\{1, \dots, \lambda_1\}} \times \symm_{\{\lambda_1 + 1, \dots, \lambda_1 + \lambda_2\}} \times \cdots \times \symm_{\{n-\lambda_k+1, \dots, n\}}$.
For example, we have that 
$\symm_{(4,4,2,1)} = \symm_{\{1,2,3,4\}} \times \symm_{\{5,6,7,8\}} \times \symm_{\{9, 10\}} \times \symm_{\{11\}} \subset \symm_{11}$.

Given any subset $X \subseteq \symm_n$, we define two group algebra elements $[X]_+, [X]_- \in \QQ \symm_n$ by
$[X]_+ := \sum_{w \in X} w$ and $[X]_- := \sum_{w \in X} \mathrm{sign}(w) w$.
When $X = \symm_{\lambda}$ for some $\lambda = (\lambda_1, \dots, \lambda_k) \vdash n$, the group algebra  elements 
$[\symm_{\lambda}]_+$ and $[\symm_{\lambda}]_-$ factor as
$[\symm_{\lambda}]_+ = [\symm_{\{1, \dots, \lambda_1\}}]_+ \cdots [\symm_{\{n-\lambda_k+1, \dots, n\}}]_+$ and
$[\symm_{\lambda}]_- = [\symm_{\{1, \dots, \lambda_1\}}]_- \cdots [\symm_{\{n-\lambda_k+1, \dots, n\}}]_-$.
For example, we have that 
\begin{equation*}
[\symm_{(3,2,2)}]_+ =  [\symm_{\{1,2,3\}}]_+ [\symm_{\{4,5\}}]_+ [\symm_{\{6,7\}}]_+ = (1 + s_1 + s_2 + s_1 s_2 + s_2 s_1 + s_1 s_2 s_1)(1 + s_4)(1 + s_6)
\end{equation*}
 within $\QQ \symm_7$.  Moreover, for any transposition $t = (i, j)$ such that $t \in \symm_{\lambda}$, we have that 
 $[\symm_{\lambda}]_+ = a_+ (1 + t) $ and $[\symm_{\lambda}]_- = a_- (1-t) $ for some group algebra elements
 $a_{\pm} \in \QQ \symm_n$.

The irreducible representations of  $\QQ \symm_n$ are naturally indexed by partitions $\lambda \vdash n$.
We let $S^{\lambda}$ denote the irreducible representation of $\QQ \symm_n$ corresponding to a partition $\lambda$. 
For example, we have that $S^{(n)} = \mathrm{triv}_{\symm_n}$ is the trivial representation and $S^{(1^n)} = \mathrm{sign}_{\symm_n}$ is the sign representation.
The following elementary lemma will be used to determine the isomorphism types of $\symm_n$-modules by checking 
whether they are annihilated by the actions of the group
algebra elements $[\symm_{\lambda}]_+, [\symm_{\lambda}]_-$ for various partitions $\lambda \vdash n$.
While this lemma is standard, we include its proof for completeness.

\begin{lemma}
\label{pincer-lemma}
Let $\lambda, \mu \vdash n$.  We have that $[\symm_{\lambda}]_+$ kills $S^{\mu}$ unless $\lambda \preceq \mu$.
Moreover, we have that $[\symm_{\lambda'}]_-$ kills $S^{\mu}$ unless
$\mu \preceq \lambda$. 
\end{lemma}

\begin{proof}
We prove only the second statement.  The proof of the first statement is similar and omitted.

For any $\mu \vdash n$, let $M^{\mu} := \mathrm{Ind}_{\symm_{\mu}}^{\symm_n}(\mathrm{triv}_{\symm_{\mu}}) \cong \symm_n / \symm_{\mu}$
be the trivial representation induced from $\symm_{\mu}$ to $\symm_n$.  
It is well known that $M^{\mu}$ decomposes into irreducibles as
$M^{\mu} = \bigoplus_{\nu \vdash n} (S^{\nu})^{\oplus K_{\nu, \mu}}$, where
$K_{\nu, \mu}$ is the {\it Kostka number} counting semistandard Young tableaux of shape $\nu$ and content $\mu$.
The Kostka numbers are unitriangular in the sense that 
$K_{\mu, \mu} = 1$ and $K_{\nu, \mu} = 0$ unless $\mu \preceq \nu$.

By the unitriangularity of the Kostka numbers, it is enough to show that 
$[\symm_{\lambda'}]_-$ kills $M^{\mu}$ unless $\mu \preceq \lambda$.
Recall that a {\it row tabloid} $T$ of shape $\mu$ is a  bijective filling of the Ferrers diagram of $\mu$
with the entries $1, 2, \dots, n$, considered up to permutation of entries within rows.  
The module $M^{\mu} = \symm_n / \symm_{\mu}$ has basis given by the set of row tabloids of shape $\mu$, where 
$w \in \symm_n$ acts on a row tabloid $T$ by permuting its entries.

Suppose that $\mu \not\preceq \lambda$.  Let $T$ be a row tabloid of shape $\mu$.  We show that $[\symm_{\lambda'}]_-.T = 0$.
Indeed, because $\mu \not\preceq \lambda$, there exists a transposition $(i, j) \in \symm_n$ such that $i$ and $j$ are in the same 
row of $T$ and $(i, j) \in \symm_{\lambda'}$.  There is a group algebra element $a \in \QQ \symm_n$ such that
$[\symm_{\lambda'}]_- = a(1 - (i, j))$.  Since $(i, j).T = T$, we have that
$a(1-(i,j)).T = a.0 = 0$.
\end{proof}

\section{Almost noncrossing partitions and skein relations}
\label{Almost}

In this section we  describe the fundamental skein moves which allow us to resolve any set partition as a linear
combination of noncrossing partitions.  At the heart of this construction are four rules which define noncrossing resolutions of
``almost, but not quite" noncrossing partitions.

Let $\pi$ be a set partition of $[n]$.  We say that $\pi$ is {\it almost noncrossing} if $\pi$ is not noncrossing, but 
there exists $1 \leq i \leq n-1$ such that the set partition $s_i(\pi)$ is noncrossing.
For example, the set partition $\pi_1 = \{ \{1, 5\}, \{2\}, \{3, 4, 6\} \}$ is almost noncrossing because 
$\pi_1$ is not noncrossing, but $s_5(\pi_1)$ is noncrossing.
The set partition $\pi_2 = \{ \{1, 2, 4\}, \{3\}, \{5, 6\} \}$ is noncrossing, and therefore not almost noncrossing.
The set partition $\pi_3 = \{ \{1, 2, 4, 5\}, \{3, 6\} \}$ is neither noncrossing nor almost noncrossing.

\begin{center}
\includegraphics[scale = 0.5]{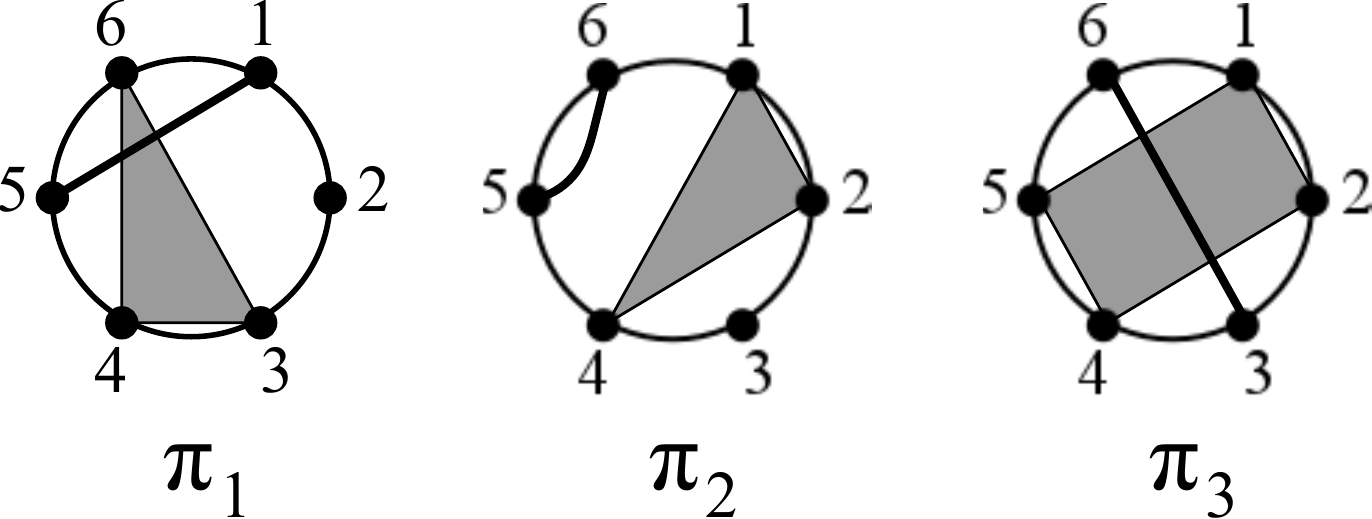}
\end{center}

For $n, k, s \geq 0$, we define sets of almost noncrossing partitions $ANC(n), ANC(n, k),$ and $ANC(n, k, s)$ by
\begin{align}
ANC(n) &:= \{\text{all almost noncrossing partitions of $[n]$} \}, \\
ANC(n, k) &:= \{ \text{all almost noncrossing partitions of $[n]$ with $k$ blocks} \}, \\
ANC(n, k, s) &:= \{ \text{all almost noncrossing partitions of $[n]$ with $k$ blocks and $s$ singletons} \}.
\end{align}
The author is unaware of enumerative formulas for any of the above sets. 

Given $\pi \in ANC(n)$ and $1 \leq i \leq n-1$, we say that $\pi$ {\it crosses at $i$} if $s_i(\pi)$ is noncrossing.  
For example, if $\pi = \{ \{1, 2, 5\}, \{3\}, \{4, 6\} \}$, then
$\pi \in ANC(6,3,1)$ and $\pi$ crosses at both $4$ and $5$. 
If $\pi_1$ is the almost noncrossing partition of $[6]$ shown above, we have that $\pi_1 \in ANC(6, 3, 1)$ and
$\pi_1$ crosses at the single index $5$.
 Also, the partition
$\{ \{1, 3\}, \{2, 4\} \} \in ANC(4,2,0)$ crosses at $1, 2,$ and $3$.  
These examples describe  the entire crossing behavior of almost noncrossing partitions.

\begin{observation}
\label{almost-noncrossing-observation}
Let $\pi \in ANC(n)$.
Then $\pi$ crosses at either one, two, or three indices $i$ for $1 \leq i \leq n-1$. 
If $\pi$ crosses at $i$, then $\val_{\pi}(i), \val_{\pi}(i+1) \in \{1, 2\}$. 
If $\pi$ crosses at two indices $i < j$ for $j \neq i+1$, then $\{i, j\}$ and $\{i+1, j+1\}$ are blocks of $\pi$.
If $\pi$ crosses at $i$ and $i+1$, then $\{i, i+2\}$ is a block of $\pi$.
If $\pi$ crosses at three indices,
then they are of the form $i, i+1,$ and $i+2$, and both $\{i, i+2\}$ and $\{i+1, i+3\}$ are blocks
of $\pi$.
\end{observation}

Let $\pi \in ANC(n)$ be an almost noncrossing partition which crosses at $i$.
We define a linear combination $\sigma(\pi) \in V(n)$ of noncrossing partitions as follows.  The `skein map'
$\sigma: ANC(n) \rightarrow V(n)$ comes in four flavors $\sigma_1 - \sigma_4$ 
depending on the local structure of $\pi$ near the vertex $i$ as shown in Figure~\ref{SkeinRelationsFigure}.  More
formally, we define $\sigma(\pi)$ as follows.

\begin{figure}
\centering
\includegraphics[scale = 0.6]{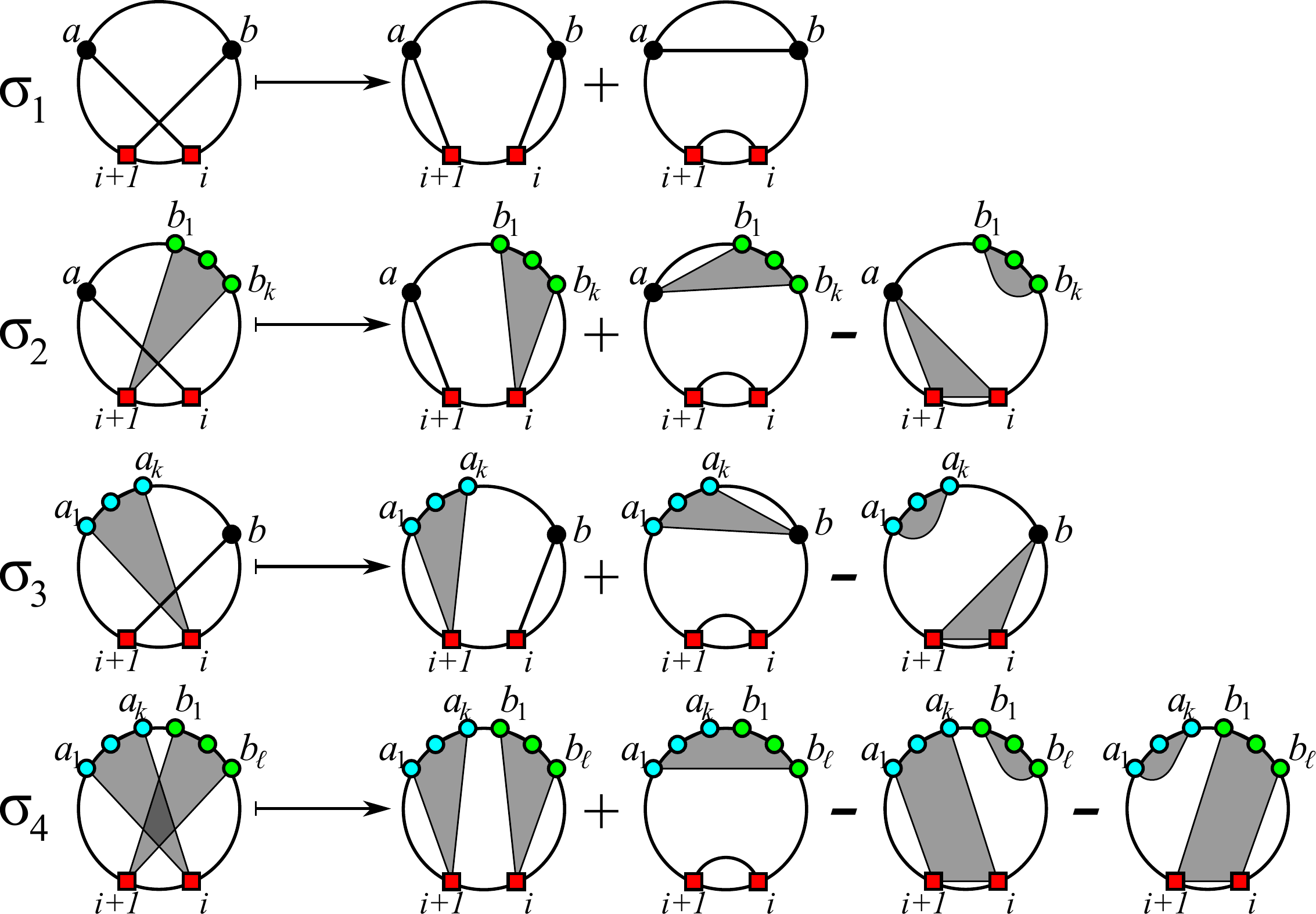}
\caption{The skein relations $\sigma_1$-$\sigma_4$ on almost noncrossing partitions.}
\label{SkeinRelationsFigure}
\end{figure}

\begin{defn}
\label{skein-definition}
Let $\pi \in ANC(n)$ and suppose that $\pi$ crosses at $i$.  We define $\sigma(\pi) \in V(n)$ as follows.
Write $B_i(\pi) = \{i, a_1, \dots, a_k\}$ and $B_{i+1}(\pi) = \{i+1, b_1, \dots, b_{\ell}\}$, so that $k, \ell \geq 1$.
If $k, \ell = 2$, set $\sigma(\pi) = \pi_1 + \pi_2 - \pi_3 - \pi_4$, where $\pi_1 = s_i(\pi)$ 
and $\pi_2, \pi_3,$ and $\pi_4$ are obtained
from $\pi$ by replacing $B_i(\pi)$ and $B_{i+1}(\pi)$ with the blocks$\dots$
\begin{itemize}
\item $\{i, i+1\}$ and $\{a_1, \dots, a_k, b_1, \dots, b_{\ell}\}$ for $\pi_2$,
\item $\{i, i+1, a_1, \dots, a_k\}$ and $\{b_1, \dots, b_{\ell}\}$ for $\pi_3$, and
\item $\{i, i+1, b_1, \dots, b_{\ell}\}$ and $\{a_1, \dots, a_k\}$ for $\pi_4$.
\end{itemize}
If $\ell = 1$ and $k = 2$, then $\sigma(\pi)$ is obtained from the above expression by deleting the term containing
$\pi_3$.  If $k = 1$ and $\ell = 2$, then $\sigma(\pi)$ is obtained from the above expression by deleting the term
containing $\pi_4$.  If $\ell, k = 1$, then $\sigma(\pi)$ is obtained from the above expression by deleting the
terms containing $\pi_3$ and $\pi_4$.
\end{defn}

The two-term skein relation $\sigma_1$ has appeared in numerous contexts.  In combinatorics, it 
represents the famous {\it Ptolemy relation}
which describes the coordinate ring of the Grassmannian 
$\mathrm{Gr}(n, 2)$ of $2$-dimensional subspaces of $\CC^n$, or equivalently gives the exchange relation for type A cluster algebras
(indeed, this is among the primary motivational examples of cluster algebras).
In topology, it is  the classical limit of the {\it Kaufmann bracket} which can be used to define a polynomial knot invariant.
In algebra, it describes an embedding
$\mathrm{TL}_n(-2) \hookrightarrow \QQ \symm_n$ of the Temperley-Lieb algebra at parameter $-2$ into the symmetric group algebra,
where the  generator $t_i$ of $\mathrm{TL}_n(-2)$ is sent to $s_i - 1 \in \QQ[\symm_n]$ for $1 \leq i \leq n-1$.

By contrast, the author is unaware of previous appearances of the three-term skein relations $\sigma_2, \sigma_3$ and 
the four-term skein relation $\sigma_4$.  These skein relations appear to be new.  In this paper, we will use them to 
resolve set partitions into linear combinations of noncrossing partitions,
build representations of the symmetric group, and solve problems related to cyclic sieving.  It could be interesting to 
extend the rich combinatorial, topological, and algebraic manifestations of $\sigma_1$ to the  relations 
$\sigma_2, \sigma_3,$ and $\sigma_4$.

The skein relations $\sigma_1, \sigma_2,$ and $\sigma_3$ can be viewed as ``degenerations" of the skein relation 
$\sigma_4$.  Namely, one applies $\sigma_4$ and removes any terms corresponding to partitions where the number of 
singleton blocks increases.  We will see in the closing remarks of Section~\ref{Symmetric} that 
our action of $\symm_n$ on $V(n)$ could be constructed using the action of $\sigma_4$ alone, so that 
$\sigma_4$ is essentially the only skein relation in this paper.  We consider the degenerate relations
$\sigma_1, \sigma_2,$ and $\sigma_3$ to avoid some annoying  module filtrations
and due to the illustrious history of the relation $\sigma_1$. 

Since an almost noncrossing partition could cross at more than one index, Definition~\ref{skein-definition}
is {\it a priori} ambiguous.  For example, let $\pi = \{ \{1, 2, 5\}, \{3\}, \{4, 6\} \}$, so that $\pi$ crosses at $4$ and $5$.  If we use the crossing
at $4$ to apply Definition~\ref{skein-definition}, we use the flavor $\sigma_2$ of Figure~\ref{SkeinRelationsFigure}
to get that
\begin{equation*}
\sigma(\pi) = \{ \{1, 2, 6\}, \{3\}, \{4, 5\} \} + \{ \{1, 2, 4\}, \{3\}, \{5, 6\} \} - \{ \{1, 2\}, \{3\}, \{4, 5, 6\} \}.
\end{equation*}
On the other hand, if we use the crossing at $5$ to apply Definition~\ref{skein-definition}, we use the flavor
$\sigma_3$ of Figure~\ref{SkeinRelationsFigure} to get the same result for $\sigma(\pi)$.  This is a general phenomenon.

\begin{lemma}
\label{sigma-restricts}
Definition~\ref{skein-definition} gives a well defined function $\sigma: ANC(n) \rightarrow V(n)$.  Moreover, 
if $\pi \in ANC(n)$ every set partition appearing
in $\sigma(\pi)$ has the same number of blocks and the same number of singleton blocks as $\pi$.  
\end{lemma}

\begin{proof}
Let $\pi \in ANC(n)$ and suppose that $\pi$ crosses at $i$.  One verifies that the terms appearing on the right hand
side of $\sigma(\pi)$ presented in Definition~\ref{skein-definition} (where we use the crossing at $i$ to compute
$\sigma(\pi)$)
are all noncrossing set partitions.  This comes from the fact that the set partition $s_i(\pi)$ is noncrossing and
is best seen using the pictorial version of the relations $\sigma_1 - \sigma_4$.  
This
implies that for any fixed crossing used to compute $\sigma(\pi)$, we have that $\sigma(\pi) \in V(n)$.  Moreover,
direct inspection shows that the total number of blocks and the number of singleton blocks appearing in any term of 
$\sigma(\pi)$ is the same as the corresponding count in $\pi$. 

 We are reduced to showing that $\sigma(\pi)$ is independent
of the choice of crossing used to compute in in Definition~\ref{skein-definition}.
To do this, assume first that $\sigma(\pi)$ crosses at precisely two distinct indices $1 \leq i < j \leq n+1$.

If $j \neq i+1$,
then $\{i, j\}$ and $\{j+1, i+1\}$ are both blocks of 
$\pi$ by Observation~\ref{almost-noncrossing-observation}, 
so that we use the relation $\sigma_1$ to compute $\sigma(\pi)$ whether we use the crossing at $i$ or the crossing at $j$.
In either case, we get the same two-term expression for $\sigma(\pi)$.  

If $j = i+1$, at least one of $\val_{\pi}(i)$ 
and $\val_{\pi}(i+1)$ is equal to $1$ by Observation~\ref{almost-noncrossing-observation}.  
If $\val_{\pi}(i) = \val_{\pi}(i+1) = 1$, then we use the relation $\sigma_1$ to compute
$\sigma(\pi)$ whether we use the crossing at $i$ or the crossing at $i+1$.  In either case, we get the same two-term 
expression for $\sigma(\pi)$.  If $\val_{\pi}(i) = 1$ and $\val_{\pi}(i+1) = 2$, then necessarily $\{i, i+2\}$ is a block of $\pi$.
If we use the crossing at $i$ to compute $\sigma(\pi)$, we use the relation $\sigma_2$.  If we use the crossing at $i+1$ to
compute $\sigma(\pi)$, we use the relation $\sigma_3$.  In either case, we get the same three-term expression.

Finally, assume that $\pi$ crosses at three indices.  By Observation~\ref{almost-noncrossing-observation},
they are of the form $i, i+1, i+2$ for some $i$ and $\{i, i+2\}$ and $\{i+1, i+3\}$ are both blocks of $\pi$.
For any choice of crossing, we use the relation $\sigma_1$ to compute $\sigma(\pi)$, resulting in the same two-term expression.
We conclude that $\sigma(\pi)$ is a well defined linear combination of noncrossing partitions.
\end{proof}

By Lemma~\ref{sigma-restricts}, the map $\sigma$ restricts to maps 
$ANC(n, k) \rightarrow V(n, k)$ and $ANC(n, k, s) \rightarrow V(n, k, s)$ for any integers $n, k,$ and $s$.
By abuse of notation, we use $\sigma$ to refer to these restrictions as well.

\section{Symmetric group action on noncrossing partitions}
\label{Symmetric}

In this section we use the skein map $\sigma$ from Section~\ref{Almost}
to endow the space $V(n)$ spanned by noncrossing set partitions of $[n]$ with an action of $\symm_n$.
We begin by describing the linear operators $\tau_i: V(n) \rightarrow V(n)$ for $1 \leq i \leq n-1$
which will give the actions of adjacent transpositions.

Observe that if $\pi \in NC(n)$ and if $1 \leq i \leq n-1$, then either $s_i(\pi)$ is noncrossing or 
$s_i(\pi)$ is an almost noncrossing partition which crosses at $i$.  Define $\tau_i(\pi) \in V(n)$ by
\begin{equation}
\tau_i(\pi) = \begin{cases}
s_i(\pi) & \text{if $s_i(\pi)$ is noncrossing and $B_i(\pi) \neq B_{i+1}(\pi)$}, \\
-\pi & \text{if $B_i(\pi) = B_{i+1}(\pi)$}, \\
\sigma(s_i(\pi)) & \text{if $s_i(\pi)$ is almost noncrossing.}
\end{cases}
\end{equation}
In other words, the operator $\tau_i$ attempts to swap $i$ and $i+1$ in a noncrossing partition $\pi$.
If this results in a crossing partition, then one applies the skein map $\sigma$ to get a linear combination of noncrossing partitions.
If $i$ and $i+1$ are blockmates in $\pi$, then $\tau_i$ negates $\pi$.

Extend $\tau_i$ linearly to a map $\tau_i: V(n) \rightarrow V(n)$ for $1 \leq i \leq n-1$.   
Lemma~\ref{sigma-restricts} implies that $\tau_i$ is a well defined linear endomorphism of $V(n)$ which
stabilizes $V(n, k)$ and $V(n, k, s)$ for any integers $k$ and $s$.  The sets $NC(n), NC(n, k),$ and $NC(n, k, s)$ are 
natural bases of the spaces $V(n), V(n, k),$ and $V(n, k, s)$; we call these the {\it skein bases}.

The goal of this section is to show
that the collection of maps $\{\tau_i \,:\, 1 \leq i \leq n-1\}$ satisfy the Coxeter relations:
\begin{equation}
\begin{cases}
\tau_i^2 = 1 & \text{for all $1 \leq i \leq n-1$,} \\
\tau_i \tau_j = \tau_j \tau_i & \text{for all $|i - j| > 1$,} \\
\tau_i \tau_{i+1} \tau_i = \tau_{i+1} \tau_i \tau_{i+1} & \text{for all $1 \leq i \leq n-2$.}
\end{cases}
\end{equation}
This will imply that the assignment $s_i \mapsto \tau_i$ defines a representation
$\symm_n \rightarrow GL(V(n))$.  

Our proof that the operators $\{\tau_i \,:\, 1 \leq i \leq n-1\}$ satisfy the Coxeter relations
is a somewhat involved verification that the images of any skein basis element
$\pi \in NC(n)$ under the left hand and right hand sides of these relations coincide.  
Fortunately,
the combinatorial definition of the skein map $\sigma$ allows us to take a direct and elementary  approach.
Unfortunately,
several of the resulting elements of $V(n)$ will be 16-term expressions.
It would be
interesting to get a shorter and more conceptual proof that the Coxeter relations are satisfied.

We handle the three Coxeter relations in three separate lemmas.  The
relation $\tau_i^2 = 1$ is the easiest to verify.

\begin{figure}
\centering
\includegraphics[scale = 0.6]{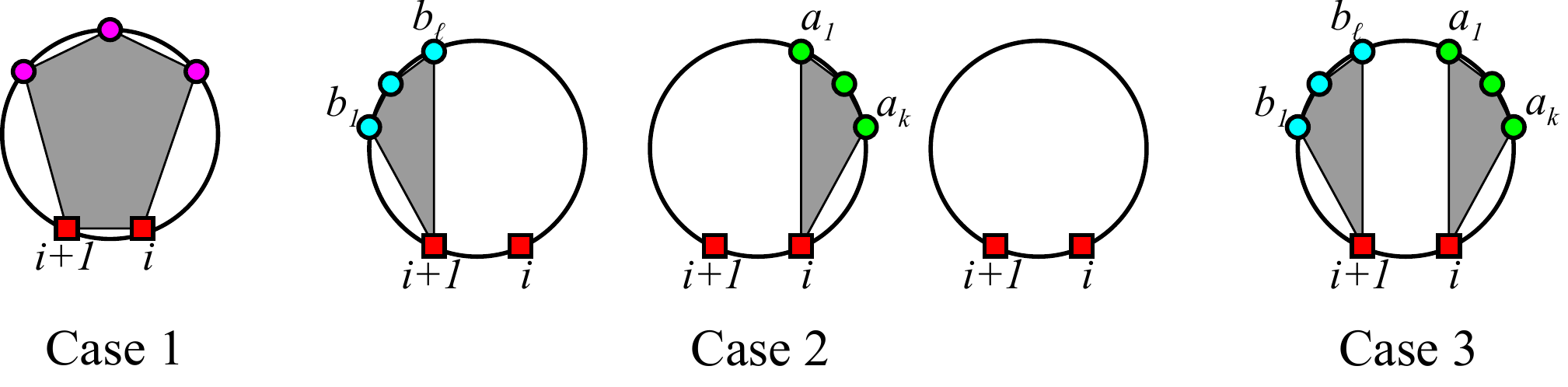}
\caption{The three cases of Lemma~\ref{first-coxeter-relation}.  The numbers of green and blue vertices are both at least one.  
The number of purple vertices could be zero.}
\label{coxeter-one-figure}
\end{figure}

\begin{lemma}
\label{first-coxeter-relation}
For any $1 \leq i \leq n-1$ we have  $\tau_i^2 = 1$.
\end{lemma}

\begin{proof}
Let $\pi \in NC(n)$.  We show that $\tau_i^2(\pi) = \pi$.  Our analysis breaks up into three cases depending
on the local structure of $\pi$ near the index $i$.  These cases are illustrated in Figure~\ref{coxeter-one-figure}.

{\bf Case 1:}  {\it $B_i(\pi) = B_{i+1}(\pi)$.}

This is shown on the left of Figure~\ref{coxeter-one-figure}.
In this case, we have that $\tau_i^2(\pi) = \tau_i(-\pi) = - \tau_i(\pi) = \pi$.

{\bf Case 2:}  {\it One of the valences $\val_{\pi}(i)$ or $\val_{\pi}(i+1)$ is zero and $B_i(\pi) \neq B_{i+1}(\pi)$.}

The three possibilities for this case are shown in the middle of Figure~\ref{coxeter-one-figure}.
In this case, we have that $s_i.\pi$ is noncrossing.  This means that
 $\tau_i^2(\pi) = \tau_i(s_i(\pi)) = s_i^2(\pi) = \pi$.

{\bf Case 3:}  {\it Both of the valences $\val_{\pi}(i)$ or $\val_{\pi}(i+1)$ are positive and $B_i(\pi) \neq B_{i+1}(\pi)$.}

This case is shown on the right of Figure~\ref{coxeter-one-figure}.
In this case, the set partition $s_i(\pi)$ is almost noncrossing.
Let $B_i(\pi) = \{i, a_1, \dots, a_k\}$ and $B_{i+1}(\pi) = \{i+1, b_1, \dots, b_{\ell}\}$, where $k, \ell \geq 1$.  
If $k, \ell > 1$,
we have that 
$\tau_i(\pi) = \pi_1 + \pi_2 - \pi_3- \pi_4$, where $\pi_1 = \pi$ and $\pi_2, \pi_3, \pi_4$ are obtained
from $\pi$ by replacing $B_i(\pi)$ and $B_{i+1}(\pi)$ with the blocks$\dots$
\begin{itemize}
\item 
$\{i, i+1\}$ and $\{a_1, \dots, a_k, b_1, \dots, b_{\ell} \}$ for $\pi_2$,
\item
$\{i, i+1, a_1, \dots, a_k\}$ and $\{b_1, \dots, b_{\ell} \}$ for $\pi_3$, 
and
\item
$\{i, i+1, b_1, \dots, b_{\ell} \}$ and $\{a_1, \dots, a_k\}$ for $\pi_4$.
\end{itemize}
If $k = 1$, then $\tau_i(\pi)$ is obtained from the above expression by removing the term involving $\pi_4$.
If $\ell = 1$, then $\tau_i(\pi)$ is obtained from the above expression by removing the term involving
$\pi_3$.
We have that
$\tau_i(\pi_2) = -\pi_2, \tau_i(\pi_3) = -\pi_3,$ and $\tau_i(\pi_4) = -\pi_4$.  This implies that
$\tau_i^2(\pi) = \pi$, as desired.  This concludes the proof of Lemma~\ref{first-coxeter-relation}.
\end{proof}

The commuting relation $\tau_i \tau_j = \tau_j \tau_i$ for $|i - j| > 1$ involves more cases in its verification.

\begin{lemma}
\label{second-coxeter-relation}
For any $1 \leq i < j \leq n-1$ with $j - i > 1$ we have $\tau_i \tau_j = \tau_j \tau_i$.
\end{lemma}

\begin{figure}
\centering
\includegraphics[scale = 0.6]{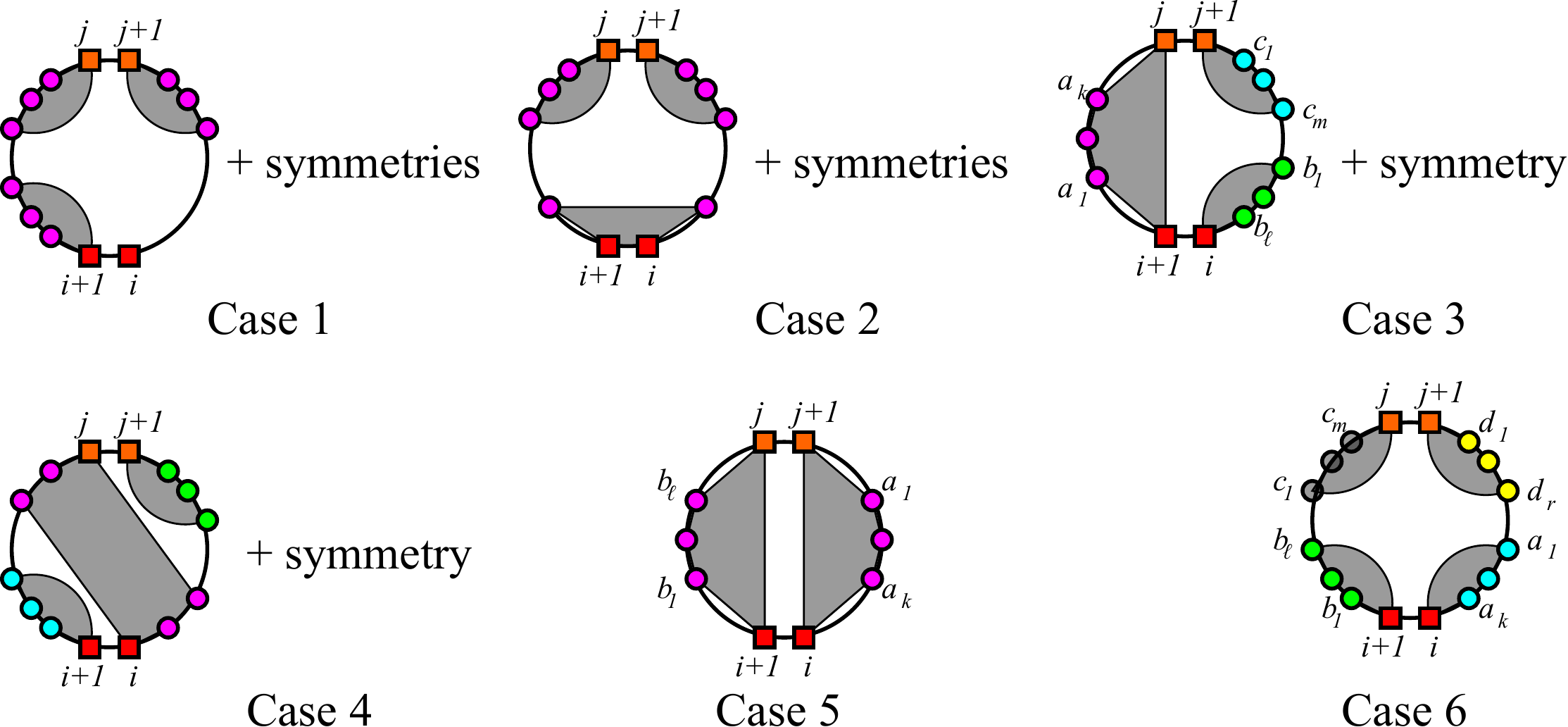}
\caption{The six cases of Lemma~\ref{second-coxeter-relation}.  
The numbers of green, blue, yellow, and gray vertices are all at least one.
The number of purple vertices could be zero.}
\label{coxeter-two-figure}
\end{figure}

\begin{proof}
Let $\pi \in NC(n)$.  We show that $\tau_i \tau_j (\pi) = \tau_j \tau_i (\pi)$.  
Our analysis breaks up into six cases depending on the structure of $\pi$ near the indices $i$ and $j$.
These cases are illustrated in Figure~\ref{coxeter-two-figure}.

{\bf Case 1:} {\it At least one of the indices $i, i+1, j,$ or $j+1$ has valence $0$ in $\pi$.}

The situation when $\val_{\pi}(i) = 0$ and $B_i(\pi), B_j(\pi),$ and $B_{j+1}(\pi)$ are distinct
is shown in the upper left of Figure~\ref{coxeter-two-figure}.  

We suppose that  $i$ has valence 0 in $\pi$ (the argument when one of $i+1, j,$ or $j+1$ has valence $0$ in $\pi$ is similar).
Since $|i - j| > 1$,  in every skein basis element $\pi'$ appearing in $\tau_j(\pi) \in V(n)$ we
have $\val_{\pi'}(i) = 0$, so that $s_i(\pi')$ is noncrossing and $B_i(\pi') \neq B_{i+1}(\pi')$.
 This means that $\tau_i \tau_j(\pi) = s_i(\tau_j(\pi))$.  On the other hand,
we also have that the set partition $s_i(\pi)$ is noncrossing and $B_i(\pi) \neq B_{i+1}(\pi)$.
 This means 
 that $s_i(\tau_j(\pi)) = \tau_j(s_i(\pi)) = \tau_j \tau_i(\pi)$.  
Combining, we get that $\tau_i \tau_j (\pi) = \tau_j \tau_i (\pi)$.  

{\bf Case 2:}  {\it $B_i(\pi) = B_{i+1}(\pi)$ or $B_j(\pi) = B_{j+1}(\pi)$.}

The situation when $B_i(\pi) = B_{i+1}(\pi)$ and $B_j(\pi) \neq B_{j+1}(\pi)$ is shown in the upper center of Figure~\ref{coxeter-two-figure}.

If $B_i(\pi) = B_{i+1}(\pi)$, then (because $|i - j| > 1$) we have that $B_i(\pi') = B_{i+1}(\pi')$ for every skein basis element $\pi'$
appearing in
$\tau_j(\pi)$.  This means that $\tau_i \tau_j(\pi) = - \tau_j(\pi) = \tau_j(-\pi) = \tau_j \tau_i(\pi)$. 
The case $B_j(\pi) = B_{j+1}(\pi)$ is similar.

{\bf Case 3:}  {\it The valences $\val_{\pi}(i), \val_{\pi}(i+1), \val_{\pi}(j),$ and $\val_{\pi}(j+1)$ are all positive.
$B_i(\pi) \neq B_{i+1}(\pi)$ and $B_j(\pi) \neq B_{j+1}(\pi)$, but either $B_i(\pi) = B_{j+1}(\pi)$ or 
$B_{i+1}(\pi) = B_j(\pi)$ (but not both).}

The situation when $B_{i+1}(\pi) = B_j(\pi)$ is shown in the upper right of Figure~\ref{coxeter-two-figure}.

We treat the possibility $B_{i+1}(\pi) = B_j(\pi)$, which necessarily implies that $B_i(\pi) \neq B_{j+1}(\pi)$.  
The argument for the other possibilities is similar and left for the reader.
Let $B_i(\pi) = \{i, b_1, \dots b_{\ell}\}, B_{i+1}(\pi) = \{i+1, j, a_1, \dots, a_k\},$ and 
$B_{j+1}(\pi) = \{j+1, c_1, \dots, c_m\}$.
We have that $k \geq 0, \ell \geq 1,$ and $m \geq 1$.
If $k , \ell, m > 1$, we have that the common value of 
$\tau_i \tau_j(\pi)$ and $\tau_j \tau_i(\pi)$ is
$\pi_1 + \cdots + \pi_8 - \pi_9 - \cdots - \pi_{16}$, where $\pi_1 = \pi$ and $\pi_2, \dots, \pi_{16}$ are obtained
from $\pi$ by replacing $B_i(\pi), B_{i+1}(\pi),$ and $B_{j+1}(\pi)$ with the blocks$\dots$
\begin{itemize}
\item $\{i, b_1, \dots, b_{\ell}\}, \{i+1, a_1, \dots, a_k, c_1, \dots, c_m\},$ and $\{j, j+1\}$ for $\pi_2$,
\item $\{i, i+1\}, \{j, a_1, \dots, a_k, b_1, \dots, b_{\ell}\},$ and $\{j+1, c_1, \dots, c_m\}$ for $\pi_3$,
\item $\{i, i+1\}, \{j, j+1\},$ and $\{a_1, \dots, a_k, b_1, \dots, b_{\ell}, c_1, \dots, c_m\}$ for $\pi_4$,
\item $\{i, i+1, b_1, \dots, b_{\ell}\}, \{j, j+1, a_1, \dots, a_k\},$ and $\{c_1, \dots, c_m\}$ for $\pi_5$,
\item $\{i, i+1, b_1, \dots, b_{\ell}\}, \{j, j+1, c_1, \dots, c_m\},$ and $\{a_1, \dots, a_k\}$ for $\pi_6$,
\item $\{i, i+1, j, j+1, a_1, \dots, a_k\}, \{b_1, \dots, b_{\ell}\},$ and $\{c_1, \dots, c_m\}$ for $\pi_7$,
\item $\{i, i+1, a_1, \dots, a_k\}, \{j, j+1, c_1, \dots, c_m\},$ and $\{b_1, \dots, b_{\ell}\}$ for $\pi_8$,
\item $\{i, b_1, \dots, b_{\ell}\}, \{i+1, j, j+1, a_1, \dots, a_k\},$ and $\{c_1, \dots, c_m\}$ for $\pi_9$,
\item $\{i, b_1, \dots, b_{\ell}\}, \{i+1, a_1, \dots, a_k\},$ and $\{j, j+1, c_1, \dots, c_m\}$ for $\pi_{10}$,
\item $\{i, i+1\}, \{j, j+1, a_1, \dots, a_k, b_1, \dots, b_{\ell}\},$ and $\{c_1, \dots, c_m\}$ for $\pi_{11}$,
\item $\{i, i+1\}, \{j, j+1, c_1, \dots, c_m\},$ and $\{a_1, \dots, a_k, b_1, \dots, b_{\ell}\}$ for $\pi_{12}$,
\item $\{i, i+1, j, a_1, \dots, a_k\}, \{j+1, c_1, \dots, c_m\},$ and $\{b_1, \dots, b_{\ell}\}$ for $\pi_{13}$,
\item $\{i, i+1, a_1, \dots, a_k, c_1, \dots, c_m\}, \{j, j+1\},$ and $\{b_1, \dots, b_{\ell}\}$ for $\pi_{14}$,
\item $\{i, i+1, b_1, \dots, b_{\ell}\}, \{j, a_1, \dots, a_k\},$ and $\{j+1, c_1, \dots, c_m\}$ for $\pi_{15}$, and
\item $\{i, i+1, b_1, \dots, b_{\ell}\}, \{j, j+1\},$ and $\{a_1, \dots, a_k, c_1, \dots, c_m\}$ for $\pi_{16}$.
\end{itemize}
In the ``degenerate" cases where $k, \ell, m > 1$ is not satisfied, we modify the above expression as follows
to obtain the common value of $\tau_i \tau_j(\pi)$ and $\tau_j \tau_i(\pi)$.
If $k = 1$, we remove the term involving $\pi_6$.  
If $k = 0$, we remove the terms involving $\pi_6, \pi_{10},$ and $\pi_{15}$.
If $\ell = 1$, we remove the terms involving $\pi_7, \pi_8, \pi_{13},$ and $\pi_{14}$.
If $m = 1$, we remove the terms involving $\pi_5, \pi_7, \pi_9,$ and $\pi_{11}$.
If $k = 0$ and $\ell = 1$, we remove the terms involving $\pi_6, \pi_7, \pi_8, \pi_{10}, \pi_{12}, \pi_{13}, 
\pi_{14},$ and $\pi_{15}$.
If $k = 0$ and $m = 1$, we remove the terms involving
$\pi_5, \pi_6, \pi_9, \pi_{10}, \pi_{11}, \pi_{15},$ and $\pi_{16}$.

{\bf Case 4:}  {\it The valences $\val_{\pi}(i), \val_{\pi}(i+1), \val_{\pi}(j),$ and $\val_{\pi}(j+1)$ are all positive.
$B_i(\pi) \neq B_{i+1}(\pi)$ and $B_j(\pi) \neq B_{j+1}(\pi)$, but $B_i(\pi) = B_j(\pi)$ or $B_{i+1}(\pi) = B_{j+1}(\pi)$.}

This is illustrated in the lower left of Figure~\ref{coxeter-two-figure}.
This is similar to Case 3.  The common value of $\tau_i \tau_j(\pi)$ and $\tau_j \tau_i(\pi)$
is a $16$-term expression non-degenerate cases and shortens in a similar way in degenerate
cases.  We leave the details to the reader.

{\bf Case 5:}  {\it The valences $\val_{\pi}(i), \val_{\pi}(i+1), \val_{\pi}(j),$ and $\val_{\pi}(j+1)$ are all positive.
$B_i(\pi) \neq B_{i+1}(\pi)$ and $B_j(\pi) \neq B_{j+1}(\pi)$, but $B_i(\pi) = B_{j+1}(\pi)$ and 
$B_{i+1}(\pi) = B_j(\pi).$}

This is illustrated in the lower middle of Figure~\ref{coxeter-two-figure}.

Write $B_i(\pi) = B_{j+1}(\pi) = \{i, j+1, a_1, \dots, a_k\}$
and $B_{i+1}(\pi) = B_j(\pi) = \{i+1, j, b_1, \dots, b_{\ell}\}$.  We have $k, \ell \geq 0$.  If $k, \ell > 1$,
the common value of  $\tau_i \tau_j(\pi)$ and $\tau_j \tau_i(\pi)$ is 
$\pi_1 + \cdots + \pi_5 - \pi_6 - \cdots - \pi_{11}$, where
$\pi_1 = \pi$ and $\pi_2, \dots, \pi_{11}$ are obtained from $\pi$ by replacing $B_i(\pi)$ and $B_j(\pi)$ with the 
blocks$\dots$
\begin{itemize}
\item $\{i, i+1, j, j+1, b_1, \dots, b_{\ell}\}$ and $\{a_1, \dots, a_k\}$ for $\pi_2$,
\item $\{i, i+1, b_1, \dots, b_{\ell}\}$ and $\{j, j+1, a_1, \dots, a_k\}$ for $\pi_3$,
\item $\{i, i+1, a_1, \dots, a_k\}$ and $\{j, j+1, b_1, \dots, b_{\ell}\}$ for $\pi_4$,
\item $\{i, i+1, j, j+1, a_1, \dots, a_k\}$ and $\{b_1, \dots, b_{\ell}\}$ for $\pi_5$,
\item $\{i, a_1, \dots, a_k\}$ and $\{i+1, j, j+1, b_1, \dots, b_{\ell}\}$ for $\pi_6$,
\item $\{i, j, j+1, a_1, \dots, a_k\}$ and $\{i+1, b_1, \dots, b_{\ell}\}$ for $\pi_7$,
\item $\{i, i+1\}$ and $\{j, j+1, a_1, \dots, a_k, b_1, \dots, b_{\ell}\}$ for $\pi_8$,
\item $\{i, i+1, j, b_1, \dots, b_{\ell}\}$ and $\{j+1, a_1, \dots, a_k\}$ for $\pi_9$,
\item $\{i, i+1, j+1, a_1, \dots, a_k\}$ and $\{j, b_1, \dots, b_{\ell}\}$ for $\pi_{10}$, and
\item $\{i, i+1, a_1, \dots, a_k, b_1, \dots, b_{\ell}\}$ and $\{j, j+1\}$ for $\pi_{11}$.
\end{itemize}
If $k, \ell > 1$ does not hold, we modify this expression as follows to get the common value
$\tau_i \tau_j(\pi) = \tau_j \tau_i(\pi)$. 
If $k = 1$, we remove the term involving $\pi_2$.
If $\ell = 1$, we remove the term involving $\pi_5$.
If $k = 0$, we remove the terms involving $\pi_2, \pi_6,$ and $\pi_9$.
If $\ell = 0$, we remove the terms involving $\pi_5, \pi_7,$ and $\pi_{10}$.

The reader may wonder whether Case 5 is a degeneration of Case 3.  In fact, it is not, as none of the  degenerations
of Case 3 involve removing $16 - 11 = 5$ terms.

{\bf Case 6:}  {\it The valences $\val_{\pi}(i), \val_{\pi}(i+1), \val_{\pi}(j),$ and $\val_{\pi}(j+1)$ are all positive.
The blocks $B_i(\pi), B_j(\pi), B_{i+1}(\pi),$ and $B_{j+1}(\pi)$ are all distinct.}

This case is shown on the lower right of Figure~\ref{coxeter-two-figure}.

 Let $B_i(\pi) = \{i, a_1, \dots, a_k\}, B_{i+1}(\pi) = \{i+1, b_1, \dots, b_{\ell}\}, B_j(\pi) = \{j, c_1, \dots, c_m\},$ and
 $B_{j+1}(\pi) = \{j+1, d_1, \dots, d_r\}$.  
 We have $k, \ell, m, r \geq 1$.  When $k, \ell, m, r > 1$,
 the common value of 
$\tau_i \tau_j(\pi)$ and $\tau_j \tau_i(\pi)$ is  
$\pi_1 + \cdots + \pi_8 - \pi_9 - \cdots - \pi_{16}$, where $\pi_1 = \pi$ and 
$\pi_2, \dots, \pi_{16}$ are obtained from $\pi$ by replacing $B_i(\pi), B_{i+1}(\pi), B_j(\pi),$ and 
$B_{j+1}(\pi)$ with the 
blocks$\dots$
\begin{itemize}
\item $\{i, a_1, \dots, a_k\}, \{i+1, b_1, \dots, b_{\ell}\}, \{j, j+1\},$ and $\{c_1, \dots, c_m, d_1, \dots, d_r\}$ for $\pi_2$,
\item $\{i, i+1\}, \{a_1, \dots, a_k, b_1, \dots, b_{\ell}\}, \{j, c_1, \dots, c_m\},$ and $\{j+1, d_1, \dots, d_r\}$ for $\pi_3$,
\item $\{i, i+1\}, \{a_1, \dots, a_k, b_1, \dots, b_{\ell}\}, \{j, j+1\},$ and $\{c_1, \dots, c_m, d_1, \dots, d_r\}$ for $\pi_4$,
\item $\{i, i+1, b_1, \dots, b_{\ell}\}, \{a_1, \dots, a_k\}, \{j, j+1, c_1, \dots, c_m\},$ and $\{d_1, \dots, d_r\}$ for $\pi_5$,
\item $\{i, i+1, b_1, \dots, b_{\ell}\}, \{a_1, \dots, a_k\}, \{j, j+1, d_1, \dots, d_r\},$ and $\{c_1, \dots, c_m\}$ for $\pi_6$,
\item $\{i, i+1, a_1, \dots, a_k\}, \{b_1, \dots, b_{\ell}\}, \{j, j+1, c_1, \dots, c_m\},$ and $\{d_1, \dots, d_r\}$ for $\pi_7$,
\item $\{i, i+1, a_1, \dots, a_k\}, \{b_1, \dots, b_{\ell}\}, \{j, j+1, d_1, \dots, d_r\},$ and $\{c_1, \dots, c_m\}$ for $\pi_8$,
\item $\{i, a_1, \dots, a_k\}, \{i+1, b_1, \dots, b_{\ell}\}, \{j, j+1, c_1, \dots, c_m\},$ and $\{d_1, \dots, d_r\}$ for $\pi_9$,
\item $\{i, a_1, \dots, a_k\}, \{i+1, b_1, \dots, b_{\ell}\}, \{j, j+1, d_1, \dots, d_r\},$ and $\{c_1, \dots, c_m\}$ for 
$\pi_{10}$
\item $\{i, i+1\}, \{a_1, \dots, a_k, b_1, \dots, b_{\ell}\}, \{j, j+1, c_1, \dots, c_m\},$ and $\{d_1, \dots, d_r\}$ for 
$\pi_{11}$
\item $\{i, i+1\}, \{a_1, \dots, a_k, b_1, \dots, b_{\ell}\}, \{j, j+1, d_1, \dots, d_r\},$ and $\{c_1, \dots, c_m\}$ for 
$\pi_{12}$
\item $\{i, i+1, b_1, \dots, b_{\ell}\}, \{a_1, \dots, a_k\}, \{j, c_1, \dots, c_m\},$ and $\{j+1, d_1, \dots, d_r\}$ for 
$\pi_{13}$,
\item $\{i, i+1, b_1, \dots, b_{\ell}\}, \{a_1, \dots, a_k\}, \{j, j+1\},$ and $\{c_1, \dots, c_m, d_1, \dots, d_r\}$ for 
$\pi_{14}$,
\item $\{i, i+1, a_1, \dots, a_k\}, \{b_1, \dots, b_{\ell}\}, \{j, c_1, \dots, c_m\},$ and $\{j+1, d_1, \dots, d_r\}$ for
$\pi_{15}$,  and
\item $\{i, i+1, a_1, \dots, a_k\}, \{b_1, \dots, b_{\ell}\}, \{j, j+1\},$ and $\{c_1, \dots, c_m, d_1, \dots, d_r\}$ for 
$\pi_{16}$.
\end{itemize}
If $k, \ell, r, m > 1$ does not hold, we alter the above expression as follows to get the common value
of $\tau_i \tau_j(\pi)$ and $\tau_j \tau_i(\pi)$.
If $k = 1$, we remove the terms involving $\pi_5, \pi_6, \pi_{13},$ and $\pi_{14}$.  
If $\ell = 1$, we remove the terms involving $\pi_7, \pi_8, \pi_{15},$ and $\pi_{16}$.
If $m = 1$, we remove the terms involving $\pi_6, \pi_8, \pi_{10},$ and $\pi_{12}$.
If $r = 1$, we remove the terms involving $\pi_5, \pi_7, \pi_9,$ and $\pi_{11}$.
This concludes the proof of Lemma~\ref{second-coxeter-relation}.
\end{proof}

Finally, we turn to the braid relation $\tau_i \tau_{i+1} \tau_i = \tau_{i+1} \tau_i \tau_{i+1}$.
The proof of this relation will be slightly shorter than that of Lemma~\ref{second-coxeter-relation}.

\begin{lemma}
\label{third-coxeter-relation}
For any $1 \leq i \leq n-2$ we have $\tau_i \tau_{i+1} \tau_i = \tau_{i+1} \tau_i \tau_{i+1}$. 
\end{lemma}

\begin{figure}
\centering
\includegraphics[scale = 0.6]{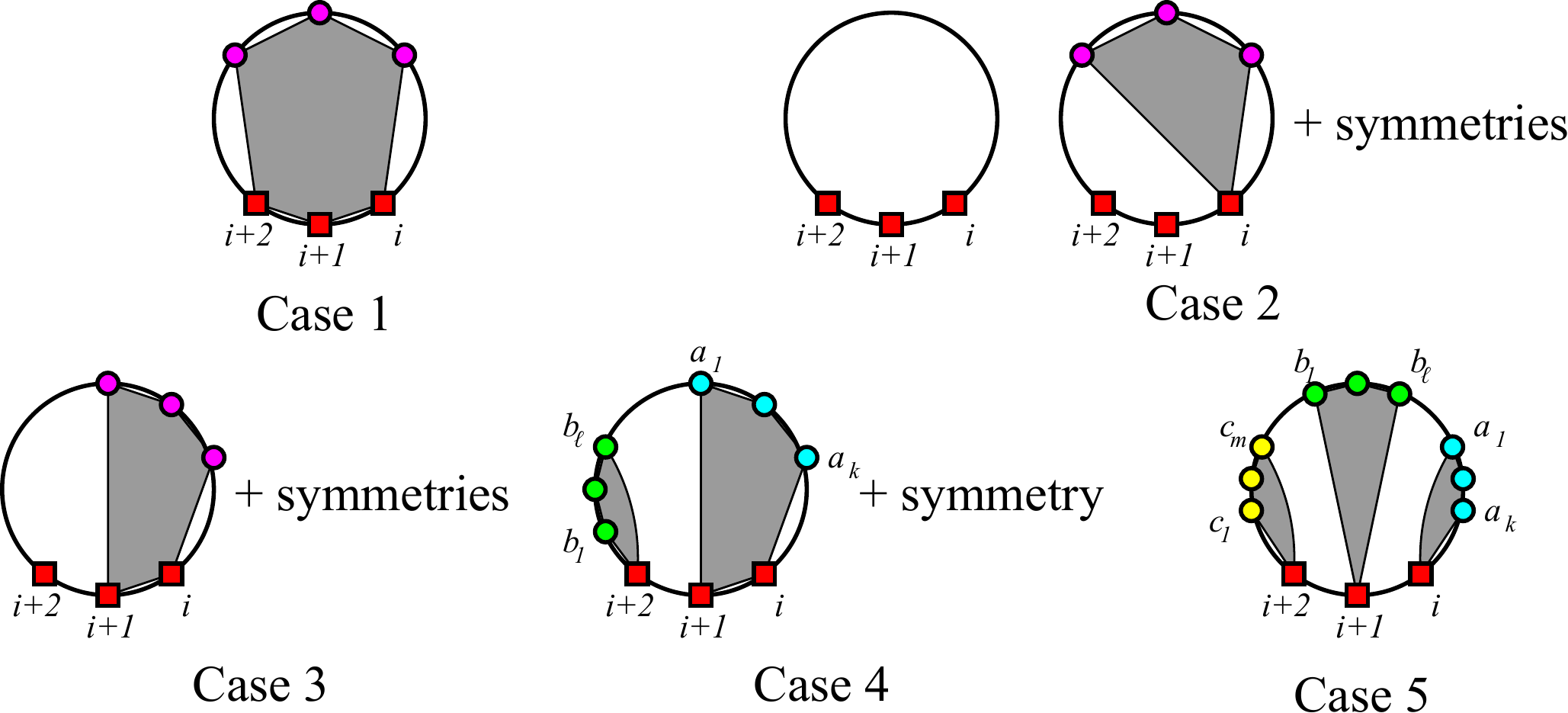}
\caption{The five cases of Lemma~\ref{third-coxeter-relation}.  
The numbers of green, blue, and yellow vertices are all at least one.
The number of purple vertices could be zero.}
\label{coxeter-three-figure}
\end{figure}

\begin{proof}
Let $\pi \in NC(n)$.  We show that $\tau_i \tau_{i+1} \tau_i(\pi) = \tau_{i+1} \tau_i \tau_{i+1}(\pi)$.
Our analysis breaks up into five cases depending on the structure of $\pi$ near the indices $i, i+1,$ and $i+2$.
These cases are illustrated in Figure~\ref{coxeter-three-figure}.

{\bf Case 1:} {\it $B_i(\pi) = B_{i+1}(\pi) = B_{i+2}(\pi)$.}

This case is shown on the upper left of Figure~\ref{coxeter-three-figure}.

In this case we have $\tau_i \tau_{i+1} \tau_i(\pi) = -\pi = \tau_{i+1} \tau_i \tau_{i+1}(\pi)$.

{\bf Case 2:} {\it At least two of the valences $\val_{\pi}(i), \val_{\pi}(i+1), \val_{\pi}(i+2)$ equal $0$.}

This case is shown on the upper right of Figure~\ref{coxeter-three-figure}.

In this case we have 
$\tau_i \tau_{i+1} \tau_i(\pi) = s_i s_{i+1} s_i(\pi) = s_{i+1} s_i s_{i+1}(\pi) =  \tau_{i+1} \tau_i \tau_{i+1}(\pi)$.

{\bf Case 3:} {\it Precisely one of the valences $\val_{\pi}(i), \val_{\pi}(i+1), \val_{\pi}(i+2)$ equals $0$ and
two of the blocks $B_i(\pi), B_{i+1}(\pi), B_{i+2}(\pi)$ coincide.}

This case is shown on the lower left of Figure~\ref{coxeter-three-figure}; two other possibilities are symmetric and not shown.

We leave it to the reader to check that we have 
$\tau_i \tau_{i+1} \tau_i(\pi) = -\pi = \tau_{i+1} \tau_i \tau_{i+1}(\pi)$.

{\bf Case 4:} {\it All of the valences $\val_{\pi}(i), \val_{\pi}(i+1), \val_{\pi}(i+2)$ are positive and 
precisely two of the blocks $B_i(\pi), B_{i+1}(\pi), B_{i+2}(\pi)$ coincide.}

This case is shown in the lower middle of Figure~\ref{coxeter-three-figure}; one other symmetric possibility is not shown.

The fact that $\pi$ is noncrossing and the assumption about valences forces $B_i(\pi) = B_{i+1}(\pi)$ or 
$B_{i+1}(\pi) = B_{i+2}(\pi)$.  We handle the possibility
$B_i(\pi) = B_{i+1}(\pi)$; the argument for the other possibility is symmetric.  
Let $B_i(\pi) = B_{i+1}(\pi) = \{i, i+1, a_1, \dots, a_k\}$ and $B_{i+2}(\pi) = \{i+2, b_1, \dots, b_{\ell}\}$. 
We have $k \geq 0$ and $\ell \geq 1$.
If $k, \ell > 1$, the common value of 
$\tau_i \tau_{i+1} \tau_i(\pi)$ and $\tau_{i+1} \tau_i \tau_{i+1}(\pi)$ is
$\pi_1 + \pi_2 - \pi_3 - \pi_4 - \pi_5$, where $\pi_1, \dots, \pi_5$ are obtained from $\pi$ by replacing the 
blocks $B_i$ and $B_{i+2}$ with the blocks$\dots$
\begin{itemize}
\item $\{i, i+1, i+2\}$ and $\{a_1, \dots, a_k, b_1, \dots, b_{\ell}\}$ for $\pi_1$,
\item $\{i, a_1, \dots, a_k\}$ and $\{i+1, i+2, b_1, \dots, b_{\ell}\}$ for $\pi_2$,
\item $\{i, a_1, \dots, a_k, b_1, \dots, b_{\ell}\}$ and $\{i+1, i+2\}$ for $\pi_3$,
\item $\{i, i+1, i+2, a_1, \dots, a_k\}$ and $\{b_1, \dots, b_{\ell}\}$ for $\pi_4$, and
\item $\{i, i+1, i+2, b_1, \dots, b_{\ell}\}$ and $\{a_1, \dots, a_k\}$ for $\pi_5$.
\end{itemize}
For the ``degenerate" cases where $k = 1, k = 0,$ or $\ell = 1$, we alter the above expression as follows
to get the common value of
$\tau_i \tau_{i+1} \tau_i(\pi)$ and $\tau_{i+1} \tau_i \tau_{i+1}(\pi)$.
If $k = 1$, we remove the term involving $\pi_5$.  If $k = 0$, we remove the terms involving
$\pi_2$ and $\pi_5$.  If $\ell = 1$, we remove the term involving $\pi_4$.  If $\ell = 1$ and $k = 0$,
we remove the terms involving $\pi_1, \pi_2, \pi_4,$ and $\pi_5$.

{\bf Case 5:} {\it All of the valences $\val_{\pi}(i), \val_{\pi}(i+1), \val_{\pi}(i+2)$ are positive and 
the blocks $B_i(\pi), B_{i+1}(\pi),$ and $B_{i+2}(\pi)$ are distinct.}

This case is shown on the lower right of Figure~\ref{coxeter-three-figure}.

Let $B_i(\pi) = \{i, a_1, \dots, a_k\}, B_{i+1}(\pi) = \{i+1, b_1, \dots, b_{\ell}\},$ and 
$B_{i+2}(\pi) = \{i+2, c_1, \dots, c_m\}$.
We have $k, \ell, m \geq 1$. 
If $k, \ell, m > 1$, then
the common value of $\tau_i \tau_{i+1} \tau_i(\pi)$ and $\tau_{i+1} \tau_i \tau_{i+1}(\pi)$ is
$\pi_1 + \cdots + \pi_8 - \pi_9 - \cdots - \pi_{16}$, where $\pi_1 = \pi$ and $\pi_2, \dots, \pi_{16}$ are obtained
from $\pi$ by replacing $B_i(\pi), B_{i+1}(\pi),$ and $B_{i+2}(\pi)$ with the blocks$\dots$
\begin{itemize}
\item $\{i, a_1, \dots, a_k\}, \{i+1, i+2\},$ and $\{b_1, \dots, b_{\ell}, c_1, \dots, c_m\}$ for $\pi_2$,
\item $\{i, i+1\}, \{i+2, a_1, \dots, a_k\},$ and $\{b_1, \dots, b_{\ell}, c_1, \dots, c_m\}$ for $\pi_3$,
\item $\{i, c_1, \dots, c_m\}, \{i+1, i+2\},$ and $\{a_1, \dots, a_k, b_1, \dots, b_{\ell}\}$ for $\pi_4$,
\item $\{i, i+1\}, \{i+2, c_1, \dots, c_m\},$ and $\{a_1, \dots, a_k, b_1, \dots, b_{\ell}\}$ for $\pi_5$,
\item $\{i, i+1, a_1, \dots, a_k\}, \{i+2, b_1, \dots, b_{\ell}\},$ and $\{c_1, \dots, c_m\}$ for $\pi_6$,
\item $\{i, i+1, i+2\}, \{a_1, \dots, a_k, c_1, \dots, c_m\},$ and $\{b_1, \dots, b_{\ell}\}$ for $\pi_7$,
\item $\{i, b_1, \dots, b_{\ell}\}, \{i+1, i+2, c_1, \dots, c_m\},$ and $\{a_1, \dots, a_k\}$ for $\pi_8$, 
\item $\{i, a_1, \dots, a_k\}, \{i+1, i+2, b_1, \dots, b_{\ell}\},$ and $\{c_1, \dots, c_m\}$ for $\pi_9$,
\item $\{i, a_1, \dots, a_k, c_1, \dots, c_m\}, \{i+1, i+2\},$ and $\{b_1, \dots, b_{\ell}\}$ for $\pi_{10}$,
\item $\{i, i+1\}, \{i+2, a_1, \dots, a_k, c_1, \dots, c_m\},$ and $\{b_1, \dots, b_{\ell}\}$ for $\pi_{11}$,
\item $\{i, i+1, i+2, a_1, \dots, a_k\}, \{b_1, \dots, b_{\ell}\},$ and $\{c_1, \dots, c_m\}$ for $\pi_{12}$,
\item $\{i, b_1, \dots, b_{\ell}, c_1, \dots, c_m\}, \{i+1, i+2\},$ and $\{a_1, \dots, a_k\}$ for $\pi_{13}$,
\item $\{i, i+1\}, \{i+2, a_1, \dots, a_k, b_1, \dots, b_{\ell}\},$ and $\{c_1, \dots, c_m\}$ for $\pi_{14}$,
\item $\{i, i+1, b_1, \dots, b_{\ell}\}, \{i+2, c_1, \dots, c_m\},$ and $\{a_1, \dots, a_k\}$ for $\pi_{15}$, and
\item $\{i, i+1, i+2, c_1, \dots, c_m\}, \{a_1, \dots, a_k\},$ and $\{b_1, \dots, b_{\ell}\}$ for $\pi_{16}$.
\end{itemize}
In the ``degenerate" cases where any of $k, \ell, m$ equal $1$, we alter the above expression
as follows to obtain the common value of  $\tau_i \tau_{i+1} \tau_i(\pi)$ and $\tau_{i+1} \tau_i \tau_{i+1}(\pi)$.
If $k = 1$, we remove the terms involving $\pi_8, \pi_{13}, \pi_{15},$ and $\pi_{16}$.  If $\ell = 1$,
we remove the terms involving $\pi_7, \pi_{10}, \pi_{11}, \pi_{12},$ and $\pi_{16}$.  If $m = 1$, we
remove the terms involving $\pi_6, \pi_9, \pi_{12},$ and $\pi_{14}$.
\end{proof}

We arrive at the main result of this section.

\begin{theorem}
\label{module-theorem}
Let $n, k, s \geq 0$.
The $\QQ$-vector space $V(n)$ has the structure of a $\symm_n$-module given by
$s_i.v = \tau_i(v)$, for $1 \leq i \leq n-1$ and $v \in V(n)$.    The subsapces 
$V(n, k, s) \subseteq V(n, k) \subseteq V(n)$ are submodules for this action.
The representing matrices for any permutation  $w \in \symm_n$ with respect to the skein bases
of  $V(n), V(n, k),$ or $V(n, k, s)$ have entries in $\ZZ$.
\end{theorem}

The $\symm_n$-modules $V(n), V(n, k),$ and $V(n, k, s)$ will be called {\it skein modules}
since the action of $\symm_n$ is defined in terms of skein relations.

\begin{proof}
The $\symm_n$-module structure of $V(n)$ follows from Lemmas~\ref{first-coxeter-relation},
\ref{second-coxeter-relation}, and \ref{third-coxeter-relation}.  By Lemma~\ref{sigma-restricts} and 
the definition of $\tau_i$,
every term appearing in $\tau_i(\pi) \in V(n)$ has the same number of  blocks and the same number 
of singleton blocks as $\pi$.  It follows that the subspaces $V(n, k)$ and $V(n, k, s)$ are both closed
under this $\symm_n$-action.  The fact that the entries of the representing matrices of permutations have entries in 
$\ZZ$ is an immediate consequence of our definition of the $\tau_i$ operators and the fact that the skein relations
$\sigma_1 - \sigma_4$ only involve integer coefficients.
\end{proof}

We remark that the skein modules could have been constructed in a slightly different way.
Given an almost noncrossing partition $\pi \in ANC(n)$ which crosses at $i$,
let $\tilde{\sigma}(\pi)$ be the four term linear combination of noncrossing partitions given
in the skein relation $\sigma_4$.  In particular, if $B_i(\pi)$ or $B_{i+1}(\pi)$ has just two elements,
then some terms in $\tilde{\sigma}(\pi)$ will have more singleton blocks than $\pi$.

For $1 \leq i \leq n-1$, let $\tilde{\tau}_i: V(n) \rightarrow V(n)$ be the linear operator
defined  on a noncrossing partition $\pi$ of $[n]$ by
\begin{equation}
\tilde{\tau}_i(\pi) = \begin{cases}
s_i(\pi) & \text{if $s_i(\pi)$ is noncrossing and $B_i(\pi) = B_{i+1}(\pi)$,} \\
-\pi & \text{if $B_i(\pi) = B_{i+1}(\pi)$,} \\
\tilde{\sigma}(s_i(\pi)) & \text{if $s_i(\pi)$ is almost noncrossing.}
\end{cases}
\end{equation}
That is, we have that $\tilde{\tau}_i$ acts on $\pi$ just like $\tau_i$, except we only use 
the flavor $\sigma_4$ of our skein relations to resolve any crossings.
As such, the definition of $\tilde{\tau}_i$ is slightly more streamlined than that of $\tau_i$.

The calculations in this section (without needing to consider various ``degenerate" cases) imply that 
the assignment $s_i \mapsto \tilde{\tau}_i$ for $1 \leq i \leq n-1$ gives $V(n)$ an $\symm_n$-module structure,
which we denote $\widetilde{V}(n)$.  In contrast to the module structure on $V(n)$ 
in Theorem~\ref{module-theorem}, the action of the $s_i$ on $\widetilde{V}(n)$ preserve the total number of blocks of a 
noncrossing partition, but could increase the number of singleton blocks.
For fixed $n$ and $k$, we therefore have an $\symm_n$-stable filtration 
$\widetilde{V}(n, k, \geq n) \subseteq \widetilde{V}(n, k, \geq n-1) \subseteq \cdots \subseteq \widetilde{V}(n, k, \geq 0)$,
where $\widetilde{V}(n, k, \geq s)$ is the subspace of $\widetilde{V}(n)$ spanned by noncrossing partitions 
of $[n]$ with $k$ total blocks and at least $s$ singleton blocks.  It is easy to see that the quotient module
$\widetilde{V}(n, k, \geq s)/\widetilde{V}(n, k, \geq s+1)$ is isomorphic to the module $V(n, k, s)$ of Theorem~\ref{module-theorem}.
To avoid the use of filtrations, we make use of the slightly more complicated action $s_i \mapsto \tau_i$
instead of $s_i \mapsto \tilde{\tau}_i$.

\section{Module structure}
\label{Module}

For any $n, k, s \geq 0$,
Theorem~\ref{module-theorem} gives three  $\symm_n$-modules 
$V(n), V(n, k),$ and $V(n, k, s)$ defined in terms of skein relations.  In this section
we completely determine the isomorphism types of these skein modules.

Since the disjoint unions
$NC(n) = \biguplus_k NC(n, k)$ and $NC(n, k) = \biguplus_s NC(n, k, s)$ give rise to the 
$\symm_n$-module decompositions
$V(n) = \bigoplus_k V(n, k)$ and $V(n, k) = \bigoplus_s V(n, k, s)$, we need only 
determine the isomorphism type of $V(n, k, s)$.  It will turn out that the modules $V(n, k, s)$ 
will be parabolic inductions of modules of the form $V(n, k, 0)$, and that the modules 
$V(n, k, 0)$ are irreducible $\symm_n$-representations of flag shape.

Let $\pi \in NC(n)$ be a skein basis element of $V(n)$.  For any $1 \leq i \leq n-1$, the action of 
$s_i$ on $\pi$ only affects the blocks $B_i(\pi)$ and $B_{i+1}(\pi)$ containing $i$ and $i+1$; all other blocks of $\pi$ are fixed
by this action.
In particular, if the interval $[i, j] = \{i , i+1, \dots, j\}$ is a union of blocks of $\pi$ for some $1 \leq i < j \leq n$,
then the action of the subgroup 
$\symm_{[i, j]} = \langle s_i, s_{i+1}, \dots, s_{j-1} \rangle$ of $\symm_n$ on $\pi$ is completely determined by the 
restriction of $\pi$ to $[i, j]$.  This ``local property" of the action of $\symm_n$ on $V(n)$ is immediate from its combinatorial
definition; we formalize it for future use.

\begin{lemma}
\label{local-property}
Let $\pi \in NC(n)$ and $1 \leq i < j \leq n$.  Assume that the interval $[i, j]$ is a union of blocks of $\pi$.
Let $\bar{\pi} \in NC(j - i + 1)$ be the set partition of $[j - i + 1]$ obtained by restricting $\pi$ to the interval $[i, j]$
and decrementing indices.
Let $w \in \symm_{[i, j]}$ be a permutation of $n$
 fixing indices outside of $[i, j]$ and let $\bar{w}$ be the corresponding permutation in the smaller symmetric
 group $\symm_{j - i + 1}$.  
 
 For any skein basis element $\pi'$ appearing  in $w.\pi \in V(n)$, the interval $[i, j]$ is a union of blocks of $\pi'$.  Moreover,
 the blocks of $\pi'$ which do not meet the interval $[i, j]$ agree with the corresponding blocks of $\pi$.  Finally,
 the  restriction of the skein basis expansion of $w.\pi$ to the interval $[i, j]$ can be obtained by
 computing the skein basis expansion of $\bar{w}.\bar{\pi} \in V(j-i+1)$ and incrementing indices.
\end{lemma}

In this section and the next, we will analyze in detail the action of $\symm_n$ on two particular skein basis elements
$\pi_0, \pi_1 \in NC(n, k, 0)$ given by
$\pi_0 := \{ \{1, 2k, 2k+1, \dots, n-1, n\}, \{2, 2k-1\}, \dots, \{k, k+1\} \}$ and
$\pi_1 := \{ \{1, 2, \dots, n-2k+2\}, \{n-2k+3, n-2k+4\}, \dots, \{n-1, n\} \}$.
These basis elements are shown in Figure~\ref{figure-pi}.

\begin{figure}
\centering
\includegraphics[scale = 0.5]{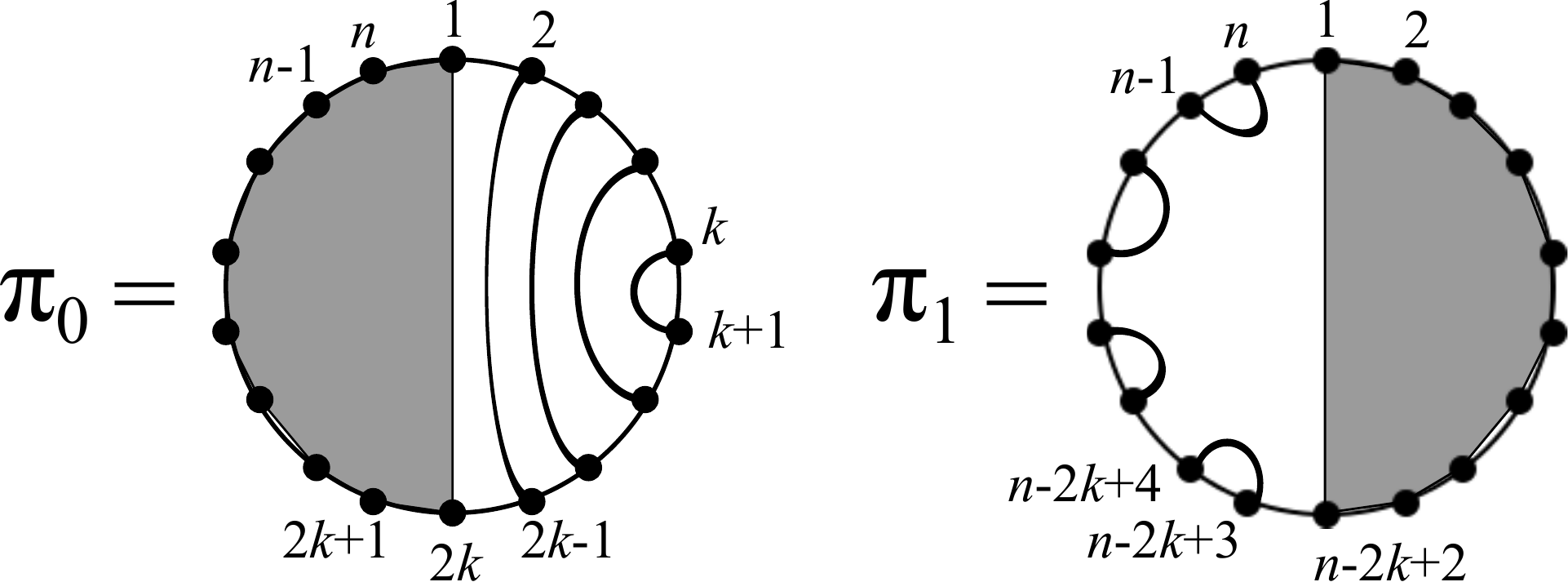}
\caption{The skein basis elements $\pi_0$ and $\pi_1$.}
\label{figure-pi}
\end{figure}

We are ready to prove that the module $V(n, k, 0)$ is isomorphic to the flag-shaped irreducible
$S^{(k, k, 1^{n-2k})}$.  The key tool is Lemma~\ref{pincer-lemma}.

\begin{proposition}
\label{isomorphism-special-case}
We have that $V(n, k, 0) \cong_{\symm_n} S^{(k, k, 1^{n-2k})}$.
\end{proposition}

\begin{proof}
Let $\lambda = (k, k, 1^{n-2k})$.  
Kathy O'Hara and Doron Zeilberger independently observed (unpublished) that the hook length formula for
$\mathrm{dim}(S^{\lambda})$ agrees with Cayley's product formula \cite{Cayley} counting the number 
of noncrossing dissections of a regular $(n-k+2)$-gon with $k-1$ diagonals.
Stanley \cite{Stanley} found a bijection between these dissections and standard Young tableaux of shape
$\lambda$.
Pechenik \cite[Section 2]{Pechenik} gave a bijection from these polygon dissections to the set of noncrossing partitions
$NC(n, k, 0)$.  We therefore have that $\mathrm{dim}(V(n, k, 0)) = \mathrm{dim}(S^{\lambda})$, so that 
the dimensions of the modules on either side of the alleged isomorphism agree.  
By Lemma~\ref{pincer-lemma},
it is enough to show that 
$[\symm_{\lambda}]_+$ does not kill $V(n, k, 0)$, but $[\symm_{\mu}]_+$ does kill $V(n, k, 0)$ for all partitions
$\mu \succ \lambda$.

To see that $[\symm_{\lambda}]_+$ does not kill $V(n, k, 0)$, it suffices to exhibit a  skein basis element
$\pi \in NC(n, k, 0)$ such that $[\symm_{\lambda}]_+.\pi \neq 0$.  Let 
$\pi_0$ be defined as in Figure~\ref{figure-pi}.  We show that
$[\symm_{\lambda}]_+.\pi_0 \neq 0$ 
by showing
 that the coefficient of $\pi_0$ in the skein basis expansion of $[\symm_{\lambda}]_+.\pi_0$ is nonzero.

To begin, observe that $\symm_{\lambda} = \symm_k \times \symm_k$ preserves the letters $\{1, 2, \dots, k \}$
as well as the letters $\{k+1, k+2, \dots, 2k\}$.  The restriction of the set partition $\pi$ to either of these two
letter sets consists entirely of singleton blocks. 
%In fact, the noncrossing condition
%guarantees that  $\pi_0$ 
%is the unique set partition in $NC(n, k, 0)$ whose restrictions to the sets $\{1, 2, \dots, k\}$
%and $\{k+1, k+2, \dots, 2k\}$ consists entirely of singleton blocks.

For any $1 \leq i \leq k-1$ or
$k+1 \leq i \leq 2k-1$, we claim that the adjacent transposition
$s_i$ preserves the subspace of $V(n, k, 0)$ spanned by
$NC(n, k, 0) - \{\pi_0\}$.  
To see this, let $\pi \in NC(n, k, 0) - \{\pi_0\}$ and consider a typical skein basis element $\pi'$
appearing in $s_i.\pi$.  
Figure~\ref{SkeinRelationsFigure} shows that either $\pi = \pi'$ or $i$ and $i+1$ are in the same block of 
$\pi'$.
Since $i$ and $i+1$ are not in the same block of $\pi_0$, this means that $\pi' \neq \pi_0$.
Therefore, the $\QQ$-span of
$NC(n, k, 0) - \{\pi_0\}$ is closed under the action of $s_i$.
By our condition on $i$,
it follows that the $\QQ$-span of
$NC(n, k, 0) - \{\pi_0\}$ is closed under the action of the Young subgroup $\symm_{\lambda}$.

For any $1 \leq i \leq k-1$ or
$k+1 \leq i \leq 2k-1$,
we have that the coefficient of $\pi_0$ in $s_i.\pi_0$ is $1$.  
Since the $\QQ$-span of
$NC(n, k, 0) - \{\pi_0\}$ is closed under the action of $\symm_{\lambda}$,
it follows
that the coefficient of $\pi_0$ in $[\symm_{\lambda}]_+.\pi_0$ is 
$|\symm_{\lambda}| = |\symm_k \times \symm_k| = (k!)^2$. Since this coefficient is nonzero, we have 
$[\symm_{\lambda}]_+.\pi_0 \neq 0$, as desired.  

On the other hand, let $\mu \vdash n$ be a partition satisfying $\mu \succ \lambda$.  We want to show
that $[\symm_{\mu}]_+$ kills $V(n, k, 0)$.  It is enough to show that 
$[\symm_{\mu}]_+.\pi = 0$ for all $\pi \in NC(n, k, 0)$.

Let $\pi \in NC(n, k, 0)$ and
write $\mu = (\mu_1, \dots, \mu_r)$.  If $\mu_1 > k$, then the fact that $\pi$ is noncrossing,
has $k$ blocks, and has no singleton blocks implies that there is an index $1 \leq i \leq \mu_1$ such 
that $i$ and $i+1$ are in the same block of $\pi$.  This means that $s_i.\pi = -\pi$, so that 
the action of $1 + s_i$ kills $\pi$.
Moreover, we have that $s_i \in \symm_{\mu}$.  
Since the group algebra element $[\symm_{\mu}]_+$ factors as 
$[\symm_{\mu}]_+ = a (1+s_i)$ for some $a \in \QQ \symm_n$, we conclude that $[\symm_{\mu}]_+.\pi = a(1+s_i).\pi = a.0 =  0$,
as desired.

By the last paragraph, we may assume that $\mu_1 \leq k$.
Since $\mu \succ \lambda$,
this implies that $\mu_1 = \mu_2 =  k$.  The condition $\mu \succ \lambda$ forces $\mu_3 > 1$.

If any two indices in the 
set $\{1, 2, \dots, k\}$ are in the same block of $\pi$,  
we get that
$[\symm_{\mu}]_+.\pi = 0$ as before.
We may therefore
assume that the indices in $\{1, 2, \dots, k\}$ are all in distinct blocks of $\pi$.  Similarly, we may
assume that the indices in $\{k+1, k+2, \dots, 2k\}$ are all in distinct blocks of $\pi$.

Since $\pi$
has $k$ blocks, no singleton blocks, and is noncrossing, the assumptions of the last paragraph force the blocks of 
$\pi$ to be
$\{k, k+1\}, \{k-1, k+2\}, \dots, \{2, 2k-1\},$ and $\{1, 2k, 2k+1, 2k+1, \dots, n\}$.  
In other words, we have that $\pi = \pi_0$, the skein basis element of Figure~\ref{figure-pi}.
In particular,
we have that
$2k$ and $2k+1$ are in the same block of $\pi$.  Since $\mu_3 > 1$, we have that 
$s_{2k} \in \symm_{\mu}$ and there is a group algebra element
$a \in \QQ \symm_n$ such that 
$[\symm_{\mu}]_+ = a (1+s_{2k})$.  We have that
$[\symm_{\mu}]_+.\pi = a (1+s_{2k}).\pi = a.0 = 0$.
\end{proof}

With Proposition~\ref{isomorphism-special-case} in hand, it is a simple matter to determine the isomorphism
types of $V(n), V(n, k),$ and $V(n, k, s)$.

\begin{theorem}  
\label{isomorphism}
We have the following isomorphisms of 
$\symm_n$-modules.
\begin{align}
V(n) &\cong_{\symm_n} \bigoplus_{k \geq 0} V(n, k), \\
V(n, k) &\cong_{\symm_n} \bigoplus_{s \geq 0} V(n, k, s), \\
V(n, k ,s) &\cong_{\symm_n}  
\Ind_{\symm_{n-s} \times \symm_s}^{\symm_n}(S^{(k-s, k-s, 1^{n-2k+s})} \otimes \mathrm{triv}_{\symm_s}).
\end{align}
\end{theorem}

\begin{proof}
The first and second identifications are trivial.
Proposition~\ref{isomorphism-special-case} gives the third identification
when $s = 0$.  

To prove the third isomorphism in general, observe that we can think of a noncrossing partition 
$\pi \in NC(n, k, s)$ as an ordered pair $(T, \pi')$, where $T$ is an $s$-element subset of $[n]$ and
$\pi' $ is the set partition in $NC(n-s, k-s, 0)$  obtained by removing the singleton
blocks of $\pi$ and decrementing indices.  We get a corresponding
vector space decomposition $V(n, k, s) = \bigoplus_{T \in {[n] \choose s}} V(n, k, T)$,
where $V(n, k, T)$ is the subspace of $V(n, k, s)$ with basis given by those set partitions
$\pi \in NC(n, k, s)$ whose corresponding pair is of the form $(T, \pi')$.  For $1 \leq i \leq n-1$, the adjacent transposition
$s_i$ acts on pairs by
\begin{equation*}
s_i: (T, \pi') \mapsto \begin{cases}
(s_i.T, \pi') & \text{if at least one of $i$ or $i+1$ is in $T$,} \\
(T, s_j.\pi') & \text{if $i, i+1 \notin T$,}
\end{cases}
\end{equation*}
where $i$ is the $j^{th}$-smallest element in $[n] - T$.
The third isomorphism follows from the definition of induced representations.
\end{proof}

\begin{example}
Let us compute the isomorphism type of the $\symm_6$-module $V(6, 3)$.  Since a set partition of $[6]$ with $3$ blocks can have at most
$2$ singleton blocks, we have 
\begin{equation*}
V(6, 3) = V(6, 3, 0) \oplus V(6, 3, 1) \oplus V(6, 3, 2).
\end{equation*}
The module $V(6, 3, 0)$ is isomorphic to the irreducible $S^{(3,3)}$.  Applying Monk's Rule, the module 
$V(6, 3, 1)$ is isomorphic to
\begin{equation*}
V(6, 3, 1) = \Ind_{\symm_5 \times \symm_1}^{\symm_6}(S^{(2,2,1)} \otimes S^{(1)}) = S^{(2,2,1,1)} \oplus S^{(2,2,2)} \oplus S^{(3,2,1)}.
\end{equation*}
Applying the Pieri Rule, the module $V(6, 3, 2)$ is isomorphic to 
\begin{equation*}
V(6, 3, 2) = \Ind_{\symm_4 \times \symm_2}^{\symm_6}(S^{(1,1,1,1)} \otimes S^{(2)}) = S^{(2,1,1,1,1)} \oplus S^{(3,1,1,1)}.
\end{equation*}
\end{example}

In light of Proposition~\ref{isomorphism-special-case}, the third isomorphism in 
Theorem~\ref{isomorphism} could also be written as
\begin{equation}
\label{isomorphism-consequence}
V(n, k ,s) \cong_{\symm_n}  
\Ind_{\symm_{n-s} \times \symm_r}^{\symm_n}(V(n-s, k-s, 0) \otimes \mathrm{triv}_{\symm_s}).
\end{equation}
In other words, the representations $V(n), V(n, k),$ and $V(n, k, s)$ are built out of copies
of representations of the form $V(n, k, 0)$.
The fact that the $V(n, k, 0)$ are irreducible of flag shape will play a key role in the next section.

When $n = 2k$, the skein  module $V(2k, k) = V(2k, k, 0)$ coincides with a representation of a Temperley-Lieb algebra studied
in \cite{PPR}.
Given a parameter $x$, the {\em Temperley-Lieb algebra} $TL_n(x)$ 
has $\QQ$-basis given by the collection of noncrossing perfect matchings  on the set $[2n]$.  
In drawing these matchings, the elements of $[2n]$ are arranged in two columns.  For example, 
a $\QQ$-basis of $TL_3(x)$ is shown below.

\begin{center}
\includegraphics[scale = 0.3]{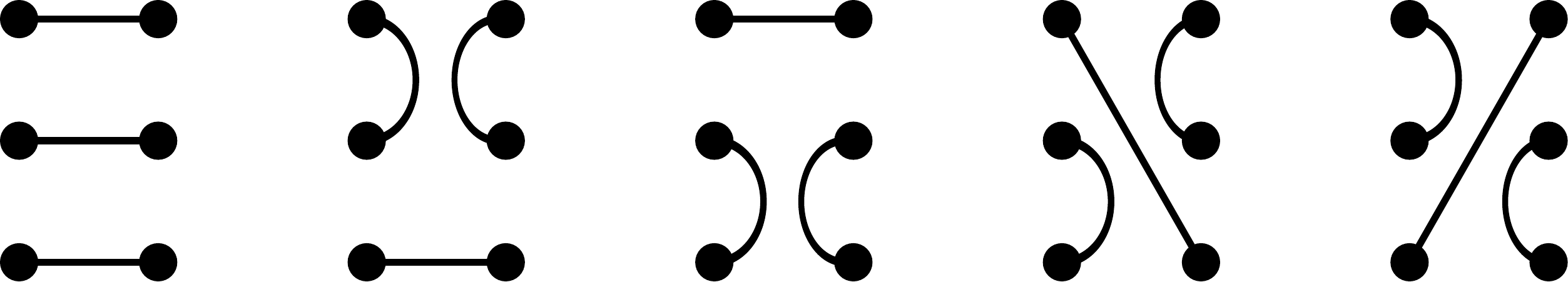}
\end{center}

Multiplication in the Temperley-Lieb algebra $TL_n(x)$ is defined on basis elements by concatenation, followed by 
replacing any loops with factors of the parameter $x$. 
An example of multiplication within $TL_6(x)$ is shown below. 
\begin{center}
\includegraphics[scale = 0.2]{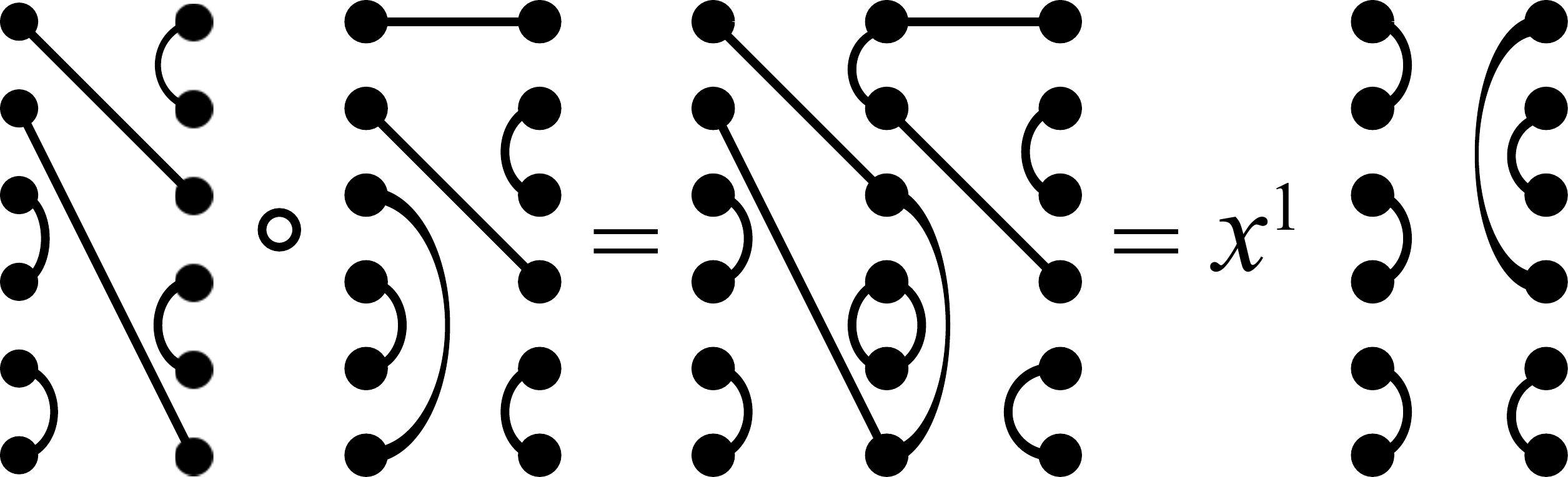}
\end{center}
This multiplication rule gives $TL_n(x)$ the structure of a unital associative $\QQ$-algebra.
The algebra $TL_n(x)$ has generators $t_1, t_2, \dots, t_{n-1}$ shown below.
\begin{center}
\includegraphics[scale = 0.2]{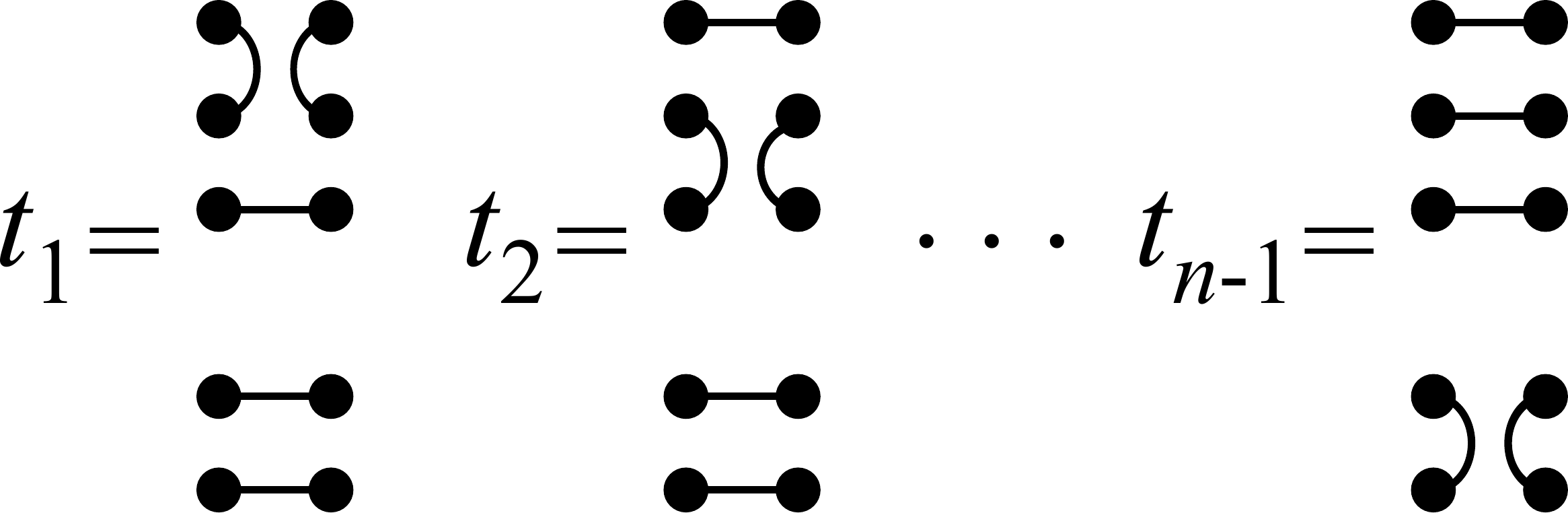}
\end{center}
With respect to this generating set, the algebra $TL_n(x)$ has presentation given by
\begin{equation}
\begin{cases}
t_i^2 = x t_i,  & \\
t_i t_j = t_j t_i & \text{for $|i - j| > 1$,} \\
t_i t_{i \pm 1} t_i = t_i.  &
\end{cases}
\end{equation}

At the parameter $x = -2$, the algebra $TL_n(-2)$ can be realized as a quotient of the symmetric group algebra $\QQ \symm_n$.
Indeed, for $n \geq 3$
the ring epimorphism $\varphi: \QQ \symm_n \rightarrow TL_n(-2)$ given by $\varphi: s_i \mapsto (1 + t_i)$ for $1 \leq i \leq n-1$
has kernel equal to the two-sided ideal $I$ in $\QQ \symm_n$ generated by the group algebra element
$[\symm_3]_- = e - s_1 - s_2 + s_1 s_2 + s_2 s_1 - s_1 s_2 s_1$.  
Precomposition with the map $\varphi$ gives any $TL_n(-2)$-module the structure of 
a $\QQ \symm_n$-module.

For $k \geq 0$, let $W(2k, k)$ denote the $\QQ$-vector space spanned by noncrossing perfect matchings on the set $[2k]$.
The Temperley-Lieb algebra $TL_{2k}(-2)$ acts on $W(2k, k)$ as follows.
Roughly speaking, the generator $t_i$ of $TL_{2k}(-2)$ attaches a  $) ($ figure between the vertices $i$ and $i+1$ of a noncrossing perfect matching
$\pi$ drawn on a disk.  If this creates a loop, this loop is traded for a factor of $-2$.
Two examples are shown below for the case $k = 4$.
\begin{center}
\includegraphics[scale = 0.4]{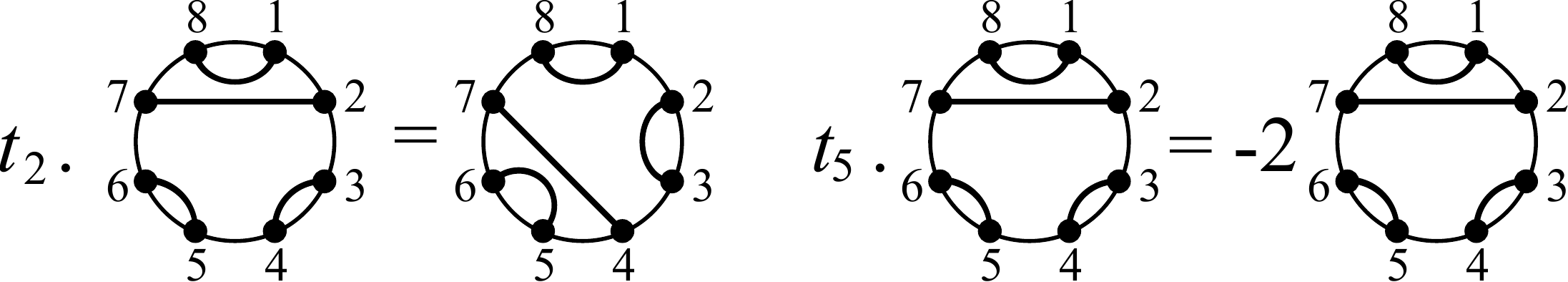}
\end{center}

More formally,
given $1 \leq i \leq 2k-1$ and a noncrossing matching $\pi \in W(2k, k)$, we define $t_i.\pi$ to be
\begin{equation}
t_i.\pi := \begin{cases}
-2 \pi & \text{if $B_{\pi}(i) = B_{\pi}(i+1),$} \\
\pi' & \text{if $B_{\pi}(i) \neq B_{\pi}(i+1),$}
\end{cases}
\end{equation}
where $\pi'$ is the noncrossing perfect matching obtained from $\pi$ by replacing 
the blocks $\{i, j\}$ and $\{i+1, \ell\}$ with $\{i, i+1\}$ and $\{j, \ell\}$.
Using the map $\varphi$, we get the $\symm_{2k}$-module structure on
$W(2k, k)$ studied in \cite{PPR}. 

 It follows from the definitions that the identity map of vector spaces
gives an isomorphism of $\symm_{2k}$-modules
$V(2k, k) \cong_{\symm_{2k}} W(2k, k)$ which identifies the skein basis of
$V(2k, k)$ with the noncrossing perfect matching basis of $W(2k, k)$.
The noncrossing perfect matching basis of $W(2k, k)$ can also be identified with the 
KL cellular basis of the $\symm_{2k}$-irreducible $S^{(k, k)}$.
The identification of $W(2k, k)$ with the irreducible $S^{(k, k)}$ is a special case of 
Theorem~\ref{isomorphism}.

Temperley-Lieb algebras can be used to give an action of $\QQ \symm_n$ on each of the spaces
$V(n, k, s)$ in a way that generalizes the case of $V(2k, k, 0)$.  However, the module structure so attained differs from the 
skein module.

To distinguish our modules, we let $W(n, k, s)$ denote another copy of the $\QQ$-vector space $V(n, k, s)$ spanned by noncrossing 
partitions of $[n]$ with $k$ total blocks and $s$ singleton blocks.  We start by defining an action of $TL_n(-2)$ on $W(n, k, 0)$
(i.e., the case of no singleton blocks).  For $1 \leq i \leq n-1$ and $\pi \in NC(n, k, 0)$, the generator $t_i \in TL_n(2)$
acts by attaching a $) ($ figure at the vertices $i$ and $i+1$ in $\pi$ and replacing any resulting loop with a factor of $-2$.
An example is shown below for $(n, k, r) = (7, 3, 1)$.

\begin{center}
\includegraphics[scale = 0.2]{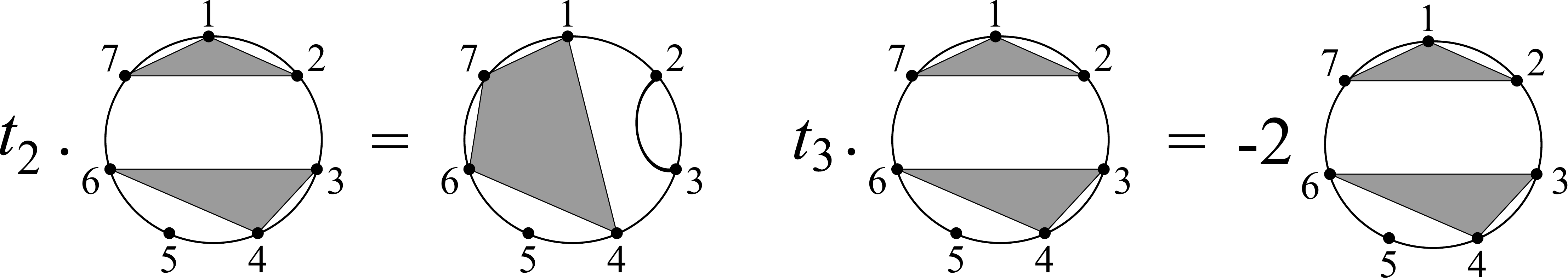}
\end{center}

More formally, we set 
\begin{equation}
t_i.\pi := \begin{cases}
-2 \pi & \text{if $B_{\pi}(i) = B_{\pi}(i+1),$} \\
\pi' & \text{if $B_{\pi}(i) \neq B_{\pi}(i+1),$}
\end{cases}
\end{equation}
where
$\pi'$ is obtained from $\pi$ by replacing the blocks
$\{i, a_1, \dots, a_k\}$ and $\{i+1, b_1, \dots, b_{\ell}\}$ with the blocks
$\{i, i+1\}$  and $\{a_1, \dots, a_k, b_1, \dots, b_{\ell}\}$.
The map $\varphi: \QQ \symm_n \rightarrow TL_n(-2)$ gives $W(n, k, 0)$ the structure of a 
$\symm_n$-module.  
For the general case of $s$ singleton blocks, we define a $\symm_n$-module structure on $W(n, k, s)$ by
$W(n, k, s) := \mathrm{Ind}_{\symm_{n-s} \times \symm_s}^{\symm_n}(W(n-s, k-s, 0) \otimes \mathrm{triv}_{\symm_s})$.

The $\symm_n$-modules $V(n, k, s)$ and $W(n, k, s)$ are {\em not} in general isomorphic.  
There are two reasons for this: one combinatorial and one algebraic.

From a combinatorial point of view, the action of the generators $t_i \in TL_n(-2)$ 
can only add doubleton blocks
to noncrossing partitions.  This means that the subspace of $W(6, 2, 0)$ spanned
by noncrossing partitions with at least one doubleton block
is a (nontrivial) $\symm_6$-submodule, so that $W(6, 2, 0)$ is not irreducible.

From an algebraic point of view, Theorem~\ref{isomorphism} asserts that 
$V(n, k, 0)$ is the flag shaped irreducible $S^{(k, k, 1^{n-2k})}$.  On the other hand,
since the $\symm_n$-structure on $W(n, k, 0)$ factors through the map
$\varphi: \QQ \symm_n \rightarrow TL_n(-2)$, it follows that 
$W(n, k, 0)$ is a direct sum of $\symm_n$-irreducibles corresponding to partitions $\lambda \vdash n$ with at most 
two rows.  In particular, unless $n = 2k$, the flag shaped irreducible corresponding to 
$\lambda = (k, k, 1^{n-2k})$ cannot appear as a component of $W(n, k, 0)$.

\begin{problem}
\label{temperley-lieb-problem}
Determine the isomorphism type of the $\symm_n$-module $W(n, k, s)$.
\end{problem}

As a first step towards solving Problem~\ref{temperley-lieb-problem}, observe that the action of 
a generator $t_i \in TL_n(-2)$ on a basis element $\pi \in W(n, k, 0)$ either leaves the number of doubleton blocks in
$\pi$ unchanged or increases the number of doubleton blocks in $\pi$ by one.  This implies that the 
$\symm_n$-module $W(n, k, s)$ has a $\symm_n$-stable filtration
\begin{equation}
W\left(n, k, s, \geq \left\lfloor \frac{n}{2} \right\rfloor \right) \subseteq  \cdots \subseteq W(n, k, s, \geq 1) \subseteq W(n, k, s, \geq 0) = W(n, k, s),
\end{equation}
where $W(n, k, r, \geq d)$ denotes the subspace of $W(n, k, s)$ spanned by noncrossing partitions of $[n]$ with
$k$ total blocks, $s$ singleton blocks, and at least $d$ doubleton blocks.
If we could determine the isomorphism type of the quotient modules 
$W(n, k, s, \geq d) / W(n, k, s, \geq d+1)$, we could solve
Problem~\ref{temperley-lieb-problem}.

Temperley-Lieb algebras arise naturally as {\it centralizer algebras} for the diagonal action of $GL_2(\CC)$ on 
the tensor space $(\CC^2)^{\otimes n}$ (where $\CC^2$ carries the defining representation).
It may be interesting to give the space $V(n, k, s)$ the structure of an associative algebra by identifying
it with a centralizer algebra for some group action.

While they have a nice combinatorial definition, the role played by the modules
$W(n, k, s)$ in cyclic sieving is unclear.
Part of the reason for this is that no permutation $w \in \symm_n$ can model the action of rotation
or reflection
on these modules in general (indeed, this is already impossible when $(n, k, s) = (6, 2, 0)$).
We will focus on the skein modules $V(n, k, s)$ for the remainder of the paper.

\section{The long element and the long cycle}
\label{Long}

In the last section, we determined the isomorphism types of the $\symm_n$-modules 
$V(n), V(n, k),$ and $V(n, k, s)$.  In this section, we will show that the representing matrices
for certain permutations $w \in \symm_n$, with respect to the skein bases of these 
modules, are (up to sign) permutation matrices for natural symmetries of noncrossing
partitions.  
By the remarks following Theorem~\ref{isomorphism}, it is enough to consider the action of $\symm_n$
on the module $V(n, k, 0)$.
We will focus on the cases of the {\it long cycle} $c = (1, 2, \dots, n) \in \symm_n$
and the {\it long element} $w_0 \in \symm_n$ whose one-line notation is given by
$w_0 = n(n-1) \dots 21$.

The action of the long element $w_0 \in \symm_n$ is easier to prove and we examine it first.

\begin{proposition}
\label{long-element-action}
Let $n, k > 0$ and consider the representation $V(n, k, 0)$ with its skein 
basis $NC(n, k, 0)$.  Let $w_0 \in \symm_n$ denote the long element.

For any noncrossing partition $\pi \in NC(n, k, 0)$, we have that 
\begin{equation}
w_0.\pi = (-1)^{\lfloor \frac{n}{2} \rfloor}t(\pi),
\end{equation}
where $t(\pi) \in NC(n, k, 0)$ is the set partition obtained by performing the reflection of the disc
which permutes vertices according to $w_0$.
\end{proposition}

\begin{proof}
For this proof, we extend scalars to $\CC$ and replace $V(n, k, 0)$ with $\CC \otimes V(n, k, 0)$.
Let $T: V(n, k, 0) \rightarrow V(n, k, 0)$ be the $\CC$-linear operator induced by reflection on $NC(n, k, 0)$.
For any $1 \leq i \leq n-1$,  the definition of the action of $s_i$ on $V(n, k, 0)$ gives
$T  s_i  = s_{n-i}  T$ 
as linear operators on $V(n, k, 0)$.  On the other hand, we also have that
$w_0 s_{n-i} = s_i w_0$ as operators 
on $V(n, k, 0)$.  Putting these equations together, we see that the operator
$w_0 T: V(n, k, 0) \rightarrow V(n, k, 0)$ commutes with the action of $s_i$ on $V(n, k, 0)$ 
for all $1 \leq i \leq n-1$.  Since the $s_i$ generate the action of $\symm_n$, this means that 
the map $w_0 T$ commutes with the action of $\symm_n$.
By Proposition~\ref{isomorphism-special-case},
the $\symm_n$-module $V(n, k, 0)$ is irreducible.
Schur's Lemma implies that there is some  scalar $\gamma \in \CC$
such that $w_0 = \gamma T$ as operators.  

The fastest way to prove that $\gamma = (-1)^{\lfloor \frac{n}{2} \rfloor}$ 
is probably to use the Murnaghan-Nakayama
rule.  Indeed, the permutation $w_0$ has $\frac{n}{2}$ $2$-cycles if $n$ is even
and $\frac{n-1}{2}$ $2$-cycles and a single $1$-cycle if $n$ is odd.  If $\lambda = (k, k, n-2k)$, the 
evaluation $\chi^{\lambda}(w_0)$ of the
irreducible character $\chi^{\lambda}: \symm_n \rightarrow \CC$ has sign given by the parity of the number
of vertical dominos in a domino tiling of the Ferrers diagram of $\lambda$ if $n$ is even and
the number of vertical dominos in a tiling of 
$\lambda$ consisting of one monomino and $\frac{n-1}{2}$ dominos
if $n$ is odd.  In either case, there is a tiling with $\lfloor \frac{n}{2} \rfloor$ vertical dominos, so
it follows that this sign is $(-1)^{\lfloor \frac{n}{2} \rfloor}$.  
This means that $\gamma = (-1)^{\lfloor \frac{n}{2} \rfloor}$.
\end{proof}

A similar Schur's Lemma argument was 
used by Stembridge \cite{Stembridge} to prove that $w_0 \in \symm_n$ acts by a sign times evacuation on
the KL cellular basis for arbitrary shapes and later adapted by the author \cite{RhoadesJDT} to show that
$c \in \symm_n$ acts by a sign times promotion on the KL cellular basis for rectangular shapes. 
Pechenik \cite{Pechenik} described a bijection from flag shaped tableaux to noncrossing partitions
which sends evacuation to reflection. 
Proposition~\ref{long-element-action}
therefore implies that the representing matrices for $w_0$ are the same in the KL and skein bases.  On the other hand,
the representing matrices for $c$ with respect to these bases are very different.  

We start our analysis of the action of $c$ on the skein basis by determining the action
of $c$ on the two skein basis elements $\pi_0$ and $\pi_1$ shown in 
Figure~\ref{figure-pi}.
This technical lemma will be used to determine a pair 
of scalars in a Schur's Lemma-type argument.

\begin{lemma}
\label{sign-determination-lemma}
Let $n, k > 0$ be such that $n \geq 2k$. Define noncrossing partitions $\pi_0, \pi_1 \in NC(n, k, 0)$ 
as in Figure~\ref{figure-pi}.

In the $\symm_n$-module $V(n, k, 0)$ we have that
$c.\pi_0 = (-1)^{n+1} r(\pi_0)$ and $c.\pi_1 = (-1)^{n+1} r(\pi_1)$,
where $r: NC(n, k, 0) \rightarrow NC(n, k, 0)$ is clockwise rotation.
\end{lemma}

The somewhat harrowing inductions involved in the proof of this lemma will hopefully persuade the reader
that  $c$ acting as plus or minus rotation on the skein basis should not be taken for granted.

\begin{proof}
If $n = 2k$, the basis elements $NC(2k, k, 0)$ are noncrossing perfect matchings and this 
result follows from Petersen, Pylyavskyy, and Rhoades \cite[Lemma 4.3]{PPR}.  We therefore assume that $n > 2k$.
We examine the actions of $c$
on $\pi_0$ and $\pi_1$
separately.
We handle the case of $\pi_1$ first.

To calculate the action of $c = s_1 s_2 \cdots s_{n-1}$ on $\pi_1$, we consider applying the operators
$s_{n-1}, s_{n-2}, \dots,$ and $s_1$ (in that order) to the noncrossing partition $\pi_1$.  The form of the
resulting elements of $V(n)$ will change after the ``break point" $s_{n-2k+3}$ and again after another 
``break point" at $s_3$.  
The action 
of $s_i s_{i+1} \cdots s_{n-1}$   on $\pi_1$ for $n-2k+3 \leq i \leq n-1$
is as follows.

{\bf Claim 1:} {\it Suppose $n-2k+3 \leq i \leq n-1$.  Then the element
$s_i s_{i+1} \cdots s_{n-1}.\pi_1 \in V(n)$ is equal to the following expression:

\begin{equation}
\label{claim-one-equation}
s_i s_{i+1} \cdots s_{n-1}.\pi_1  = \begin{cases}
- \pi_1^{(i,1)} & \text{if $n-i$ is odd,} \\
- \pi_1^{(i,2)} - \pi_1^{(i,3)} & \text{if $n-i$ is even,} 
\end{cases}
\end{equation}

where $\pi_1^{(i,1)}, \pi_1^{(i,2)},$ and $\pi_1^{(i,3)}$ are the following noncrossing partitions:
\begin{itemize}
\item $\pi_1^{(i,1)}$ has blocks $\{1, 2, \dots, n-2k+2\}, \{n-2k+3, n-2k+4\}, \{n-2k+4, n-2k+5\}, \dots,
\{i-2, i-1\}, \{i, n\}, \{i+1, i+2\}, \{i+3, i+4\}, \dots, \{n-2, n-1\}$,
\item $\pi_1^{(i,2)}$ has blocks
$\{1, 2, \dots, n-2k+2\}, \{n-2k+3, n-2k+4\}, \{n-2k+4, n-2k+5\}, \dots,
\{i+2, i+3\}, \{i+1, n\}, \{i-1, i\}, \{i-3, i-2\}, \dots, \{n-2, n-1\}$, and
\item $\pi_1^{(i,3)}$ has blocks
$\{1, 2, \dots, n-2k+2\}, \{n-2k+3, n-2k+4\}, \{n-2k+4, n-2k+5\}, \dots,
\{i, i+1\}, \{i-1, n\}, \{i-3, i-2\}, \dots, \{n-2, n-1\}$.
\end{itemize}}

We prove Claim 1 by downward induction on $i$.  If $i = n-1$,  since $\{n-1, n\}$ is a block of $\pi_1$ we have
$s_{n-1}.\pi_1 = -\pi_1$ and the result holds.  

For the inductive step, we  apply the operator
$s_{i-1}$ to the right hand side of Equation~\ref{claim-one-equation}. 
 If $n - i$ is odd,
then  $B_{i-1}(\pi_1^{(i,1)}) \neq B_i(\pi_1^{(i,1)})$. 
 It follows that 
$s_{i-1}.(-\pi_1^{(i,1)}) = -\pi_1^{(i-1,2)} - \pi_1^{(i-1,3)}$.  If $n - i$ is even, then
$B_{i-1}(\pi_1^{(i,2)}) = B_i(\pi_1^{(i,2)})$
but $B_{i-1}(\pi_1^{(i,3)}) \neq B_i(\pi_1^{(i,3)})$.  
It follow that 
$s_{i-1}.(-\pi_1^{(i,2)} - \pi_1^{(i,3)}) = \pi_1^{(i,2)} - \pi_1^{(i,2)} - \pi_1^{(i-1,1)} = -\pi_1^{(i-1,1)}$.
This completes the proof of Claim 1.

By Claim 1, we  get that 
$s_{n-2k+3} \cdots s_{n-2} s_{n-1}.\pi_1 = - \pi_1^{(n-2k+3,1)}$.  Applying the operator
$s_{n-2k+2}$ to this expression gives the following element of $V(n)$.  (Here we use the assumption that 
$n > 2k$, so that the ``big block" of $\pi_1^{(n-2k+3,1)}$ has $> 2$ elements.)
\begin{equation}
\label{claim-one-to-two-equation}
s_{n-2k+2} s_{n-2k+3} \cdots s_{n-2} s_{n-1}.\pi_1 =
- \pi_1^{(n-2k+3,1)} - \pi_1'^{(n-2k+3,1)} +  \pi_1''^{(n-2k+3,1)},
\end{equation}
where
\begin{itemize}
\item $\pi_1'^{(n-2k+3,1)}$ has blocks 
$\{1, 2, \dots, n-2k+1, n\}, \{n-2k+2, n-2k+3\}, \{n-2k+4, n-2k+5\}, \dots,$ and $\{n-2, n-1\}$, and
\item $\pi_1''^{(n-2k+3,1)}$ has blocks
$\{1, 2, \dots, n-2k+1\}, \{n-2k+2, n-2k+3, n\}, \{n-2k+4, n-2k+5\}, \dots,$ and $\{n-2,n-1\}$.
\end{itemize}

For $3 \leq i \leq n-2k+1$, the expansion of
$s_i s_{i+1} \cdots s_n.\pi_1$ in the skein basis is as follows.

{\bf Claim 2:} {\it Suppose $3 \leq i \leq n-2k+1$.  Then the element 
$s_i s_{i+1} \cdots s_{n-1}.\pi_1 \in V(n)$ is given by
\begin{equation}
\label{claim-two-equation}
s_i s_{i+1} \cdots s_{n-1}.\pi_1 =
(-1)^{n+i} \pi_1^{(i,4)} + (-1)^{n+i} \pi_1^{(i,5)} + (-1)^{n+i+1} \pi_1^{(i,6)} + (-1)^{n+i+1} \pi_1^{(i,7)},
\end{equation}
where $\pi_1^{(i,4)}, \dots, \pi_1^{(i,7)}$ are the following noncrossing partitions:
\begin{itemize}
\item $\pi_1^{(i,4)}$ has blocks $\{1, 2, \dots, i, n\}, \{i, i+1, \dots, n-2k+3\}, \{n-2k+4, n-2k+5\}, \dots,
\{n-2, n-1\}$,
\item $\pi_1^{(i,5)}$ has blocks $\{1, 2, \dots, i-1\}, \{i, i+1, \dots, n-2k+3, n\}, \{n-2k+4, n-2k+5\}, \dots,
\{n-2, n-1\}$,
\item $\pi_1^{(i,6)}$ has blocks $\{1, 2, \dots, i-1, n\}, \{i, i+1, \dots, n-2k+3\}, \{n-2k+4, n-2k+5\}, \dots,
\{n-2, n-1\}$, and
\item $\pi_1^{(i,7)}$ has blocks $\{1, 2, \dots, i\}, \{i+2, i+2, \dots, n-2k+3, n\}, \{n-2k+4, n-2k+5\}, \dots,
\{n-2, n-1\}$.
\end{itemize}}

We prove Claim 2 by downward induction on $i$.  If $i = n-2k+1$, we apply the operator 
$s_{n-2k+1}$ to the expression on the right hand side of Equation~\ref{claim-one-to-two-equation}.
We have that $B_{n-2k+1}(\pi^{(n-2k+3,1)}) = B_{n-2k+2}(\pi^{(n-2k+3,1)})$, 
so that
$s_{n-2k+1}.(-\pi_1^{(n-2k+3,1)}) = \pi_1^{(n-2k+3,1)}$.  The application of 
$s_{n-2k+1}$ to $-\pi_1'^{(n-2k+3,1)}$ gives rise to the $3$-term expression
\begin{equation}
s_{n-2k+1}.(-\pi_1'^{(n-2k+3,1)}) = -\pi_1^{(n-2k+1,4)} - \pi_1'^{(n-2k+1,4)} + \pi_1^{(n-2k+1,6)},
\end{equation}
where $\pi_1'^{(n-2k+1,4)}$ is the set partition with blocks
$\{1, 2, \dots, n-2k, n-2k+3, n\}, \{n-2k+1, n-2k+2\}, \{n-2k+4, n-2k+5\}, \dots, \{n-2, n-1\}$.
The application of $s_{n-2k+1}$ to $\pi_1''^{(n-2k+3,1)}$ gives rise to the $4$-term expression
\begin{equation}
s_{n-2k+1}.\pi_1''^{(n-2k+3,1)} = \pi_1^{(n-2k+1,7)} + \pi_1'^{(n-2k+1,4)} - \pi_1^{(n-2k+1,5)} - 
\pi_1^{(n-2k+3,1)}.
\end{equation}
Adding these three expressions, we see that 
$s_{n-2k+1} \cdots s_{n-2} s_{n-1}(\pi_1)$ has the form given by Claim 2.  

For the inductive step,
we  apply the operator $s_{i-1}$ to the right hand side of Equation~\ref{claim-two-equation}.  Since 
$B_{i-1}(\pi_1^{(i,4)}) = B_i(\pi_1^{(i,4)})$, we have that 
$s_{i-1}.((-1)^{n+i}\pi_1^{(i,4)}) = (-1)^{n+i-1} \pi_1^{(i,4)}$.  
Since 
$B_{i-1}(\pi^{(i,7)}) = B_i(\pi^{(i,7)})$, 
we have that 
$s_{i-1}.((-1)^{n+i+1}\pi_1^{(i,7)}) = (-1)^{n+i} \pi_1^{(i,7)}$.  The application of 
$s_{i-1}$ to $(-1)^{n+i+1} \pi_1^{(i,6)}$ gives rise to the $4$-term expression
\begin{equation}
s_{i-1}.((-1)^{n+i+1} \pi_1^{(i,6)}) = (-1)^{n+i-1} \pi_1^{(i-1,4)} + (-1)^{n+i-1} \pi_1'^{(i,6)}
+ (-1)^{n+i} \pi_1^{(i-1,6)} + (-1)^{n+i} \pi_1^{(i,4)}, 
\end{equation}
where $\pi_1'^{(i,6)}$ has blocks $\{1, 2, \dots, i-1, i+2, i+3, \dots, n-2k+3, n\}, \{i, i+1\},
\{n-2k+4, n-2k+5\}, \dots, \{n-2, n-1\}$.  The application of $s_{i-1}$ to
$(-1)^{n+i} \pi_1^{(i,5)}$ gives rise to the $4$-term expression
\begin{equation}
s_{i-1}.((-1)^{n+i} \pi_1^{(i,5)}) =
(-1)^{n+i} \pi_1^{(i-1,7)} + (-1)^{n+i} \pi_1'^{(i,6)} + (-1)^{n+i-1} \pi_1^{(i-1,5)} 
+ (-1)^{n+i-1} \pi_1^{(i,7)}.
\end{equation}
Adding these four expressions, we see that 
$s_{i-1} s_i \cdots s_{n-1}.\pi_1$ has the form given by Claim 2.  This completes the proof
of Claim 2.

By Claim 2, we get that 
$s_3 s_4 \cdots s_{n-1}.\pi_1 = 
(-1)^{n+3} \pi_1^{(3,4)} + (-1)^{n+3} \pi_1^{(3,5)} + (-1)^{n+4} \pi_1^{(3,6)} + (-1)^{n+4} \pi_1^{(3,7)}$.
Applying $s_2$ to both sides yields the $3$-term expression
\begin{equation}
\label{after-claim-two}
s_2 s_3 \cdots s_{n-1}.\pi_1 = (-1)^{n+2} \pi_1^{(3,6)} + (-1)^{n+1} r(\pi_1) + 
(-1)^{n+1} \pi_1'^{(3,6)},
\end{equation}
where $\pi_1'^{(3,6)}$ has blocks $\{1,2\}, \{3, 4, \dots, n-2k+3, n\}, \{n-2k+4, n-2k+5\}, \dots, \{n-2, n-1\}$.  Finally,
applying $s_1$ to both sides of Equation~\ref{after-claim-two} yields
$c.\pi_1 = s_1 s_2 \dots s_{n-1}.\pi_1 = (-1)^{n+1} r(\pi_1)$, as claimed in the statement of the lemma.

We now focus on calculating the action of $c$ on $\pi_0$.  This is a shorter argument than the case of $\pi_1$, but slightly more 
conceptually involved.

Since $B_{2k}(\pi_0) = B_{2k+1}(\pi_0) = \cdots = B_n(\pi_0)$, 
we have that
\begin{equation}
\label{pi-two-first}
s_{2k} s_{2k+1} \cdots s_{n-1}.\pi_0 = (-1)^{n-2k-1}\pi_0 = (-1)^{n-1} \pi_0.
\end{equation}
Applying $s_{2k-1}$ to both sides of Equation~\ref{pi-two-first} yields the $3$-term expression
\begin{equation}
\label{pi-two-second}
s_{2k-1} s_{2k} \cdots s_{n-1}.\pi_0 = (-1)^{n-1} \pi_0  + (-1)^{n-1} \pi_0^{(1)} + (-1)^n \pi_0^{(2)},
\end{equation}
where
\begin{itemize}
\item $\pi_0^{(1)}$ has blocks $\{1, 2, 2k+1, 2k+2, \dots, n\}, \{3, 2k-2\}, \{4, 2k-3\}, \dots, \{k, k+1\}, \{2k-1,2k\}$, and
\item $\pi_0^{(2)}$ has blocks $\{1, 2k+1, 2k+2, \dots, n\}, \{2, 2k-1, 2k\}, \{3,2k-2\}, \{4, 2k-3\}, \dots, \{k, k+1\}$.
\end{itemize}
We may inductively assume that the statement of the lemma holds for all smaller values of $n$.  
Observe that the interval $[2, 2k-1]$ is a union of blocks of $\pi_0$ and $\pi_0^{(2)}$ 
and that the interval $[3, 2k-1]$ is a union of blocks of $\pi_0^{(1)}$.
We can now apply the ``local action" Lemma~\ref{local-property} and induction
 to interpret $s_2 s_3 \cdots s_{2k-2}.\pi_0, s_3 s_4 \cdots s_{2k-2}.\pi_0^{(1)},$ and
$s_2 s_3 \cdots s_{2k-2}.\pi_0^{(2)}$ in terms of rotations.

More precisely, observe that the restrictions of $\pi_0$ to $[2, 2k-1]$ or
$\pi_0^{(1)}$ to $[3, 2k-1]$ consist entirely of two-element blocks (i.e., these restrictions 
are noncrossing perfect matchings).
By Lemma~\ref{local-property} and \cite[Lemma 4.3]{PPR} we have that
\begin{equation}
\label{pi-two-third}
\begin{cases}
s_2 s_3 \cdots s_{2k-2}.((-1)^{n-1}\pi_0) = (-1)^{n-1} \pi_0^{(3)}, \\
s_3 s_4 \cdots s_{2k-2}.((-1)^{n-1}\pi_0^{(1)}) = (-1)^n r(\pi_0), 
%s_2 s_3 \cdots s_{2k-2}.((-1)^n\pi_0^{(2)}) = (-1)^n  \pi_0^{(4)},
\end{cases}
\end{equation}
where $\pi_2^{(3)}$ has blocks $\{1, 2k, 2k+1, \dots, n\}, \{2, 3\}, \{4, 2k-1\}, \{5, 2k-2\}, \dots, \{k+1, k+2\}$.

On the other hand, the restriction of $\pi_0^{(2)}$ to $[2, 2k-1]$ consists 
of the blocks $\{2, 2k-1, 2k\}, \{3, 2k-2\}, \{4, 2k-3\}, \dots, \{k, k+1\}$.  
After decrementing these indices by 1, we get a set partition 
$\pi_0$ as in the statement of the lemma corresponding to the smaller values of $n$ and $k$ given by
$(n, k) \rightarrow (2k-1, k-1)$.  Induction and Lemma~\ref{local-property} therefore imply that
\begin{equation}
\label{pi-two-third-prime}
s_2 s_3 \cdots s_{2k-2}.((-1)^n\pi_0^{(2)}) = (-1)^n  \pi_0^{(4)},
\end{equation}
where $\pi_2^{(4)}$ has blocks $\{1, 2k+1, 2k+2, \dots, n\}, \{2, 3, 2k\}, \{4, 2k-1\}, \{5, 2k-2\}, \dots, \{k+1, k+2\}$.

Applying $s_2$ to both sides of the lower equality of Equation~\ref{pi-two-third} gives
\begin{equation}
\label{pi-two-fourth}
s_2 s_3 \cdots s_{2k-2}.((-1)^{n-1}\pi_2^{(1)}) = (-1)^n r(\pi_2) + (-1)^n \pi_2^{(3)} + (-1)^{n-1} \pi_2^{(4)}.
\end{equation}
Summing  Equation~\ref{pi-two-third-prime}, Equation~\ref{pi-two-fourth},
and the upper line of Equation~\ref{pi-two-third} yields
\begin{equation}
\label{pi-two-fifth}
s_2 s_3 \cdots s_{n-1}.(\pi_2) = (-1)^n r(\pi_2).
\end{equation}
Since $B_1(r(\pi_2)) = B_2(r(\pi_2))$, 
we have that
$c.\pi_2 = s_1 s_2 \cdots s_{n-1}(\pi_2) = (-1)^{n+1} r(\pi_2)$, as claimed in the statement of the lemma.
This concludes the proof of the lemma.
\end{proof}

The next result is another technical lemma which will be used to show that 
the skein basis elements $\pi_0$ and $\pi_1$ in Figure~\ref{figure-pi}
are not both isolated in the same isotypic component of the restriction of $V(n, k, 0)$ from $\symm_n$ to
$\symm_{n-1}$.

\begin{lemma}
\label{nonvanishing-lemma}
Assume $n > 2k$ and 
let $\pi_0, \pi_1 \in NC(n, k, 0)$ be the skein basis elements shown in Figure~\ref{figure-pi}.
Let $\mu = (k, k-1, 1^{n-2k}) \vdash n-1$ and $\nu = (k, k, 1^{n-2k-1}) \vdash n-1$.
Then $[\symm_{\mu'}]_-.\pi_1 \neq 0$ and $[\symm_{\nu}]_+.\pi_0 \neq 0$.
\end{lemma}

\begin{proof}
In the group algebra $\QQ \symm_{n-1}$, the element $[\symm_{\mu'}]_-$ factors as
\begin{equation}
[\symm_{\mu'}]_- = [\symm_{n-2k+2}]_- (1-s_{n-2k+4}) \cdots (1-s_{n-4}) (1-s_{n-2}).
\end{equation}
By examining the blocks of $\pi_1$, we see that $(1-s_{n-2i}).\pi_1 = 2 \pi_1$ for $1 \leq i \leq k-2$
and $[\symm_{n-2k+2}]_-.\pi_1 = (n-2k+2)! \pi_1$.  Therefore, we get that
$[\symm_{\mu'}]_-.\pi_1 = 2^{k-2} (n-2k+2)!  \pi_1 \neq 0$.

Let $\lambda = (k, k, 1^{n-2k}) \vdash n$.  Then $\symm_{\nu}$ is naturally a subgroup of $\symm_{\lambda}$,
so that $[\symm_{\nu}]_+$ is a factor in the group algebra $\QQ \symm_n$ of 
$[\symm_{\lambda}]_+$.  It therefore suffices to show that 
$[\symm_{\lambda}]_+.\pi_0 \neq 0$.  This argument appears in the first four paragraphs of the proof of
Proposition~\ref{isomorphism-special-case}.
\end{proof}

We are in a good place to show that the long cycle in $\symm_n$ acts as plus or minus rotation on the skein basis of $V(n, k, 0)$.

\begin{theorem}
\label{long-cycle-action}
Let $n, k > 0$ and 
consider the representation $V(n, k, 0)$ with its skein basis $NC(n, k, 0)$.
Let $c = (1, 2, \dots, n) \in \symm_n$ be the long cycle.

For any skein basis element $\pi \in NC(n, k, 0)$, we have that 
\begin{equation}
c.\pi = (-1)^{n+1} r(\pi),
\end{equation}
where $r: NC(n, k, 0) \rightarrow NC(n, k, 0)$ is clockwise rotation.
\end{theorem}

\begin{proof}
We extend scalars to $\CC$ for this proof.
The case $n = 2k$ is proven in \cite[Lemma 4.3]{PPR}, so we assume $n > 2k$.

By Proposition~\ref{isomorphism-special-case}, we have that
$V(n, k, 0) \cong_{\symm_n} S^{\lambda}$, where $\lambda = (k, k, 1^{n-2k}) \vdash n$ is a flag shaped partition.
By the branching rule for symmetric groups, we have that the restriction
$\Res^{\symm_n}_{\symm_{n-1}} V(n, k, 0)$ is isomorphic to a direct sum of two nonisomorphic
$\symm_{n-1}$-irreducibles.  Namely, we have that 
$\Res^{\symm_n}_{\symm_{n-1}} V(n, k, 0) \cong_{\symm_{n-1}} S^{\mu} \oplus S^{\nu}$,
where $\mu = (k, k-1, 1^{n-2k}) \vdash n-1$ and $\nu = (k, k, 1^{n-2k-1}) \vdash n-1$.  The $\CC$-algebra of 
$\symm_{n-1}$-endomorphisms of $\Res^{\symm_n}_{\symm_{n-1}} V(n, k, 0)$ is therefore
isomorphic to $\CC^2$.

Let $R: V(n, k, 0) \rightarrow V(n, k, 0)$ be the linear isomorphism induced by the rotation map 
$r: NC(n, k, 0) \rightarrow NC(n, k, 0)$ on the skein basis.  
For any $2 \leq i \leq n-1$,  we have that
$R^{-1} \circ s_i = s_{i-1} \circ R^{-1}$ 
as operators on $V(n, k, 0)$.  On the other hand, we have that 
$c \circ s_{i-1} = s_i \circ c$ as operators on $V(n, k, 0)$.  Putting this together, we have that
$(R^{-1} \circ  c)  \circ s_i = s_i  \circ (R^{-1} \circ  c)$ as operators on $V(n, k, 0)$.  Since
the adjacent transpositions
$s_2, s_3, \dots, s_{n-1}$ generate $\symm_{\{2, 3, \dots, n\}} = \symm_{n-1}$, 
this means that 
the operator $R^{-1} \circ c$ is an $\symm_{n-1}$-equivariant linear operator on $V(n, k, 0)$.
Therefore,
there exist two scalars 
$\gamma_1, \gamma_2 \in \CC$ such that 
$R^{-1} \circ c$ acts as $\gamma_1$ on the copy of $S^{\mu}$ inside $\Res^{\symm_n}_{\symm_{n-1}} V(n, k, 0)$
and $R^{-1} \circ c$ acts as $\gamma_2$ on the copy of $S^{\nu}$ inside
$\Res^{\symm_n}_{\symm_{n-1}} V(n, k, 0)$.  We want to show that 
$\gamma_1 = \gamma_2 = (-1)^{n+1}$, so that 
$c$ acts as $(-1)^{n+1} R$ on all of $\Res^{\symm_n}_{\symm_{n-1}}V(n, k, 0)$,
and therefore also on
$V(n, k, 0)$.

To show that $\gamma_1 = \gamma_2 = (-1)^{n+1}$, let $\pi_0, \pi_1 \in NC(n, k, 0)$ be the 
skein basis elements defined in Figure~\ref{figure-pi}.
Lemma~\ref{sign-determination-lemma} implies that
$R^{-1} \circ c(\pi_0) = (-1)^{n+1} \pi_0$ and $R^{-1} \circ c(\pi_1) = (-1)^{n+1} \pi_1$.
It is enough to show that the $2$-dimensional space $\mathrm{span}\{\pi_0, \pi_1\}$ 
is not contained in either isotopic component of
$\Res^{\symm_n}_{\symm_{n-1}} V(n, k, 0) = S^{\mu} \oplus S^{\nu}$.  Indeed, observe that $\mu \prec \nu$.  By Lemmas~\ref{pincer-lemma} and \ref{nonvanishing-lemma},
we know that $\pi_0 \notin S^{\mu}$ and $\pi_1 \notin S^{\nu}$.  It follows that $\mathrm{span}\{\pi_1, \pi_2\}$ 
is contained in neither isotypic component of $\Res^{\symm_n}_{\symm_{n-1}} V(n, k, 0)$.  
We conclude that $\gamma_1 = \gamma_2 = (-1)^{n+1}$, finishing the proof of the theorem.
\end{proof}

It is worth comparing the proof of Theorem~\ref{long-cycle-action} to the proofs of similar results appearing in the literature.
\begin{enumerate}
\item
In Stembridge's proof \cite{Stembridge}
 that the long element $w_0 \in \symm_n$ acts as plus or minus evacuation on the KL basis of an arbitrary 
$\symm_n$-irreducible $S^{\lambda}$ or in the proof of Proposition~\ref{long-element-action} that the long element
$w_0 \in \symm_n$ acts as plus or minus reflection on the skein basis, one proves directly that 
$T^{-1} \circ w_0$ commutes with the action of $\symm_n$ on the desired irreducible module, where $T$ is the linear operator
induced by the candidate combinatorial action on the relevant basis.  Schur's Lemma and a single scalar determination complete the proof.
\item
In the proof that the long cycle $c \in \symm_n$ acts as plus or minus promotion on the KL basis of a rectangular $\symm_n$-irreducible
$S^{\lambda}$ \cite{RhoadesJDT}, 
one first proves the weaker result that $J^{-1} \circ c$ commutes with the action of a smaller symmetric group 
$\symm_{n-1}$, where $J$ is the linear operator on $S^{\lambda}$ induced by promotion on the KL basis.  Since the restriction
of a rectangular $\symm_n$-module to $\symm_{n-1}$ remains irreducible, Schur's Lemma and a single scalar determination still suffice
to complete the proof.
\item
In the proof of Theorem~\ref{long-cycle-action} that $c \in \symm_n$ acts as plus or minus rotation on the skein basis of a flag shaped
$\symm_n$-irreducible $S^{\lambda}$, one again starts by proving that $R^{-1} \circ c$ commutes with the action of the smaller
symmetric group $\symm_{n-1}$, where $R$ is the linear operator on $S^{\lambda}$ induced by rotation on the skein basis.
However, the restriction of $S^{\lambda}$ from $\symm_n$ to $\symm_{n-1}$ is no longer irreducible.  Since $\lambda$ has two
outer corners, we need to determine {\it two} scalars to complete the proof with a Schur's Lemma argument.
\end{enumerate}
To the author's knowledge, Theorem~\ref{long-cycle-action} is the first instance of a representation theoretic result underlying a 
CSP involving a cyclic action of order $> 2$ which concerns a non-rectangular $\symm_n$-module.  
Theorem~\ref{long-cycle-action} breaks the rectangular barrier to include flag shaped irreducibles in the cyclic sieving family.
The progression illustrated here suggests that including more complicated irreducibles (corresponding to partitions $\lambda$
with more outer corners) will involve still more intricate sign determinations in the style of Lemma~\ref{sign-determination-lemma}.

\begin{example}
\label{six-example}
Let us consider the action of $c \in \symm_6$ on the skein basis element $\pi = \{ \{1, 2, 3\}, \{4, 5, 6\} \} \in V(6,2,0)$.  
Computing $c.\pi = s_1 s_2 s_3 s_4 s_5.\pi$ gives the following, where we indicate cancellations with lines of the same slope.

\vspace{0.1in}
\includegraphics[scale = 0.35]{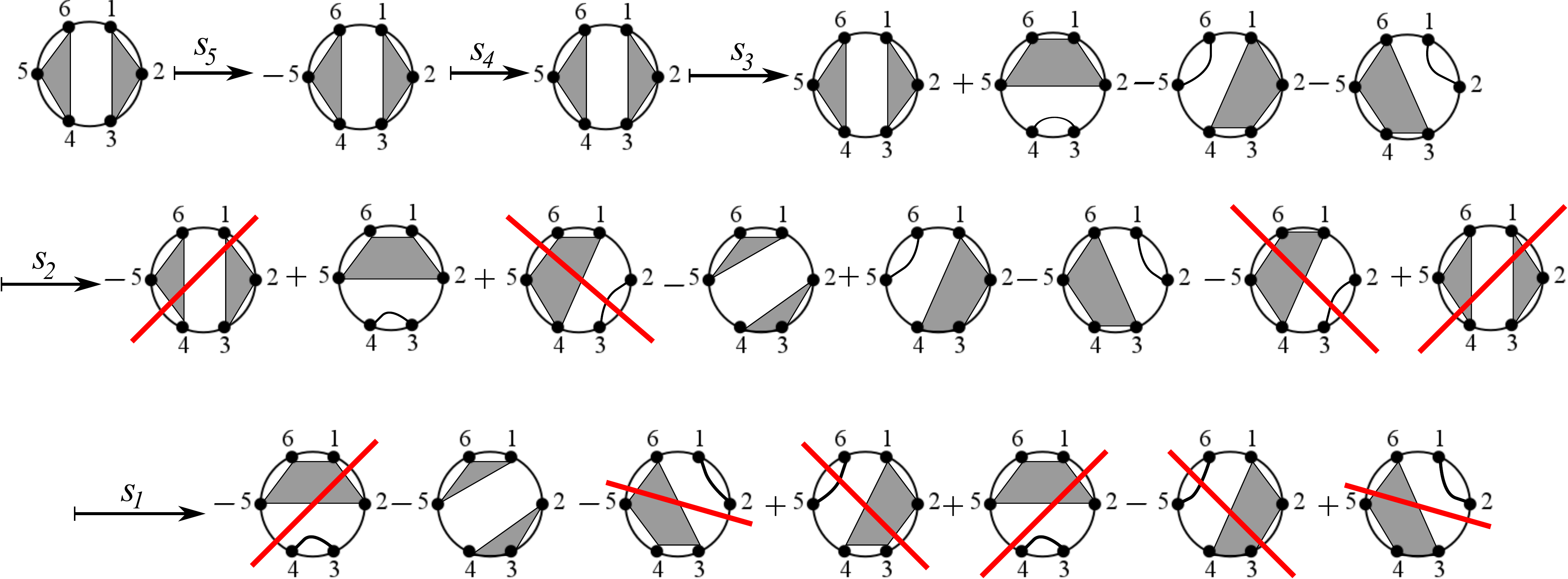}

As predicted by Theorem~\ref{long-cycle-action}, we get that $c.\pi = (-1)^{6+1} r(\pi)$.
\end{example}

\begin{remark}
Let $\lambda = (2, 2, 1, 1) \vdash 6$.  It can be shown that none of the representing matrices for permutations 
in $\symm_6$ acting on $S^{\lambda}$ with respect to the KL cellular basis is a monomial matrix corresponding 
to rotation on $NC(6, 2, 0)$.  Thus, Theorem~\ref{long-cycle-action} shows that 
the skein basis and the KL cellular basis are {\em different} in general.
For the cyclic sieving purposes of this paper, the skein basis is more advantageous.

On the other hand, the skein and KL cellular bases coincide for the special case of the 
$\symm_{2n}$-module
$V(2n, n, 0) \cong S^{(n, n)}$  spanned by noncrossing perfect matchings of $[2n]$.
\end{remark}

As a  corollary to Theorem~\ref{long-cycle-action}, we can describe action of the
``affine transposition" $s_n := (1, n) \in \symm_n$ on  the skein basis of $V(n, k, 0)$.
The relation $s_n = c s_{n-1} c^{-1}$ gives the following formula.

\begin{corollary}
\label{affine-transposition-action}
Let $\pi \in NC(n, k, 0)$ be a skein basis element of $V(n, k, 0)$.  We have that 
\begin{equation}
s_n.\pi = \begin{cases}
s_n(\pi) & \text{if $B_1(\pi) \neq B_n(\pi)$ and $s_n(\pi)$ is noncrossing,} \\
-\pi & \text{if $B_1(\pi) = B_n(\pi)$,} \\
r^{-1}(\sigma(r(\pi))) & \text{if $s_n(\pi)$ is crossing,}
\end{cases}
\end{equation}
where $r: NC(n, k, 0) \rightarrow NC(n, k, 0)$ is clockwise rotation.
\end{corollary}

This section has focused on the actions of $w_0, c,$ and $s_n$ on the skein basis of $V(n, k, 0)$.
The actions of $w_0, c,$ and $s_n$ on the skein basis of $V(n)$ (or on its submodules $V(n, k), V(n, k, s)$) is easily obtainable by
combining Theorem~\ref{isomorphism} with
Theorem~\ref{long-element-action}, Theorem~\ref{long-cycle-action}, and
Corollary~\ref{affine-transposition-action}.  
We have that $w_0$ acts as plus or minus reflection and 
$c$ acts as plus or minus clockwise rotation on these skein bases.
In the next section we will reprove and generalize this result.

\section{Resolving crossings in set partitions}
\label{Resolving}

Recall that $P(n)$ is the $\QQ$-vector space spanned by the set $\Pi(n)$ of partitions of $[n]$, equipped with the
$\star$ action of $\symm_n$ given in Lemma~\ref{rhos-satisfy-coxeter}.
In Section~\ref{Almost} we defined a function $\sigma: ANC(n) \rightarrow V(n)$ which expressed any almost noncrossing partition
as a linear combination of noncrossing partitions.
The purpose of this section is to generalize $\sigma$ to a projection $p: P(n) \twoheadrightarrow V(n)$.   Given
any set partition $\pi \in \Pi(n)$, we can think of $p(\pi)$ as a ``resolution" of the crossings of $\pi$.

The $\QQ$-linear map $p: P(n) \twoheadrightarrow V(n)$
will be characterized by following three properties.
\begin{enumerate}
\item For any noncrossing partition $\pi \in NC(n)$ we have $p(\pi) = \pi$, so that $p$ is a linear projection onto $V(n)$.
\item For any almost noncrossing partition $\pi \in ANC(n)$ we have $p(\pi) = -\sigma(\pi)$.
\item The map $p$ is $\symm_n$-equivariant: for any $w \in \symm_n$ and $\pi \in \Pi(n)$ we have
$p(w \star \pi) = w . p(\pi)$.
\end{enumerate}
The minus sign in Property 2 appears so that Properties 1 and 3 can be achieved.  
Due to the ``partial minus sign" appearing in the $\tau_i$ operators in Section~\ref{Symmetric},
we are unsure how to alter our setup to get the more natural $p(\pi) = \sigma(\pi)$.
Since every set partition $\pi \in \Pi(n)$ is $\symm_n$-conjugate to a noncrossing partition and almost noncrossing partitions
are the images of noncrossing partitions under the generating set $\{s_1, \dots, s_{n-1}\}$ of $\symm_n$, 
 there is at most one $\QQ$-linear map $p: P(n) \twoheadrightarrow V(n)$ satisfying 
Properties 1-3.  We need to show that such a map $p$ exists.

It will be useful to identify canonical noncrossing orbit representatives for the action of $\symm_n$
on $\Pi(n)$.  Given any partition $\lambda = (\lambda_1, \lambda_2, \dots, \lambda_k) \vdash n$, let 
$\pi_{\lambda} \in NC(n)$ denote the noncrossing set partition
\begin{equation*}
\pi_{\lambda} := \{ \{1, 2, \dots, \lambda_1\}, \{\lambda_1+1, \lambda_1+2, \dots, \lambda_2\}, \dots,
\{n-\lambda_k+1, \dots, n-1, n\} \}.
\end{equation*}
For example, we have that 
$\pi_{(4,2,2,1)} = \{ \{1, 2, 3, 4\}, \{5, 6\}, \{7, 8\}, \{9\} \}$.  Any set partition in $\Pi(n)$ is conjugate
under the action of $\symm_n$ to a unique partition of the form $\pi_{\lambda}$ for 
some $\lambda \vdash n$.  

Our candidate for $p: P(n) \twoheadrightarrow V(n)$ is defined as follows.
Recall that $\star$ and $.$ denote the actions of $\symm_n$ on the modules
$P(n)$ and $V(n)$, respectively.
Given $\pi \in \Pi(n)$, there exists $w \in \symm_n$ such that $w(\pi)$ is a noncrossing set partition of the form $\pi_{\lambda}$
for some $\lambda \vdash n$.
Then $w \star \pi = \pm w(\pi) \in V(n)$, so that $w^{-1}.(w \star \pi)$ is an element of $V(n)$.
We set 
\begin{equation}
\label{equation-for-p}
p(\pi) := w^{-1}.(w \star \pi).
\end{equation}
At this stage, it is not clear that $p(\pi)$ is well defined (i.e., independent of the choice of $w$); we will prove this 
in Theorem~\ref{projection-corollary}.
We extend $p$ to a $\QQ$-linear map $p: P(n) \rightarrow V(n)$.

Explicitly, to compute $p(\pi)$, one finds a permutation $w \in \symm_n$ such that the set partition
$w(\pi)$ is of the form $\pi_{\lambda}$.  Then one writes $w = s_{i_1} \cdots s_{i_k}$ as a product of adjacent transpositions and acts,
within the module $V(n)$, by the composition $s_{i_k} \cdots s_{i_1}$ on 
$w \star \pi = \pm w(\pi)$.  
Assuming that the right hand side of  Equation~\ref{equation-for-p} is independent of $w$,
the $\symm_n$-equivariance of $p$  in Property 3 comes from the calculation
$p(v \star \pi) = (wv^{-1})^{-1}.((wv^{-1}) \star (v \star \pi)) = v.w^{-1}.(w \star \pi) = v.p(\pi)$,
where $\pi \in \Pi(n)$, $w(\pi)$ is of the form $\pi_{\lambda}$, and $v \in \symm_n$. 
It remains to show that $p$ is well defined, a projection onto $V(n)$ (Property 1), and agrees with $-\sigma$ on almost noncrossing partitions (Property 2).

To show that $p$ is well defined, we will need to analyze the $\symm_n$-stabilizer of set partitions of the form $\pi_{\lambda}$.
Given $\lambda \vdash n$,
we let $\Stab(\pi_{\lambda}) = \{w \in \symm_n \,:\, w(\pi_{\lambda}) = \pi_{\lambda} \}$ denote the group 
of permutations which stabilize $\pi_{\lambda}$.  The group
$\Stab(\pi_{\lambda})$ can be expressed as a direct product of wreath products of smaller symmetric groups,
with one wreath product for each part multiplicity in $\lambda$.  
While we will not need this explicit decomposition, we will use the fact that $\Stab(\pi_{\lambda})$ is generated by two
kinds of permutations:
\begin{enumerate}
\item adjacent transpositions  $s_i$ such that $i$ and $i+1$ belong to the same block of $\pi_{\lambda}$, and
\item permutations $w \in \symm_n$ which interchange two consecutive blocks of $\pi_{\lambda}$ of the same size while preserving the 
relative order of the elements in each block.
\end{enumerate}
For example, if $\lambda = (4, 2, 2, 1) \vdash 9$, then 
$\Stab(\pi_{\lambda}) \subset \symm_9$ is generated by the adjacent transpositions
$s_1, s_2, s_3, s_5,$ and $s_7$, together with the permutation $(5,7)(6,8)$.  

To show that our candidate for
$p: P(n) \twoheadrightarrow V(n)$ is well defined, we will need to analyze the action of $\Stab(\pi_{\lambda})$ on the  
noncrossing partition
$\pi_{\lambda}$ inside the modules $P(n)$ and $V(n)$. 
Namely, we need to show that these actions coincide.

\begin{lemma}
\label{stabilizer-lemma}
Let $\lambda \vdash n$ and let $w \in \Stab(\pi_{\lambda})$.  We have that 
\begin{equation}
w \star \pi_{\lambda} = w.\pi_{\lambda},
\end{equation}
where the action on the left is inside $P(n)$ and the action on the right is inside $V(n)$.
\end{lemma}

\begin{proof}
It is enough to check the lemma when $w$ is a generator of $\Stab(\pi_{\lambda})$.  If $i$ and $i+1$ are in the same block of
$\pi_{\lambda}$, then
$s_i \star \pi_{\lambda} = - \pi_{\lambda} = s_i.\pi_{\lambda}$, as desired.  

Suppose that two consecutive blocks of $\pi_{\lambda}$ have the same size.  Write these blocks as
$\{i, i+1, \dots, i+d-1\}$ and $\{i+d, i+d+1, \dots, i+2d-1\}$ for some $i$ and $d$.  Let $w \in \symm_n$ be the permutation 
whose cycle notation is $w = (i, i+d)(i+1, i+d+1) \cdots (i+d-1, i+2d-1)$.   In the module $P(n)$ we have that
\begin{equation}
w \star \pi_{\lambda} = 
\begin{cases}
\pi_{\lambda} & \text{if $d = 1$} \\
(-1)^d \pi_{\lambda} & \text{if $d > 1$.}
\end{cases}
\end{equation} 
We are reduced to showing that $w.\pi_{\lambda}$ is given by the right hand side above in the module $V(n)$.
If $d = 1$ this is clear, so we assume that $d > 1$.

Consider the interval $[i, i+2d-1]$.
Since $w$ fixes indices outside this interval and this interval is a union of blocks of $\pi_{\lambda}$,
Lemma~\ref{local-property} implies that $w.\pi_{\lambda}$ can be calculated by acting on 
$\{ \{1, 2, \dots, d\}, \{d+1, d+2, \dots, 2d\} \} = \pi_{(d,d)}$ by $c_{2d}^d = (1,d+1)(2,d+2) \cdots (d,2d)$
(where $c_{2d} = (1, 2, \dots, 2d) \in \symm_{2d}$ is the long cycle), incrementing all indices by $i-1$,
and adding on the blocks of $\pi_{\lambda}$ lying outside the interval $[i, i+2d-1]$.
By Theorem~\ref{long-cycle-action},  we know that 
$c_{2d}^d.\pi_{(d,d)} = (-1)^{d(2d+1)} r^d (\pi_{(d,d)}) = (-1)^{d(2d+1)} \pi_{(d,d)} = (-1)^d \pi_{(d,d)}$.
\end{proof}

\begin{theorem}
\label{projection-corollary}
Equation~\ref{equation-for-p} gives a well defined projection of $\QQ \symm_n$-modules $p: P(n) \twoheadrightarrow V(n)$.
In particular, we have that $p(\pi) = \pi$ for any noncrossing partition $\pi \in NC(n)$.
Moreover,
we have that $p(\pi) = -\sigma(\pi)$ for any almost noncrossing partition $\pi \in ANC(n)$.
\end{theorem}

\begin{proof}
We check that $p$ is well defined.  Let $\pi \in \Pi(n)$.  There exists a unique $\lambda \vdash n$ such that 
$\pi$ is in the $\symm_n$-orbit of $\pi_{\lambda}$.  Suppose that $w(\pi) = v(\pi) = \pi_{\lambda}$ for some $w, v \in \symm_n$.
We need to show that 
$w^{-1}.(w \star \pi) = v^{-1}.(v \star \pi)$ inside $V(n)$.  Since 
$v w^{-1}$ is contained in $\Stab(\pi_{\lambda})$, Lemma~\ref{stabilizer-lemma} implies that
$(v w^{-1}).(w \star \pi) = (v w^{-1}) \star (w \star \pi) = v \star \pi$.  Acting on this equation within $V(n)$ by $v^{-1}$ yields the desired equality.
It follows that $p$ is a well defined $\QQ$-linear map $P(n) \rightarrow V(n)$.  The remarks following Equation~\ref{equation-for-p}
imply that $p$ is also $\symm_n$-equivariant.

If $\lambda \vdash n$, we know that $p(\pi_{\lambda}) = \pi_{\lambda}$ by applying Equation~\ref{equation-for-p}
with $w = e$, the identity permutation.  If $\pi \in NC(n)$ is a general noncrossing partition, there exists a unique
$\lambda \vdash n$ be such that $\pi = w(\pi_{\lambda})$ for some $w \in \symm_n$.  
We have that 
$p(w \star \pi_{\lambda}) = w.p(\pi_{\lambda}) = w.\pi_{\lambda}$.  
Since $w \star \pi_{\lambda} = \pm \pi$, to show that $p(\pi) = \pi$ it suffices to prove the following claim.

\noindent
{\bf Claim:} {\it Let $\pi \in NC(n)$ and let $\lambda \vdash n$ be the unique partition such that $\pi$ is $\symm_n$-conjugate to $\pi_{\lambda}$.
There exists a permutation $w \in \symm_n$ such that $w(\pi_{\lambda}) = \pi$ and
$w.\pi_{\lambda} = w \star \pi_{\lambda}$.}

The proof of this claim is similar to the proof of Lemma~\ref{stabilizer-lemma}.  The key observation is that $\pi_{\lambda}$ can be ``wound up" into
$\pi$ by a sequence of moves $\pi_{\lambda} = \pi_1, w_1(\pi_1) = \pi_2, w_2(\pi_2) = \pi_3, \dots, w_{N-1}(\pi_{N-1}) = \pi_N = \pi$,
where 
\begin{itemize}
\item
each permutation $w_j \in \symm_n$ is of the form $w_j = (i_j, i_j+1, \dots, i_j + k_j)$, 
\item the set partition
$\pi_j$ is noncrossing for all $1 \leq j \leq N$, and 
\item
the interval
$[i_j, i_j + k_j]$ is a union of blocks of $\pi_j$ for all $1 \leq j \leq N$.
\end{itemize}

To give an example of this process, suppose that $\pi$ is the noncrossing partition of $[8]$
given by $\pi = \{ \{1, 2, 8\}, \{3, 4, 5, 7\}, \{6\} \}$.  Then $\lambda = (4,3,1)$ and $\pi$ is $\symm_8$-conjugate to the set partition
$\pi_{\lambda} = \{ \{1, 2, 3, 4\}, \{5, 6, 7\}, \{8\} \}$.  Let us consider how to go from $\pi_{\lambda}$ to $\pi$ using moves described in the last paragraph.
Applying first $(5,6,7,8)$, then $(1,2,3,4,5)$ to $\pi_{\lambda}$
yields the partition $\pi' = \{ \{1\}, \{2, 3, 4, 5\}, \{6, 7, 8\} \}$.
Applying $(1,2,3,4,5,6,7,8)^2$ to $\pi'$ yields
$\pi'' = \{ \{1, 2, 8\}, \{3\}, \{4,5,6,7\} \}$.
Finally, applying 
$(3,4,5,6,7)^3$ to $\pi''$ yields $\{ \{1, 2, 8\}, \{3, 4, 5, 7\}, \{6\} \} = \pi$.

Suppose we have a sequence
 $\pi_{\lambda} = \pi_1, w_1(\pi_1) = \pi_2, w_2(\pi_2) = \pi_3, \dots, w_{N-1}(\pi_{N-1}) = \pi_N = \pi$
 as described above.
We can use 
Lemma~\ref{local-property} and Theorem~\ref{long-cycle-action} to get that
$w_j \star \pi_j = w_j.\pi_j$ for all $1 \leq j \leq N-1$.  Therefore, we have that
$(w_{N-1} \cdots w_1) \star \pi_{\lambda} = (w_{N-1} \cdots w_1).\pi_{\lambda}$, so that the Claim holds for 
$w = w_{N-1} \cdots w_1$. 
This finishes the proof of the Claim.

To complete the proof of the theorem, let $\pi \in ANC(n)$ be an almost noncrossing partition which crosses at $i$.
Then $s_i(\pi)$ is noncrossing, so that 
$p(\pi) = s_i.(s_i \star \pi) = -s_i.(s_i(\pi)) = - \sigma(\pi)$.
\end{proof}

The fact that $p: P(n) \twoheadrightarrow V(n)$ is a map of $\symm_n$-modules means that for any $w \in \symm_n$, we have the following 
commutative square.

\begin{center}
\includegraphics[scale = 0.4]{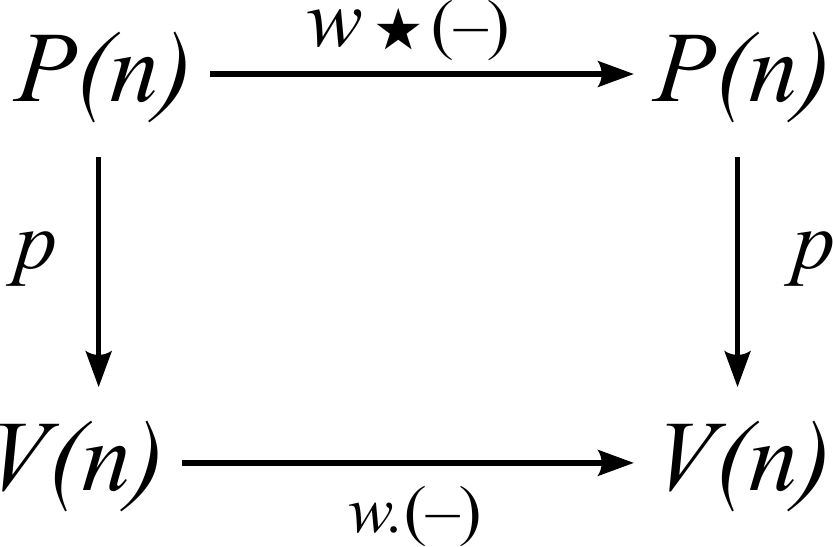}
\end{center}

The commutativity of this square immediately gives the following corollary, which roughly states that 
the action of $\symm_n$ on the skein basis of $V(n)$ preserves any ``local symmetry" of a specific noncrossing
partition $\pi \in NC(n)$ (not just a ``global symmetry" of $NC(n)$ itself).

\begin{corollary}
\label{symmetry-corollary}
Let $\pi$ be a noncrossing partition of $[n]$ and suppose that $w \in \symm_n$ is such that $w(\pi)$ is also noncrossing.
We have that $w.\pi = w \star \pi$, so that in particular $w.\pi = \pm \pi$.
\end{corollary}

Ignoring signs, the action of $w_0$ on $V(n)$ given in Theorem~\ref{long-element-action} and of
$c$ on $V(n)$ given in Theorem~\ref{long-cycle-action} are special cases of Corollary~\ref{symmetry-corollary}.
In fact, the signs involved in the actions of $w_0$ and $c$ on the skein basis of $V(n)$ are easy to deduce from the signs
involved in the $\star$ actions $w_0 \star \pi$ and $c \star \pi$ for $\pi \in NC(n)$.
However, the proof of Corollary~\ref{symmetry-corollary} relies heavily on our knowledge of the action of 
$c$ on the skein basis of $V(n)$.  It would be interesting to prove Corollary~\ref{symmetry-corollary} directly, without use
of Theorem~\ref{long-cycle-action}.

As another corollary to Theorem~\ref{projection-corollary}, we get an easier way to compute the image $p(\pi)$ of a set partition
$\pi \in P(n)$ under the map $p: P(n) \twoheadrightarrow V(n)$: find any $w \in \symm_n$ such that 
$w(\pi)$ is noncrossing and compute $p(\pi) = w^{-1}.(w \star \pi)$.
When $w$ is chosen to have small Coxeter length,
the following example illustrates that this process can be substantially more efficient than using
Equation~\ref{equation-for-p} to compute $p(\pi)$.

\begin{example}
Let $\pi$ be the set partition of $[8]$ given by $\pi = \{ \{1, 4, 8\}, \{2, 3, 5, 7\}, \{6\} \}$, which is not noncrossing.  Let us compute the
noncrossing resolution
$p(\pi) \in V(8)$.

We have that $w(\pi) = \{ \{1, 2, 8\}, \{3, 4, 5, 7\}, \{6\} \}$ is noncrossing, where $w = s_2 s_3 \in \symm_8$.
Moreover, we have that $w \star \pi = s_2 s_3 \star \pi = - s_2 \star s_3(\pi) =  s_2 s_3(\pi) = w(\pi)$ by the definition of the $\star$ action.
To compute $p(\pi)$, we apply the composition $s_3 s_2$ to the skein basis element $w(\pi)$ within the module $V(8)$.
This process is shown below.

\begin{center}
\includegraphics[scale = 0.3]{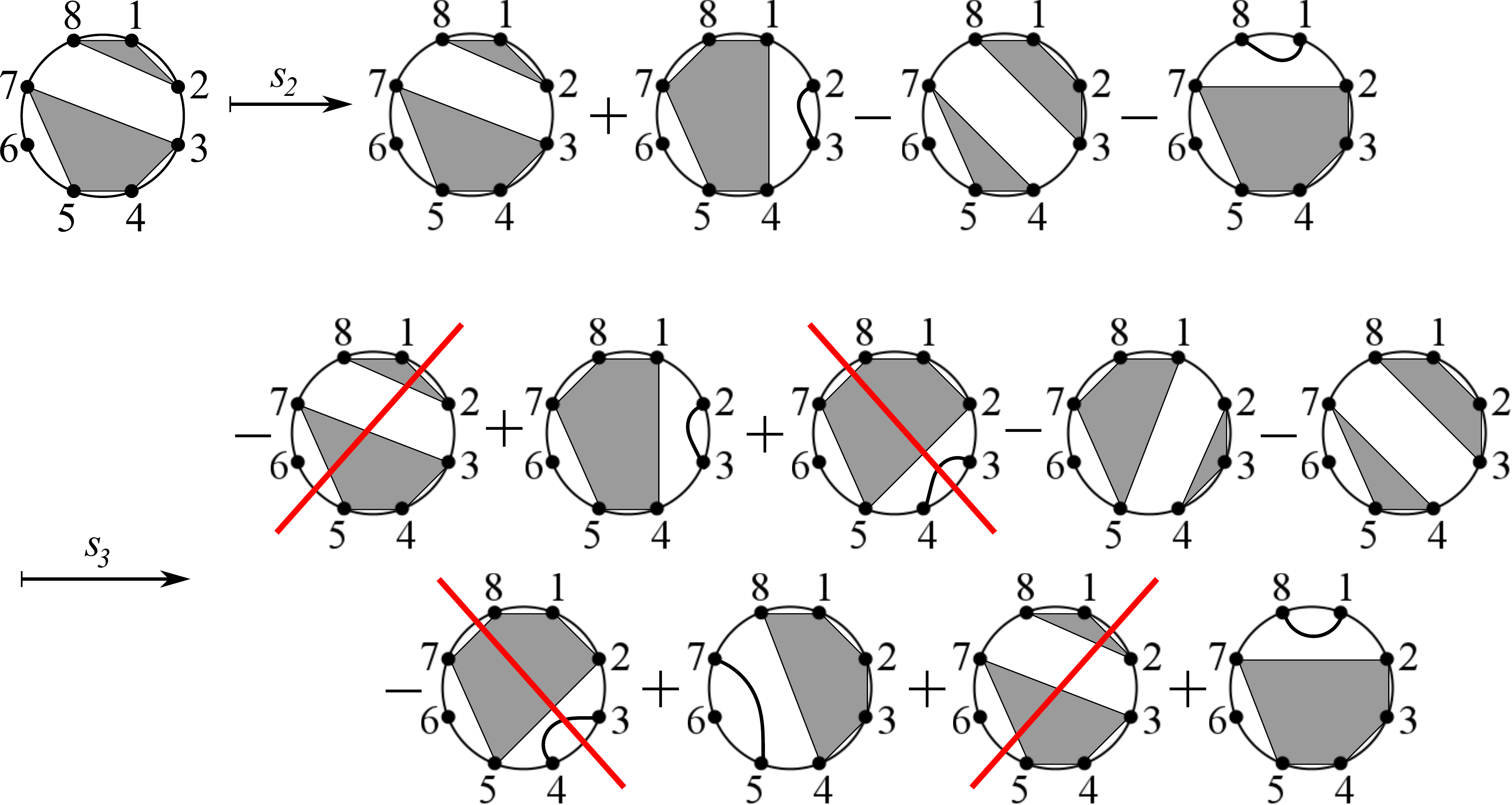}
\end{center}

We conclude that $p(\pi)$ is the expression given on the first page of this paper.

On the other hand, computing $p(\pi)$ using Equation~\ref{equation-for-p} would involve noticing that 
$\pi$ is $\symm_8$-conjugate to $\pi_{(4,3,1)} = \{ \{1, 2, 3, 4\}, \{5, 6, 7\}, \{8\} \}$.  The unique permutation $u$ of minimal Coxeter length
such that $u(\pi) = \pi_{\lambda}$ has one-line notation $u = 51263847$, so that $w$ can be expressed in Coxeter generators as
$u = s_4 s_3 s_2 s_1 s_5 s_4 s_7 s_6$.  Computing $p(\pi) = u^{-1}.(u \star \pi)$ using this method would involve acting on the skein basis by eight simple 
reflections instead of two.  We leave it to the reader to check that one gets the same answer.
\end{example}

Our method of computing $p(\pi)$ for a set partition $\pi$ is recursive.  It may be interesting 
to determine the expansion $p(\pi) \in V(n)$ explicitly.

\begin{problem}
Let $\pi$ be a set partition of $[n]$.  Give a non-recursive formula for $p(\pi)$.
\end{problem}

\section{Applications to cyclic sieving}
\label{Applications}

In this section we explain how to apply Proposition~\ref{long-element-action} and
Theorem~\ref{long-cycle-action} to get cyclic sieving results for the action of reflection and
rotation on various classes of noncrossing partitions.  

In order to equate irreducible character evaluations for the symmetric group with root of unity evaluations,
we will make use of the following theorem.  This result has appeared many times in  algebraic
proofs of cyclic sieving phenomena.

\begin{theorem}
\label{springer} (Springer's Theorem \cite{Springer})
Let $\lambda \vdash n$ and let $w \in \symm_n$ be a power of an $n$-cycle or an $(n-1)$-cycle.  
Let $\zeta \in \CC$ be a root of unity with order equal to the order of $w \in \symm_n$.

The irreducible character evaluation $\chi^{\lambda}(w)$ is given by
\begin{equation}
\chi^{\lambda}(w) = f^{\lambda}(\zeta),
\end{equation}
where $f^{\lambda}(q) \in \NN[q]$ is the fake degree polynomial for the $\lambda$-isotypic
component of the coinvariant algebra attached to $\symm_n$.
\end{theorem}

Springer's Theorem holds in  greater generality than presented here.  It can be used to calculate
any irreducible character $\chi$ of any  complex reflection group $W$ evaluated at a regular element 
$w \in W$
as a root of unity evaluation of the Hilbert polynomial of the $\chi$-isotypic component of the 
associated coinvariant algebra.  Theorem~\ref{springer} is the special case of Springer's Theorem
for $W = \symm_n$.
The polynomial $f^{\lambda}(q)$ is  $q^{b(\lambda)}$ times the $q$-analog 
$\prod_{c \in \lambda} \frac{[n]!_q}{[h_c]_q}$ of the hook length formula corresponding to $\lambda$.

We begin by using Theorem~\ref{long-cycle-action} to give an algebraic proof of 
Pechenik's CSP given in Theorem~\ref{no-singletons-csp}.  We begin by recalling its statement. \\

\noindent{\bf Theorem 1.3.} {\it (Pechenik \cite[Theorem 1.4]{Pechenik})  Let $n, k > 0$ be such that  $n \geq 2k$.  Let $X$ be the 
set of noncrossing partitions of $[n]$ with exactly $k$ blocks and no singleton blocks and let $C = \ZZ_n$
act on $X$ by rotation.  

The triple $(X, C, X(q))$ exhibits the cyclic sieving phenomenon, where
\begin{equation*}
X(q) = \frac{[k]_q}{[n-k]_q[n-k+1]_q} \times \frac{[n]!_q}{([k]!_q)^2 [n-2k]!_q}
\end{equation*}}

\begin{proof}
Let $\lambda = (k, k, 1^{n-2k})$ be the flag shaped partition in the statement of the theorem.  
Let $\zeta = e^{\frac{2 \pi i}{n}}$ be a primitive $n^{th}$ root of unity and suppose $d | n$.  We want to show 
that $|X^{r^d}| = X(\zeta^d)$, where $X = NC(n, k, 0)$, $r: X \rightarrow X$ is rotation, and
$X(q) \in \NN[q]$ is the $q$-hook length formula attached to the partition $\lambda$.

Let $\chi^{\lambda}: \symm_n \rightarrow \CC$ be the irreducible character associated to the partition
$\lambda$.
We have that
\begin{align}
|X^{r^d}| &=  (-1)^{d(n+1)} \chi^{\lambda}(c^d) \\
&= (-1)^{d(n+1)} f^{\lambda}(\zeta^d) \\
&= (-1)^{d(n+1)} (\zeta)^{d \cdot b(\lambda)} X(\zeta^d),
\end{align}
where the first equality uses Proposition~\ref{isomorphism-special-case} and 
Theorem~\ref{long-cycle-action}, the second equality uses Theorem~\ref{springer}, the the third
equality is the definition of the fake degree polynomial $f^{\lambda}(q)$.  We are therefore reduced to
showing that the complex number $(-1)^{d(n+1)} (\zeta)^{d \cdot b(\lambda)} |X^{r^d}|$ is a nonnegative integer.

Unfortunately, there exist vales of $n, k,$ and $d$ such that the product $(-1)^{d(n+1)} (\zeta)^{d \cdot b(\lambda)}$
is not even a real number, let alone a nonnegative integer.  
We argue that both sides of the equation
$|X^{r^d}| = (-1)^{d(n+1)} (\zeta)^{d \cdot b(\lambda)} X(\zeta^d)$ vanish at such values of $d$.
In particular, we have the following claim.

{\bf Claim:}  {\it Either $k(k-1) \equiv 0$ (mod $\frac{n}{d}$) or $X(\zeta^d) = 0$.}

There are two ways to prove this claim.  Arguing combinatorially, the action of $\frac{n}{d}$-fold rotation on
noncrossing partitions of $[n]$ breaks the blocks of any noncrossing partition up into free orbits, with the possible exception of 
one singleton orbit.  Since $k$ is the number of blocks, this says that either $k(k-1) \equiv 0$ (mod $\frac{n}{d}$) or 
$|X^{r^d}| = 0$.

On the other hand, we have a purely algebraic argument using the Murnaghan-Nakayama rule.  
Consider the flag shape $\lambda = (k, k, 1^{n-2k})$.  One shows that if there is a tiling of $\lambda$ with 
ribbons of size $\frac{n}{d}$, we necessarily have that $k(k-1) \equiv 0$ (mod $\frac{n}{d}$).  The 
Murnaghan-Nakayama Rule implies that $k(k-1) \equiv 0$ (mod $\frac{n}{d}$) or
$X(\zeta^d) = 0$.

Our Claim in hand, we argue that $X(\zeta^d)$ is a nonnegative integer.
We have that $b(\lambda) = \sum (i-1)\lambda_i = (n-k) + \frac{(n-2k)(n-2k+1)}{2}$.  We wish to show
that the product $(-1)^{d(n+1)} (\zeta)^{d \cdot b(\lambda)} |X^{r^d}|$ is a nonnegative integer.  We may assume 
that $|X^{r^d}| > 0$.  Be our Claim, we have
\begin{align}
d \cdot b(\lambda) &= d(n-k) + \frac{d(n-2k)(n-2k+1)}{2} \\
& \equiv 2dk(k-1) + \frac{dn(n+1)}{2} \\
 &\equiv \frac{dn(n+1)}{2} \\
& \equiv \begin{cases} \frac{dn}{2} & \text{$n$ even,} \\ 0 & \text{$n$ odd,} \end{cases}
\end{align}
where all equivalences are modulo $n$, the first equality follows from the definition of $b(\lambda)$,
the next congruence is a simple manipulation, the following congruence arises from the 
assumption $|X^{r^d}| > 0$ (so that $k(k-1) \equiv 0$ (mod $\frac{n}{d}$)), and the final congruence
is again a direct computation.  We conclude that $\zeta^{d \cdot b(\lambda)} = 1$ when
$n$ is odd and $\zeta^{d \cdot b(\lambda)} = (-1)^d$ when $n$ is even.  In either case, we have that
$(-1)^{d(n+1)} (\zeta)^{d \cdot b(\lambda)} |X^{r^d}|$ is a nonnegative integer.  This concludes our proof
of Theorem~\ref{no-singletons-csp}.
\end{proof}

Our next goal is to deduce the Narayana CSP in Theorem~\ref{narayana-csp} of Reiner, Stanton, and
White from the CSP of Pechenik.  There are two  equivalent ways to do this: use the action of $c \in \symm_n$
together with the isomorphism type of the module $V(n, k)$ proven in Proposition~\ref{isomorphism-special-case}
and
Theorem~\ref{isomorphism}
or use the close relationship between the action of rotation on noncrossing partitions with $k$ blocks
to the action of rotation on noncrossing partitions without singleton blocks together with
Theorem~\ref{narayana-csp} and a $q$-polynomial identity relating the associated cyclic sieving polynomials.
We present an argument that uses the latter approach, but the reader should be able to fill in the details
of an algebraic approach without too much difficulty.  Our basic tool is the following  $q$-analog
of the Chu-Vandermonde convolution.

\begin{theorem}
\label{chu}
Let $m, n,$ and $k$ be nonnegative integers.  We have that
\begin{equation}
{m+n \brack k}_q = \sum_{j = 0}^k q^{j(m-k+j)}{m \brack k-j}_q {n \brack j}_q.
\end{equation}
\end{theorem}

We will also use the ``fundamental" example of a CSP.

\begin{theorem}
\label{sets-csp} 
(Reiner-Stanton-White \cite[Theorem 1.1]{RSWCSP})
Let $n \geq k > 0$, let $X = {[n] \choose k}$ be the set of $k$-element subsets of $[n]$, and let 
$C = \ZZ_n$ act on $X$ by the long cycle $(1, 2, \dots, n)$.  The triple $(X, C, X(q))$ exhibits the cyclic sieving
phenomenon, where $X(q) = {n \brack k}_q$.
\end{theorem}

We recall the statement of the Narayana CSP.  \\

\noindent {\bf Theorem 1.2.}  
{\it (Reiner-Stanton-White \cite[Theorem 7.2]{RSWCSP}) Let $n \geq k > 0$, let $X$ be the set of noncrossing partitions of $[n]$ with exactly $k$ blocks,
and let $C = \ZZ_n$ act on $X$ by rotation.
The triple $(X, C, X(q))$ exhibits the cyclic sieving phenomenon, where 
\begin{equation*}
X(q) = \Nar_q(n,k) = \frac{1}{[n]_q}{n \brack k}_q {n \brack k-1}_q
\end{equation*}
is the $q$-Narayana number.}

\begin{proof}
Let $\zeta = e^{\frac{2 \pi i}{n}}$ and let $1 \leq d \leq n$ with $d | n$.   Let $r: X \rightarrow X$ be the rotation map. 
A typical set partition in the fixed point set 
$X^{r^d}$ consists of singleton blocks and non-singleton blocks.  The singleton blocks must form
a $\frac{n}{d}$-fold symmetric subset of $[n]$ and the non-singleton blocks must form a 
$\frac{n}{d}$-fold symmetric noncrossing partition which contains no singleton blocks.

{\bf Claim:}  {\em The number
of such set partitions with precisely $s$ singleton blocks equals the root of unity evaluation
\begin{equation}
\label{evaluation}
\left[ q^{(k-s-1)(k-s)} {n \brack s}_q \frac{[n-s]!_q}{([k-s]!_q)^2 [n-2k+s]!_q} \times \frac{[k-s]_q}{[n-k]_q[n-k+1]_q} \right]_{q = \zeta^d}.
\end{equation}}

By Theorems~\ref{no-singletons-csp} and \ref{sets-csp}, the polynomial evaluation in (\ref{evaluation}) 
equals the cardinality of the set in question up to modulus.
To prove our claim it is enough to show that we have
$\zeta^{d(k-s-1)(k-s)} = 1$ when the quantity in (\ref{evaluation}) does not vanish.  As in the proof of 
Theorem~\ref{no-singletons-csp}, this can be argued combinatorially or algebraically.  

From a combinatorial point of view, since 
the $k-s$ non-singleton blocks of any noncrossing set partition fixed by $\frac{n}{d}$-fold rotation consist
of free orbits together with possibly one singleton orbit, we have that 
$d(k-s)(k-s-1) \equiv 0$ (mod $n$) if $X^{r^d}$ is nonempty.

From an algebraic point of view, the  
evaluation of ${n \brack s}_q$ at $q = \zeta^d$ is zero unless $d | s$.
Moreover, the expression
$\frac{[n-s]!_q}{([k-s]!_q)^2 [n-2k+s]!_q} \times \frac{[k-s]_q}{[n-k]_q[n-k+1]_q}$ is the $q$-hook length formula
corresponding to the partition $\lambda = (k-s, k-s, n-2k+s) \vdash n-s$.
As in the proof of Theorem~\ref{no-singletons-csp}, the Murnaghan-Nakayama rule 
implies that the evaluation of this polynomial at $q = \zeta^d$ is zero unless
$(k-s)(k-s-1) \equiv 0$ (mod $\frac{n}{d}$).  Putting these facts together, we get that the polynomial evaluation in
(\ref{evaluation}) is always a nonnegative integer, and therefore equal to the quantity in question.

With our Claim in hand, we can get the fixed point set size $|X^{r^d}|$ by summing the polynomial evaluations
in (\ref{evaluation}) over $s$.
Summing the unevaluated polynomial in
(\ref{evaluation}) over $s$ yields
\begin{align*}
\sum_{s = 0}^k q^{(k-s-1)(k-s)}& {n \brack s}_q \frac{[n-s]!_q}{([k-s]!_q)^2 [n-2k+s]!_q}  \times \frac{[k-s]_q}{[n-k]_q[n-k+1]_q} \\
&= \sum_{s = 0}^k q^{(k-s-1)(k-s)} {n \brack k}_q 
\frac{[k]!_q [n-k]!_q [k-s]_q}{[s]!_q ([k-s]!_q)^2 [n-2k+s]!_q [n-k]_q [n-k+1]_q}  \\
&= \frac{1}{[n-k+1]_q} {n \brack k}_q \sum_{s = 0}^k q^{(k-s-1)(k-s)} {k \brack s}_q {n-k-1 \brack k-s-1}_q \\
&= \frac{1}{[n-k+1]_q} {n \brack k}_q {n-1 \brack k-1}_q \\
&= \frac{1}{[n]_q} {n \brack k}_q {n \brack k-1}_q \\
&= X(q),
\end{align*}
where we used the quantum Chu-Vandermonde identity in Theorem~\ref{chu} for the third equality.
The desired equality $|X^{r^d}| = X(\zeta^d)$ follows.
\end{proof}

We conclude this section by showing (for the sake of self-containment) how to derive the Catalan CSP in Theorem~\ref{catalan-csp} from 
the Narayana CSP
Theorem~\ref{narayana-csp}. 
We recall the 
statement of Theorem~\ref{catalan-csp}.  

\noindent {\bf Theorem 1.1.}  
{\it (Reiner-Stanton-White \cite{RSWCSP}) Let $n \geq 0$, let $X$ be the set of noncrossing partitions of $[n]$,
and let $C = \ZZ_n$ act on $X$ by rotation.
The triple $(X, C, X(q))$ exhibits the cyclic sieving phenomenon, where 
\begin{equation*}
X(q) = \Cat_q(n) = \frac{1}{[n+1]_q}{2n \brack n}_q
\end{equation*}
is the $q$-Catalan number.}

\begin{proof}
Let $r: X \rightarrow X$ be the rotation operator and let $1 \leq d \leq n$ with $d | n$.  Let $\zeta \in \CC$ be a primitive $n^{th}$ root of unity.

{\bf Claim:}  {\em The number of partitions in $X$ with exactly $k$ blocks fixed by $r^d$ is the polynomial evaluation
\begin{equation}
\label{catalan-csp-fudge}
\left[ \frac{q^{k(k-1)}}{[n]_q} {n \brack k}_q {n \brack k-1}_q \right]_{q = \zeta^d}.
\end{equation}}

In fact, this claim is the original form of the Narayana CSP as proven by Reiner, Stanton, and White \cite{RSWCSP}.
Given our version of Theorem~\ref{narayana-csp}, this claim is equivalent to showing that $k(k-1) \equiv 0$ (mod $\frac{n}{d}$) whenever
the above expression is nonzero.
This can be seen combinatorially from the fact that the action of $\frac{n}{d}$-fold rotation breaks the blocks of any noncrossing partition of $[n]$ into
free orbits, with the possible exception of one singleton orbit.  
Alternatively, Theorem~\ref{sets-csp} implies that both of the $q$-binomials
${n \brack k-1}_q$ and ${n \brack k}_q$ evaluate to zero at $q = \zeta^d$ unless $k(k-1) \equiv 0$ (mod $\frac{n}{d}$).  Since
$[n]_q$ has a zero of order at most one at $q = \zeta^d$, we get that the overall expression
(\ref{catalan-csp-fudge}) vanishes unless $k(k-1) \equiv 0$ (mod $\frac{n}{d}$).

With our Claim in hand, the theorem follows from summing the unevaluated polynomial 
in (\ref{catalan-csp-fudge}) over $k$.  We have that
\begin{align*}
\sum_{k = 0}^n \frac{q^{k(k-1)}}{[n]_q} {n \brack k}_q {n \brack k-1}_q &= 
\frac{1}{[n+1]_q} \sum_{k = 0}^n q^{k(k-1)} {n+1 \brack k}_q {n-1 \brack n-k}_q \\
&= \frac{1}{[n+1]_q} {2n \brack n}_q,
\end{align*}
where we used Theorem~\ref{chu} in the second equality.
\end{proof}

\section{Closing Remarks}
\label{Closing}

In this paper we used skein relations to define an action of $\symm_n$ on the $\QQ$-vector space $V(n)$ spanned by noncrossing partitions of $[n]$.
Moreover, we defined a linear projection $p: P(n) \twoheadrightarrow V(n)$, where
$P(n)$ was the $\QQ$-vector space spanned by all set partitions of $[n]$. 
This projection became an epimorphism of $\symm_n$-modules when $P(n)$ was endowed with a slightly sign-twisted version of its usual
$\symm_n$-module structure.  We propose some possible directions in which to generalize our construction.

For a parameter $q$, the {\it Hecke algebra} $H_n(q)$ attached to $\symm_n$ is the $\CC$-algebra with generators
 $\{T_1, T_2, \dots, T_{n-1}\}$ and relations
 \begin{equation}
 \begin{cases}
 T_i^2 = (1-q)T_i + q, \\
  T_i T_j = T_j T_i, & \text{for $|i - j| > 1$,} \\
 T_i T_{i+1} T_i = T_{i+1} T_i T_{i+1}.
 \end{cases}
 \end{equation}
 At $q = 1$, this specializes to the symmetric  group algebra $\CC \symm_n$.  At $q = 0$,
 one obtains the {\it $0$-Hecke algebra} which has seen a great deal of recent attention in the theory of quasisymmetric functions.
 For generic values of $q$, it is known that the algebras $\CC \symm_n$ and $H_n(q)$ are isomorphic, but writing down such an isomorphism 
 explicitly is a difficult problem.
 
 \begin{problem}
 \label{q-analog}
 Construct a quantum analog of the modules $P(n)$ and $V(n)$ by defining $H_n(q)$ modules $P(n)_q$ and $V(n)_q$ which specialize to 
 $P(n)$ and $V(n)$ at $q = 1$.  Find an epimorphism $p: P(n)_q \twoheadrightarrow V(n)_q$ of $H_n(q)$-modules relating these structures.
 \end{problem}
 
 Some motivation for Problem~\ref{q-analog} comes from Kazhdan-Lusztig theory.  The KL basis of $\CC \symm_n$ is naturally the $q = 1$ specialization
 of the KL basis of $H_n(q)$.  Moreover, the KL cellular basis for $\CC \symm_n$-irreducibles is the $q = 1$ specialization of the 
 KL cellular basis of $H_n(q)$-irreducibles.   
 Given the important role played by KL bases in cyclic sieving, it would be interesting to develop a $q$-generalization of our new skein structures.
 
 A second possible direction of generalization comes from Coxeter-Catalan theory.  Let $W$ be an irreducible real reflection group 
 with reflection representation $V$ 
 and let 
 $c \in W$ be a choice of Coxeter element.
 Let $\mathcal{L}$ be the lattice of flats of the reflection arrangement for $W$ in $V$.  The flats in $\mathcal{L}$ give a $W$-analog of set partitions.
 Moreover, those flats $X \in \mathcal{L}$ which have the form $X = V^w := \{v \in V \,:\, w.v = v\}$
 for some $w$ belonging to the absolute order interval $[e, c]_T$ are the $W$-analog
 of noncrossing set partitions.  
 The group $W$ acts on the set $\mathcal{L}$ by translation and
 the set $NC(W)$ of noncrossing flats is stabilized by the cyclic subgroup $\langle c \rangle$ of $W$.
 Let $P(W)$ denote the $\QQ$-vector space spanned by $\mathcal{L}$ and let $V(W)$ denote the 
 $\QQ$-vector subspace spanned by $NC(W)$.

\begin{problem}
\label{W-analog}
Endow $P(W)$ and $V(W)$ with $W$-module structures so that there is an epimorphism
$p: P(W) \twoheadrightarrow V(W)$ which is also a projection.
\end{problem}
 
While we have no conjectural approach for solving Problem~\ref{W-analog} in a case free way, 
when $W$ is the hyperoctohedral group $W(B_n)$ we can identify $V(W)$ and $P(W)$ with the subspaces of $V(2n)$ and $P(2n)$ 
spanned by centrally symmetric basis elements with at most one centrally symmetric block.  This combinatorial model
may allow for ``centrally symmetric" analogs of our skein relations in defining an action of $W(B_n)$ on $V(W)$.

The third direction of generalization we propose is topological.  One can think of noncrossing partitions of $[n]$ 
as embeddings of intervals and loops into the disc $D$.
It is natural to ask what happens when one replaces the  $D$ with a more interesting surface $S$.

Let $S$ be a surface equipped with a finite collection of marked boundary points $M \subset \partial S$.  Define an {\it (S, M)-noncrossing partition} $\pi$ to be an embedding
 $\pi: S^1 \amalg \cdots \amalg S^1 \amalg I \cdots \amalg I \hookrightarrow S$ of a disjoint union of some number of circles and intervals into $S$ such that
\begin{itemize}
\item for any interval $I = [0, 1]$,  we have $\pi(0), \pi(1) \in M$ and $\pi(x) \in \mathrm{Interior}(S)$ for $0 < x < 1$, and
\item for any circle $S^1$, we have that $\pi(S^1)$ contains at least three elements of $M$ and $\pi(S^1) \cap \partial S \subseteq M$.
\end{itemize}
We consider $(S,M)$-noncrossing partitions $\pi$ up to ambient isotopy.
We also consider two $(S, M)$-noncrossing partitions equivalent if they can be obtained from one another by rotating any
boundary disc of $S$ by some multiple of $2 \pi$ radians.  
A noncrossing partition 
$\pi: S^1 \amalg I \amalg I \hookrightarrow S$
on the pair of pants $S$ is shown below.  

\begin{center}
\includegraphics[scale = 0.3]{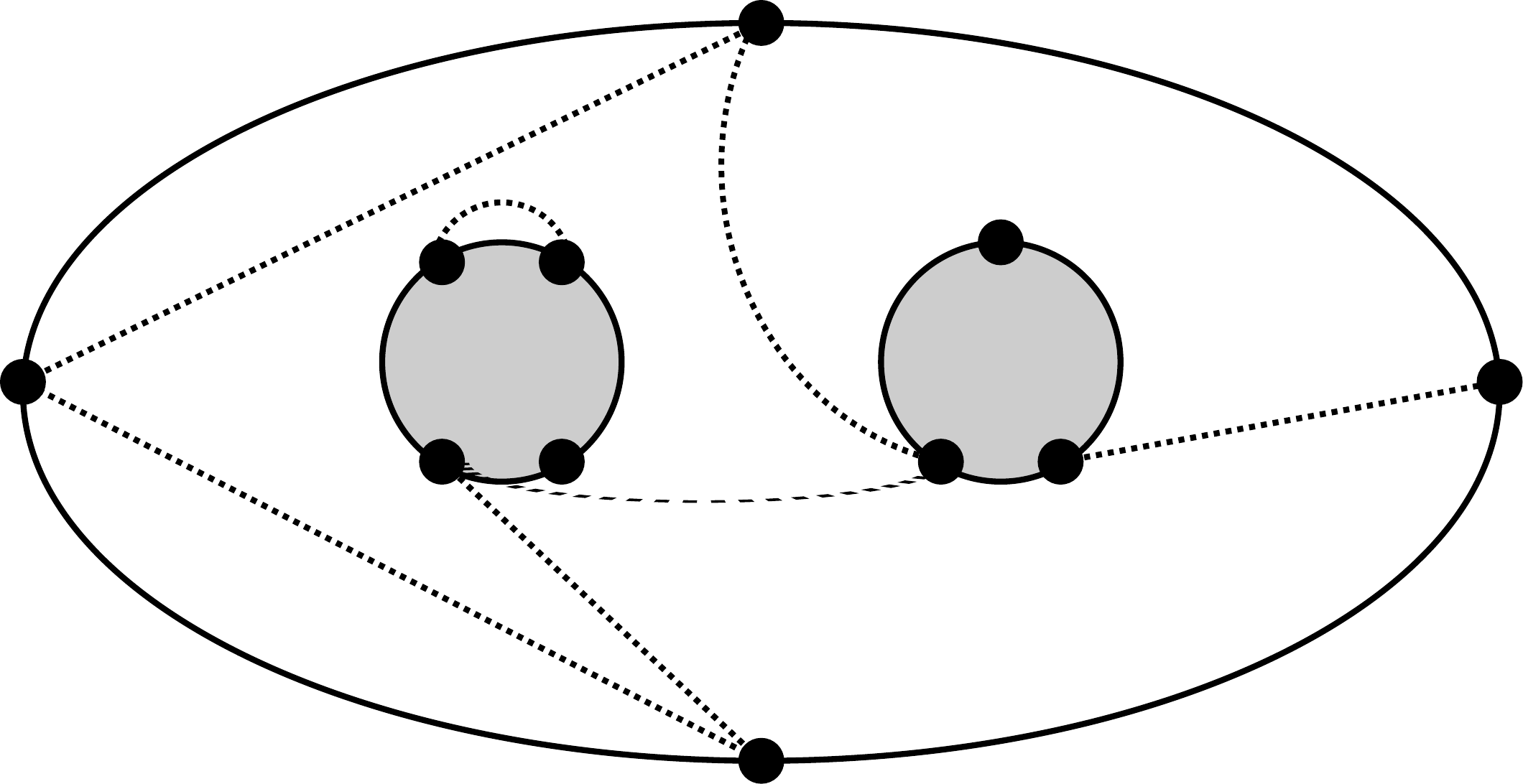}
\end{center}

Let $V(S, M)$ denote the vector space of formal $\QQ$-linear combinations of $(S, M)$-noncrossing partitions.  
If $C$ is a boundary component of $S$ with at least three marked points and $p, p'$ are clockwise consecutive marked points on $C$,
we can define a linear operator $\tau_p: V(S, M) \rightarrow V(S, M)$ by mimicking the definition of the $\tau_i$ operators in
Section~\ref{Symmetric}.  
Specifically, we swap $p$ and $p'$ in a given $(S, M)$-noncrossing partition $\pi$ and use the skein map $\sigma$ to  resolve any crossings
One ignores any terms arising from application of the skein map $\sigma$ resulting in interior loops.
The fact that $C$ has at least three marked points removes potential topological ambiguity in the direction used to swap $m$ and $m'$.

\begin{conjecture}
\label{topology}
Let $(S, M)$ be as above and assume that every boundary component of $S$ contains at least three marked points.  The operators 
$\{ \tau_p \,:\, p \in M \}$  define an action of
$\symm_{m_1} \times \cdots \times \symm_{m_k}$ on $V(S, M)$, where 
$\partial S = C_1 \amalg \cdots \amalg C_k$ and $|C_i \cap M| = m_i$.
\end{conjecture}

For example, the marked pair of pants shown above should yield a representation of 
$\symm_4 \times \symm_4 \times \symm_3$, with one symmetric group factor for each boundary circle.
If Conjecture~\ref{topology} is true, it would be interesting to determine the isomorphism type of the module $V(S, M)$.

The original motivation for this project was Pechenik's CSP in Theorem~\ref{no-singletons-csp}.  Although we phrased our cyclic sieving results 
in terms of rotation acting on noncrossing partitions, Pechenik's paper is more focused on an equivalent action of K-promotion on 
rectangular increasing tableaux.  The results of Pechenik's paper and this paper can be viewed as providing brute force and algebraic proofs, respectively,
for a CSP involving K-promotion acting on two-row rectangular increasing tableaux.  
The case of general rectangles could provide the fourth generalization of our module $V(n)$; let us define these terms.

Given a partition $\lambda \vdash n$ and a positive integer $k$, a {\em $k$-increasing tableau of shape $\lambda$}
is a {\em surjective} filling $T: \lambda \rightarrow \{1, 2, \dots, k\}$ of the boxes of the Ferrers diagram of $\lambda$ with the numbers
$1, 2, \dots, k$ in such a way that entries strictly increase down columns and across rows.  We let 
$\mathrm{Inc}_k(\lambda)$ denote the set of $k$-increasing tableaux of shape $\lambda$.  For $\lambda \vdash n$ we have that 
$\mathrm{Inc}_k(\lambda)$ is the set of standard tableaux of shape $\lambda$ when $k = n$.  We also have that 
$\mathrm{Inc}_k(\lambda)$ is empty when $k > n$.
When $k = 5$ and $\lambda = (3,3)$, the five tableaux in $\mathrm{Inc}_5(3,3)$ are listed below.

\begin{center}
\begin{Young}
1 & 2 & 4 &,   &1  & 2 & 3 &,  & 1 & 2 & 4   &,  & 1 & 2 & 3 &,  & 1 & 3 & 4\cr
2 & 3 & 5 &, , &2  & 4 & 5 &,,  & 3 & 4 & 5  &,, & 3 & 4 & 5 &,, & 2 & 4 & 5
\end{Young}
\end{center}

Analogous to the key role played by standard tableaux in the structure of the cohomology ring of the Grassmannian, the more general
increasing tableaux play an important role in the K-theory of the Grassmannian.   {\it Jeu-de-taquin promotion} is an operator $j$ defined 
using jeu-de-taquin slides on 
standard  tableaux of a fixed shape $\lambda$. 
For a partition $\lambda$ and a positive integer $k$, Thomas and Yong \cite{TYJDT} defined 
a generalized version of jeu-de-taquin on increasing tableaux which naturally leads to a {\it K-promotion}
operator $j_K: \mathrm{Inc}_k(\lambda) \rightarrow \mathrm{Inc}_k(\lambda)$.

The operator $j_K: \mathrm{Inc}_k(\lambda) \rightarrow \mathrm{Inc}_k(\lambda)$ is defined as follows.  
Let $T \in \mathrm{Inc}_k(\lambda)$ be an increasing tableau.  
The increasing tableau $j_K(T)$ is defined by replacing every instance of $k$ in $T$ with a hole,
using jeu-de-taquin slides to move the resulting hole to the northwest corner while preserving increase across rows and down columns,
increasing every entry by $1$, and filling the hole with $1$.  To preserve the increasing condition, holes and/or entries may have to be 
cloned and/or merged during the sliding procedure.
An example of $K$-promotion acting on an increasing tableau $T \in \mathrm{Inc}_7(4, 4, 3)$ is shown below.

\begin{center}
\begin{small}
\begin{Young}
, \cr
, $T =$ \cr
, \cr
\end{Young}
\hspace{0.05in}
\begin{Young}
1 &2  &3 & 4 \cr 
2 &3 & 6 &7 \cr
4& 5 & 7\cr
\end{Young}
\begin{Young}
, \cr
, $\mapsto$ \cr
, \cr
\end{Young}
\begin{Young}
1 &2  &3 & 4 \cr 
2 &3 & 6 &$\bullet$ \cr
4& 5 & $\bullet$ \cr
\end{Young}
\begin{Young}
, \cr
, $\mapsto$ \cr
, \cr
\end{Young}
\begin{Young}
1 &2  &3 & 4 \cr 
2 &3 & $\bullet$ &6 \cr
4& 5 & 6\cr
\end{Young}
\begin{Young}
, \cr
, $\mapsto$ \cr
, \cr
\end{Young}
\begin{Young}
1 &2  &$\bullet$ & 4 \cr 
2 & $\bullet$ & 3 &6 \cr
4& 5 & 6\cr
\end{Young} 

\vspace{0.05in}

\hspace{0.6in}
\begin{Young}
, \cr
, $\mapsto$ \cr
, \cr
\end{Young}
\begin{Young}
1 & $\bullet$  & 2 & 4 \cr 
$\bullet$ & 2 & 3 &6 \cr
4& 5 & 6\cr
\end{Young}
\begin{Young}
, \cr
, $\mapsto$ \cr
, \cr
\end{Young}
\begin{Young}
$\bullet$ & 1  & 2 & 4 \cr 
1 & 2 & 3 &6 \cr
4& 5 & 6\cr
\end{Young}
\begin{Young}
, \cr
, $\mapsto$ \cr
, \cr
\end{Young}
\begin{Young}
$\bullet$ & 2  & 3 & 5 \cr 
2 & 3 & 4 & 7 \cr
5 & 6 & 7\cr
\end{Young}
\begin{Young}
, \cr
, $\mapsto$ \cr
, \cr
\end{Young}
\begin{Young}
1 & 2  & 3 & 5 \cr 
2 & 3 & 4 & 7 \cr
5 & 6 & 7\cr
\end{Young}
\begin{Young}
, \cr
, $= j_K(T)$ \cr
, \cr
\end{Young}
\end{small}
\end{center}

When $k = n$, no cloning or merging is required and the $K$-promotion operator $j_K$ specializes to the usual promotion operator
$j$ on standard tableaux.  
While the action of $j$ on standard tableaux can be somewhat chaotic for general partitions $\lambda \vdash n$, this action 
is much better understood when $\lambda$ is a rectangular partition \cite{FK, RhoadesJDT}.
It is natural to generalize by asking for the cycle structure of the action of $j_K$ on 
$\mathrm{Inc}_k(\lambda)$ for a rectangular partition $\lambda \vdash n$.

Suppose $\lambda = (b, b) \vdash n$ is a two-row rectangle.  Pechenik \cite{Pechenik} described a bijection from the set
$\mathrm{Inc}_{b + c}(b, b)$ of increasing tableaux to  the set $NC(b+c, c, 0)$ of noncrossing partitions without singleton blocks which 
sends K-promotion to rotation.
Theorem~\ref{no-singletons-csp} therefore gives a CSP for the action of K-promotion on the set $\mathrm{Inc}_k(\lambda)$
for any value of $k$ when $\lambda$ is a two-row rectangle.  
On the other hand, Kazhdan-Lusztig theory can be used \cite{RhoadesJDT} to give a CSP for the action of ordinary promotion
on the set of standard tableaux of shape $\lambda$, where $\lambda$ is an arbitrary rectangle.  The following problem asks for 
a common generalization.

\begin{problem}
\label{K-promotion-problem}
Let $\lambda \vdash n$ be a rectangular partition and let $k$ be a positive integer.
Give a CSP for the action of K-promotion on $k$-increasing tableaux of shape $\lambda$.
Prove this CSP using representation theoretic means.
\end{problem}

One complication involved with the action of K-promotion is that 
$(j_K)^k(T) = T$ does {\it not} hold for all rectangular $k$-increasing  tableaux $T$.
On the other hand, Pechenik \cite{Pechenik} described a subclass of $k$-increasing rectangular tableaux $T$ in terms of 
a K-analog of evacuation for which
$(j_K)^k(T) = T$ always holds.

In this paper, we gave algebraic proofs for CSPs related to the action of rotation on
noncrossing partitions of $[n]$.
It is natural to ask for a Fuss-analog of these constructions involving CSPs for
the action of rotation on
{\it $k$-divisible} noncrossing partitions of $[kn]$.
A CSP for the action of rotation on all $k$-divisible
noncrossing partitions of $[kn]$ was proven by brute force
by Krattenthaler and M\"uller \cite{KM} and later received an algebraic proof in \cite{RhoadesParking}.
A Fuss-analog of the skein basis ought to give an algebraic proof of CSPs involving actions
on $k$-divisible noncrossing partitions with a fixed number of total blocks (and perhaps also
a fixed number of $k$-element blocks).

\section{Acknowledgements}
\label{Acknowledgements}

The author thanks Drew Armstrong, Michelle Bodnar,
 Rosa Orellana, Oliver Pechenik, Vic Reiner,  Heather Russell, and Hans Wenzl for many helpful discussions.
The author was partially supported by NSF Grant DMS-1068861.

\end{document}